\def\version{{\tiny Version September 14, 2018 (typeset: \today)}}
\def\version{}

\documentclass[reqno,twoside,11pt]{amsart}
\usepackage{cite}
\usepackage{amsmath}
\usepackage{amsfonts}
\usepackage{amssymb}
\usepackage{epsfig}
\usepackage{verbatim}
\usepackage{color}

\IfFileExists{epsf.def}{\input epsf.def}{\usepackage{epsf}}

\IfFileExists{mathptmx.sty}{\usepackage{mathptmx}}{\usepackage{mathpazo}}
\usepackage{mathrsfs}

\DeclareFontFamily{OT1}{eusb}{} \DeclareFontShape{OT1}{eusb}{m}{n} {<5> <6> <7> <8> <9> <10> <11> <12> <14.4> eusb10}{}
\DeclareMathAlphabet{\eusb}{OT1}{eusb}{m}{n}

\DeclareFontFamily{OT1}{eusm}{} \DeclareFontShape{OT1}{eusm}{m}{n} {<5> <6> <7> <8> <9> <10> <11> <12> <14.4> eusm10}{}
\DeclareMathAlphabet{\eusm}{OT1}{eusm}{m}{n}

\DeclareFontFamily{OT1}{eufm}{} \DeclareFontShape{OT1}{eufm}{m}{n} {<5> <6> <7> <8> <9> <10> <11> <12> <14.4> eufm10}{}
\DeclareMathAlphabet{\mathfrak}{OT1}{eufm}{m}{n}

\DeclareFontFamily{OT1}{fraktura}{}
\DeclareFontShape{OT1}{fraktura}{m}{n} {<5> <6> <7> <8> <9> <10> <11> <12> <13> <14.4> [1.1] eufm10}{}
\DeclareMathAlphabet{\fraktura}{OT1}{fraktura}{m}{n}

\DeclareFontFamily{OT1}{cmfi}{} \DeclareFontShape{OT1}{cmfi}{m}{n} {<5> <6> <7> <8> <9> <10> <11> <12> <13> <14.4> [0.9] cmfi10}{}
\DeclareMathAlphabet{\cmfi}{OT1}{cmfi}{b}{n}

\DeclareFontFamily{OT1}{cmss}{} \DeclareFontShape{OT1}{cmss}{m}{n} {<5> <6> <7> <8> <9> <10> <11> <12> <13> <14.4> cmss10}{}
\DeclareMathAlphabet{\cmss}{OT1}{cmss}{m}{n}

\newtheoremstyle{thm}{1.5ex}{1.5ex}{\itshape\rmfamily}{} {\bfseries\rmfamily}{}{2ex}{}

\newtheoremstyle{def}{1.5ex}{1.5ex}{\rmfamily\sl}{} {\bfseries\rmfamily}{}{2ex}{}

\newtheoremstyle{rem}{1.3ex}{1.3ex}{\rmfamily}{} {\bfseries\rmfamily}{}{2ex}{}

\newtheoremstyle{ass}{1.5ex}{1.5ex}{\rmfamily\sl}{} {\bfseries\rmfamily}{}{2ex}{}

\newenvironment{proofsect}[1] {\vskip0.1cm\noindent{\rmfamily\itshape#1.}}{\qed\vspace{0.15cm}}

\theoremstyle{thm}
\newtheorem{theorem}{Theorem}[section]
\newtheorem{lemma}[theorem]{Lemma}
\newtheorem{proposition}[theorem]{Proposition}

\newtheorem*{Main Theorem}{Main Theorem.}
\newtheorem{corollary}[theorem]{Corollary}
\newtheorem{conjecture}[theorem]{Conjecture}

\newtheoremstyle{named}{}{}{\itshape}{}{\bfseries}{}{.5em}{\thmnote{#3}}
\theoremstyle{named}

\theoremstyle{def}

\theoremstyle{rem}
\newtheorem{remark}[theorem]{{Remark}}

\numberwithin{equation}{section}
\renewcommand{\theequation}{\arabic{section}.\arabic{equation}}

\renewcommand{\section}{\secdef\sct\sect}
\newcommand{\sct}[2][default]{\refstepcounter{section}
\addcontentsline{toc}{section}
{{\tocsection {}{\thesection}{\!\!\!\!#1\dotfill}}{}}
\vspace{0.7cm}
\centerline{ 
\scshape\arabic{section}.\ #1} \nopagebreak \vspace{0.2cm}}
\newcommand{\sect}[1]{
\vspace{0.4cm} \centerline{\large\scshape\rmfamily #1}
\vspace{0.2cm}}

\renewcommand{\subsection}{\secdef\subsct\sbsect}
\newcommand{\subsct}[2][default]{\refstepcounter{subsection}
\addcontentsline{toc}{subsection}
{{\tocsection{\!\!}{\hspace{1.2em}\thesubsection}{\!\!\!\!#1\dotfill}}{}}
\nopagebreak\vspace{0.45\baselineskip} {\flushleft\bf
\arabic{section}.\arabic{subsection}~\bf #1.~}
\\*[3mm]\noindent
\nopagebreak}
\newcommand{\sbsect}[1]{
\vspace{0.1cm}\noindent
\textbf{#1.~}\vspace{0.1cm}}

\renewcommand{\subsubsection}{%
\secdef \subsubsect\sbsbsect}
\newcommand{\subsubsect}[2][default]{%
\refstepcounter{subsubsection} 
\addcontentsline{toc}{subsubsection}{{\tocsection{\!\!}
{\hspace{3.05em}\thesubsubsection}{\!\!\!\!#1\dotfill}}{}}
\nopagebreak
\vspace{0.15\baselineskip} \nopagebreak {\flushleft\rmfamily
\itshape\arabic{section}.\arabic{subsection}.\arabic{subsubsection}
\ \rmfamily #1\/.}\ }
\newcommand{\sbsbsect}[1]{\vspace{0.1cm}\noindent
\rmfamily \itshape
\arabic{section}.\arabic{subsection}.\arabic{subsubsection} \
\sffamily #1\/.\ }

\newcounter{obrazek}
\renewcommand{\caption}[1]{%
\vglue0.5cm
\refstepcounter{obrazek}
\begin{minipage}{0.9\textwidth}\small {\sc Figure~\theobrazek. }#1\end{minipage}}

\newcommand{\dist}{\operatorname{dist}}
\newcommand{\supp}{\operatorname{supp}}

\newcommand{\textd}{\text{\rm d}\mkern0.5mu}

\newcommand{\texte}{\text{\rm  e}\mkern0.7mu}

\newcommand{\Var}{\text{\rm Var}}
\newcommand{\Cov}{\text{\rm \text{\rm Cov}}}
\newcommand{\1}{{1\mkern-4.5mu\textrm{l}}}
\renewcommand{\1}{\text{\sf 1}}

\newcommand{\FF}{\mathcal F}

\newcommand{\HH}{\mathcal H}

\newcommand{\NN}{\mathcal N}

\newcommand{\C}{\mathbb C}
\newcommand{\D}{\mathbb D}

\newcommand{\N}{\mathbb N}

\newcommand{\BbbP}{\mathbb P}
\newcommand{\Q}{\mathbb Q}
\newcommand{\R}{\mathbb R}

\newcommand{\T}{\mathbb T}

\newcommand{\Z}{\mathbb Z}

\newcommand{\twoeqref}[2]{(\ref{#1}--\ref{#2})}

\newcommand{\cc}{{\text{\rm c}}}

\def\myffrac#1#2 in #3{\raise 2.6pt\hbox{$#3 #1$}\mkern-1.5mu\raise 0.8pt\hbox{$#3/$}\mkern-1.1mu\lower 1.5pt\hbox{$#3 #2$}}
\newcommand{\ffrac}[2]{\mathchoice%
	{\myffrac{#1}{#2} in \scriptstyle}
	{\myffrac{#1}{#2} in \scriptstyle}
	{\myffrac{#1}{#2} in \scriptscriptstyle}
	{\myffrac{#1}{#2} in \scriptscriptstyle}
}

\newcommand{\wh}{\widehat}
\newcommand{\wt}{\widetilde}

\newcommand{\laweq}{\,\overset{\text{\rm law}}=\,}

\newcommand{\lf}{\lfloor}
\newcommand{\rf}{\rfloor}

\newcommand{\leb}{{\rm Leb}}
\newcommand{\dd}{\textd}

\newcommand{\Lawarrow}{{\,\overset{\text{\rm law}}\longrightarrow\,}}
\newcommand{\Lawlongarrow}{{\,\,\overset{\text{\rm law}}\longrightarrow\,\,}}

\newcommand{\rad}{{\text{\rm rad}}}
\newcommand{\fraka}{\fraktura a}

\begin{document}

\title[]
{\large Conformal symmetries in the extremal process\\of two-dimensional discrete Gaussian Free Field}

\author[\hfill  \version \hfill Biskup and Louidor]
{Marek~Biskup$^{1,2}$ and Oren~Louidor$^3$}
\thanks{\hglue-4.5mm\fontsize{9.6}{9.6}\selectfont\copyright\,\textrm{2019}\ \ \textrm{M.~Biskup, O.~Louidor.
Reproduction, by any means, of the entire
article for non-commercial purposes is permitted without charge.\vspace{2mm}}}
\maketitle

\vspace{-5mm}
\centerline{\textit{$^1$
Department of Mathematics, UCLA, Los Angeles, California, USA}}
\centerline{\textit{$^2$
Center for Theoretical Study, Charles University, Prague, Czech Republic}}
\centerline{\textit{$^3$
Faculty of Industrial Engineering and Management, Technion, Haifa, Israel}}

\vskip0.3cm
\begin{quote}
\footnotesize \textbf{Abstract:}
We study the extremal process associated with the Discrete Gaussian Free Field on the square lattice and elucidate how  the conformal symmetries manifest themselves in the scaling limit. Specifically, we prove that the joint process of spatial positions~($x$) and centered values~($h$) of the extreme local maxima in lattice versions of a bounded domain~$D\subset\C$ converges, as the lattice spacing tends to zero, to a Poisson point process with intensity measure $Z^D(\textd x)\otimes\texte^{-\alpha h}\textd h$, where~$\alpha$ is a constant  and~$Z^D$ is a random a.s.-finite measure on~$D$. The random measures $\{Z^D\}$ are naturally interrelated; restrictions to subdomains are governed by a Gibbs-Markov property and images under analytic bijections~$f$ by the transformation rule $(Z^{f(D)}\circ f)(\textd x)\laweq |f'(x)|^4\, Z^D(\textd x)$. Conditions are given that determine the laws of these measures uniquely. These identify~$Z^D$ with the critical Liouville Quantum Gravity associated with the Continuum Gaussian Free Field. 
\end{quote}

\tableofcontents
\vglue-1.0cm

\section{Introduction}
\label{sec:1}\noindent
The Discrete Gaussian Free Field (DGFF) is a Gaussian process $\{h_x^V\colon x\in V\}$, indexed by the vertices in a finite subset~$V$ of an infinite graph with covariance given by the Green function, i.e., the expected number $G_V(x,y)$ of visits to~$y$ of the simple symmetric random walk started from~$x$ and killed upon exiting~$V$. We usually take the field to be zero mean, $E h_x^V=0$, and regard it to be identically zero outside~$V$. In this paper, we describe the limit statistics of the \emph{extremal values} of the DGFF in proper subsets of~$\Z^2$ that approximate, in a sense to be defined precisely below, a given bounded continuum domain $D\subset\C$. Our particular interest is in the precise way  these  statistics depends on~$D$. 

It is well known that, in the said limit, the DGFF tends to a continuum object called the Continuum Gaussian Free Field (CGFF). A distinguished feature of the two-dimensional CGFF is that it is invariant under all conformal maps (i.e., analytic bijections) of the underlying domain. Unfortunately, the CGFF is too rough to exist as a proper function and it has to be interpreted as a Gaussian random linear functional over the Hilbert space of functions $\cmss H^1_0(D)$ see, e.g., the review by Sheffield~\cite{Sheffield-review}. In particular, there is no meaning to its maximum and/or extremal values. One of our goals here is thus to elucidate how the conformal symmetries of the CGFF manifest themselves at the level of limit laws for extremal values of the DGFF.

A first step towards the above goal has been taken by the authors in~\cite{Biskup-Louidor} by proving that the process of extreme local maxima,
\begin{equation}
\eta_{N,r}:=\sum_{x\in V_N}1_{\{h^{V_N}_x=\max_{y\colon|y-x|<r}h^{V_N}_y\}}\delta_{x/N}\otimes\delta_{h^{V_N}_x-m_N}\,,
\end{equation}
on lattice versions~$V_N:=(0,N)^2\cap\Z^2$ of the unit square~$D:=(0,1)^2$, admits, for any sequence~$r_N$ (defining the meaning of ``local'') such that $r_N\to\infty$ and $N/r_N\to\infty$, the limit
\begin{equation}
\label{E:1.4a}
\eta_{N,r_N}\,\,\underset{N\to\infty}\Lawarrow\,\,\text{PPP}\bigl(Z(\textd x)\otimes\texte^{-\alpha h}\textd h\bigr).
\end{equation}
Here ``PPP'' designates a Poisson point process, $\alpha:=2/\sqrt{g}$  for~$g$ a constant such that $G^{V_N}(x,x)=g\log N+O(1)$ at points~$x$ ``deep'' inside~$V_N$ (which, in our normalization, gives~$g:=2/\pi$)  and $Z(\textd x)$ is a \emph{random} intensity measure on~$[0,1]^2$. The centering sequence
\begin{equation}
\label{E:1.3}
	m_N := 2\sqrt{g} \log N - \frac34 \sqrt{g} \log \log N,
\end{equation}
captures the growth rate of the absolute maximum (Bramson and Zeitouni~\cite{BZ}). The law of the centered maximum $\max_{x\in V_N}h^{V_N}_x-m_N$ is known to converge (Bramson, Ding and Zeitouni~\cite{BDingZ}). For the limit law,~\eqref{E:1.4a} yields the representation
\begin{equation}
\label{E:1.5a}
P\bigl(\,\max_{x\in V_N}h_x^{V_N}-m_N\le t\bigr)\,\underset{N\to\infty}\longrightarrow\,E\bigl(\texte^{-\alpha^{-1}\texte^{-\alpha t} Z([0,1]^2)}\bigr).
\end{equation}
Notice that the right-hand side is the Laplace transform of the total mass of the~$Z$-measure.


There are two directions in which the work~\cite{Biskup-Louidor} called for further extension. The first one concerns the behavior of the full extremal process; indeed, the results of~\cite{Biskup-Louidor} addressed only local maxima and ignored, for good reasons, the points (still extremal) lying nearby thereof. This aspect has now been fully resolved in Biskup and Louidor~\cite{BL-new}. The second direction concerns the global correlations encoded into the random measure~$Z(\textd x)$. Two natural questions~arise:
\begin{enumerate}
\item[(1)] Does \eqref{E:1.4a} generalize to domains other than squares? And, if so, how are the corresponding $Z$-measures related?
\item[(2)] Can the law of the $Z$-measure be independently characterized? In particular, what properties determine their laws uniquely?
\end{enumerate}
The aim of the present work is to completely resolve both of these questions.

Our strategy is as follows:  We first  demonstrate that \eqref{E:1.4a} extends to a large class~$\mathfrak D$ of bounded open sets~$D\subset\C$.  Then we  show that the resulting random measures~$\{Z^D\colon D\in\mathfrak D\}$ are quite interdependent. In particular, they transform canonically under the restriction to a subdomain (the Gibbs-Markov property in Theorem~\ref{thm:2}) and conformal maps between domains (Theorem~\ref{thm:3}). An essential ingredient for the latter is invariance of the law of~$Z^D$ under \emph{rotations} of~$D$; indeed, the Gibbs-Markov property localizes the conformal map to a rotation and a dilation and~$Z^D$ transforms canonically under dilations thanks to the very existence of the limit~\eqref{E:1.4a}. 

A by-product of the proof of conformal invariance is a list of properties that identifies the laws of the measures~$\{Z^D\colon D\in\mathfrak D\}$ uniquely. By way of somewhat tedious estimates we then verify that all properties on this list are fulfilled by the ``white-noise'' version of the \emph{Liouville Quantum Gravity}  (LQG)  measure introduced in Duplantier, Rhodes, Sheffield and Vargas~\cite{DRSV1,DRSV2}, and thus identify our~$Z^D$'s with this object up to an overall multiplicative constant.

Following up on a referee's suggestion, we note that the term LQG measure used in probability is technically rather different from that used in physics. Indeed, there the LQG is a theory of a random metric derived from a CGFF and where the underlying probability law is, in the absence of other interacting fields, that of the CGFF tilted by the total mass of the associated probabilistic LQG measure; see Rhodes and Vargas~\cite[Section~5.2]{Rhodes-Vargas} for some details. Moreover, the CGFF has to be taken with Neumann boundary conditions instead of Dirichlet considered here.



\section{Main results}
\nopagebreak
\vglue-0.4cm
\subsection{Limit in general domains}
We now move to the statement of our main results. As noted above, our first goal is to generalize \eqref{E:1.4a} to a representative family of domains in the complex plane. Let~$\mathfrak D$ denote the class of all non-empty, bounded, open sets~$D\subset\C$ with a finite number of connected components and boundary $\partial D$ that has a finite number of connected components each of which is of positive Euclidean diameter. Note that~$\mathfrak D$ is closed under shifts, rotations, dilations and finite unions and that it includes all bounded simply connected domains in~$\C$. 

Let $\textd_\infty(x,y)$ denote $\ell^\infty$-distance on~$\R^2$. Given~$D\in\mathfrak D$ we then consider a family $\{D_N\colon N\ge1\}$ of its discrete approximations for which we assume
\begin{equation}
\label{E:1.1}
	 D_N \subseteq \bigl\{x \in \Z^2 \colon \textd_\infty(x/N,D^\cc)>\ffrac1N\bigr\} ,
\end{equation}
and, for each $\delta>0$ and $N$ sufficiently large, also
\begin{equation}
\label{E:1.1a}
D_N\supseteq\bigl\{x \in \Z^2 \colon \textd_\infty(x/N,D^\cc)>\delta\bigr\}
\end{equation}
We will often write $h_N^D(x)$ for~$h^{D_N}_x$. The reason for requiring \eqref{E:1.1} is that, if $D$ has connected components $D^1,\dots,D^m$, then $D_N^i$ and $D_N^j$ for any $i\ne j$ are at least two lattice steps apart and so the fields $(h^{D^1}_N,\dots, h^{D^m}_N)$ are independent. See Fig.~1 for an illustration.

Associated with each sample of $h^ D_N$ is a point process of centered extreme local maxima. This process is realized as the random Borel measure on $\overline D\times\R$ defined by
\begin{equation}
	\eta^ D_{N,r} := \sum_{x \in  D_N}  
		\1_{\{h_{N}^D(x) = \max_{ z \in \Lambda_r(x)} h_{N}^D(z)\}}\delta_{x/N}\otimes\delta_{ \,h_{N}^D(x) - m_N },
\end{equation}
where
\begin{equation}
\label{E:2.4ui}
\Lambda_r(x):=\{z\in\Z^2\colon |z-x|\le r\}
\end{equation}
defines the meaning of ``local,'' the centering sequence $m_N$ is as in \eqref{E:1.3}, $\1_A$ is the indicator of event~$A$ and $\delta_z$ denotes the unit (Dirac) point-mass at~$z$. 

\begin{figure}[t]
\vglue0.2cm
\centerline{\includegraphics[width=3.4truein]{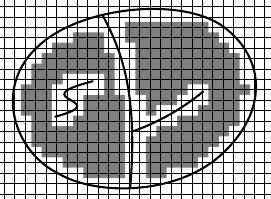}}
\begin{quote}
\small
\vglue0.2cm
{\sc Fig~1.\ }
\label{fig1}
An illustration of a domain $D\in\mathfrak D$ (the region in the plane bounded by the solid curves) along with its discrete approximation~$D_N$ (the lattice vertices contained in the gray areas) corresponding to the set on the right of \eqref{E:1.1}.  The lattice spacing is $1/N$.
\normalsize
\end{quote}
\end{figure}
 

The family $\{\eta_{N,r}\}$ belongs to the space of locally-finite Borel measures on a locally compact metric space which, if endowed with vague topology, permits us to consider its distributional limit points. Let $\leb(\cdot)$ denote the Lebesgue measure on~$\R^2$ and, given a sigma-finite Borel measure~$\lambda$, let $\text{\rm PPP}(\lambda)$ denote the associated Poisson point process. If~$\lambda$ is itself random then the law of $\text{\rm PPP}(\lambda)$ is also averaged over the law of~$\lambda$. The starting point of our derivations is:

\begin{theorem}[Limit process]
\label{thm:1}
For each $ D \in\mathfrak D$ there is a random Borel measure $Z^D(\textd x)$ on~$\overline D$ with $Z^D(\overline D)<\infty$ a.s.\ such that for any $r_N$ satisfying $r_N\to\infty$ and $r_N/N\to0$,  
\begin{equation}
\label{e:1.9}
\eta^ D_{N,r_N}\,\underset{N\to\infty}\Lawlongarrow\, \text{\rm PPP}\bigl(Z^ D(\textd x)\otimes\texte^{-\alpha h}\textd h\bigr) \,,
\end{equation}
where $\alpha:=2/\sqrt g$. The measure~$Z^D$ has the following properties almost surely:
\begin{enumerate}
\item[(1)] $Z^D$ is non-atomic and $Z^D(\partial D)=0$,
\item[(2)] for any non-random Borel~$A\subset\overline D$ with $\leb(A)=0$ we have $Z^D(A)=0$ ,
\item[(3)] $Z^D(A)>0$ for any $A\subseteq D$ non-empty open and so $\supp Z^D=\overline D$.
\end{enumerate}
Moreover, if~$D$ is the union of pairwise-disjoint sets $D^1,\dots,D^m\in\mathfrak D$, then 
\begin{equation}
\label{e:1.5}
Z^{ D}\laweq Z^{ D^1} + \dots + Z^{ D^m} ,
\end{equation}
where $Z^{ D^1},\dots, Z^{ D^m}$ are independent and regarded as measures on~$D$.
\end{theorem}

We will sometimes refer to property~(2) as stochastic absolute continuity.
Since~$Z^D$ puts zero mass on~$\partial D$ almost surely, we will regard it from now on as a measure on~$D$ only. As a consequence of the existence of the limit we get:

\begin{corollary}[Shift and dilation invariance]
\label{cor-shift-scale}
For~$a\in\C$ and $\lambda>0$, let us denote $a+\lambda D:=\{a+\lambda z\colon z\in D\}$. For all~$D\in\mathfrak D$, all~$a\in\C$ and all~$\lambda>0$,
\begin{equation}
\label{E:2.6t}
Z^{\,a+\lambda D}(a+\lambda\textd x)\laweq \lambda^4 \,Z^D(\textd x).
\end{equation}
\end{corollary}

These properties will play an important role in the sequel.
 
To give some immediate explanation of the fourth-power of~$\lambda$ appearing in \eqref{E:2.6t}, note that \eqref{e:1.9} shows, for $S_K:=(0,K)^2$,
\begin{equation}
P\bigl(\,\max_{x\in V_{KN}}h^{V_{KN}}_x\le m_N+t\bigr)\,\,\underset{N\to\infty}\longrightarrow\,\,E\bigl(\texte^{-\alpha^{-1}\texte^{-\alpha t}\,Z^{S_K}(S_K)}\bigr)
\end{equation}
while, interpreting~$V_{KN}$ as the lattice version of~$S_1$ at scale~$KN$,
\begin{equation}
P\bigl(\,\max_{x\in V_{KN}}h^{V_{KN}}_x\le m_{KN}+t\bigr)\,\,\underset{N\to\infty}\longrightarrow\,\,E\bigl(\texte^{-\alpha^{-1}\texte^{-\alpha t}\,Z^{S_1}(S_1)}\bigr).
\end{equation}
Since $m_{KN}-m_N=2\sqrt g\log K+o(1)$, this yields $Z^{S_K}(S_K)\laweq K^4Z^{S_1}(S_1)$. The same argument works with scaled up versions of arbitrary domains.

\subsection{Relations between $Z^D$-measures}
As mentioned above, the laws of random measures $\{Z^ D \colon   D \in\mathfrak D\}$ are  interrelated. To state these relations, we need some more definitions.

For $D\in\mathfrak D$, let $\Pi^D(x,\textd y)$ denote the Poisson kernel, i.e., the harmonic measure, in~$D$ relative to~$x$.  We may take $\Pi^D$ to be defined as  the exit distribution from~$D$ of the Brownian motion started at~$x$.  (Necessarily, $\Pi^D(x,\cdot)$ is then concentrated on~$\partial D$.)  Given $D,\widetilde D\in\mathfrak D$ subject to $\widetilde D\subseteq D$, let $C^{D,\widetilde D}\colon\widetilde D\times\widetilde D\to\R$ be defined by
\begin{equation}
\label{E:1.6}
C^{D,\widetilde D}(x,y)
:=g\int_{\partial  D}\Pi^{ D}(x,\textd z)\log|z-y|-g\int_{\partial \wt{ D}}\Pi^{\wt{ D}}(x,\textd z)\log|z-y|,
\end{equation}
where~$g$ is as in, e.g., \eqref{E:1.3} above. We note:

\begin{lemma}
\label{lemma-2.3a}
$x,y\mapsto C^{D,\widetilde D}(x,y)$ is symmetric and positive definite on~$\wt D\times\wt D$. It is also harmonic on~$\wt D$ in each coordinate. Consequently, there is a centered Gaussian field $\Phi^{D,\wt D}$ on~$\wt D$ with covariance~$C^{D,\widetilde D}$. The sample paths of this field are harmonic, and thus smooth,~on~$\wt D$,~a.s.
\end{lemma}

\begin{figure}[t]
\vglue0.2cm
\centerline{\includegraphics[width=4truein]{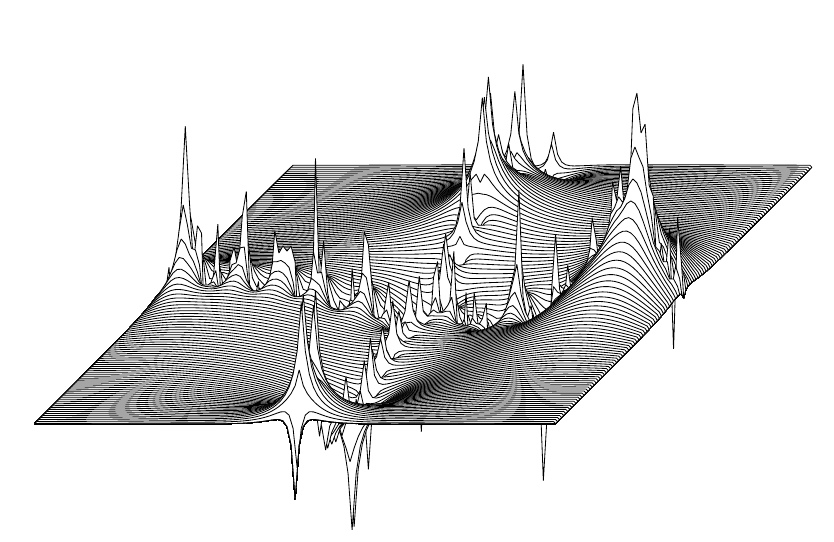}}
\begin{quote}
\small
\vglue0.2cm
{\sc Fig~2.\ }
\label{fig2a}
A sample of the field~$\Phi^{D,\wt D}$ for~$D:=(-1,1)^2$ and $\wt D$ equal to~$D$ with the coordinate axes removed. The sample paths of~$\Phi^{D,\wt D}$ are smooth on~$\wt D$ but they become very rough (in fact, undefined as functions) on~$D\smallsetminus\wt D$.
\normalsize
\end{quote}
\end{figure}

A very useful facts about the DGFF is that it obeys the \emph{Gibbs-Markov property}: Given $U\subseteq V$, conditional on $\{h_x^V\colon x\in V \setminus U\}$, the field in~$U$ is the sum of~$h^U$ and an independent random Gaussian function, which is the harmonic extension to~$U$ of the values of~$h^V$ in~$V\smallsetminus U$; see \eqref{e:2.3} for a precise formulation. This is reflected for measures $Z^D$ as follows:

\begin{theorem}[Gibbs-Markov property]
\label{thm:2}
Let $D,\widetilde D\in\mathfrak D$ with $\wt{ D} \subseteq  D$ but $\leb(D\smallsetminus\widetilde D)=0$. Then 
\begin{equation}
\label{e:1.9a}
	Z^D(\textd x) \laweq  \texte^{\alpha  \Phi^{ D, \wt{ D}}(x)}\,Z^{\wt{ D}}(\textd x),
\end{equation}
where $Z^{\wt{ D}}$ and $ \Phi^{ D, \wt{ D}}$ on the right are independent and (as before) $\alpha:=2/\sqrt g$. (The right-hand side~\eqref{e:1.9a} is meaningful since $Z^{\widetilde D}$ (as well as~$Z^D$) puts zero mass on $D\smallsetminus\widetilde D$ a.s.)
\end{theorem}

The next property explicates on how conformal invariance of the CGFF reflects itself in the behavior of (the law of)~$Z^D$ under conformal transformations of~$D$:

\begin{theorem}[Conformal invariance]
\label{thm:3}
Let~$D\in\mathfrak D$ and let $f$ be  an analytic bijection of~$D$ onto $f(D)\in\mathfrak D$. The laws of~$Z^D$ and~$Z^{f(D)}$ are then related by the transformation rule
\begin{equation}
\label{E:1.10a}
(Z^{f(D)}\circ f)\,(\textd  x)\laweq |f'(x)|^4\,Z^D(\textd x).
\end{equation}
In particular, for $D$ simply connected and with $\rad_D(x)$ denoting the conformal radius of~$D$ at~$x$, the law of the random measure
\begin{equation}
\label{E:1.17c}
\rad_D(x)^{-4}\,Z^D(\textd x),
\end{equation}
is invariant under conformal maps of~$D$. 
\end{theorem}

Obviously, \eqref{E:1.10a} vastly generalizes the statement in Corollary~\ref{cor-shift-scale}.
Moreover, in combination with Theorem~\ref{thm:1}, one can regard \twoeqref{E:1.10a}{E:1.17c} as a form of universality of the law of extreme local maxima with respect to changes in the underlying domain.

For the benefit of the reader we recall that, for simply connected~$D$, the conformal radius $\rad_D(x)$ is the (unique) value of $|f'(x)|^{-1}$ for any conformal map $f\colon D\to\D$, where $\D$ denotes the unit disc $\D:=\{z\in\C\colon |z|<1\}$, such that $f(x)=0$. Interestingly, this notion is closely related to objects introduced earlier in this section (specifically, the kernel $C^{D,\widetilde D}$) via 
\begin{equation}
\label{E:rad}
\rad_D(x)=\exp\Bigl\{\int_{\partial D}\Pi^D(x,\textd z)\log|z-x|\Bigr\}.
\end{equation}
This formula thus naturally generalizes the conformal radius to non-simply connected domains.
 
As is immediate from Theorem~\ref{thm:3} and the Riemann Mapping Theorem, the conformal transformation rule \eqref{E:1.10a} determines the law of~$Z^D$ from that of~$Z^\D$ for any simply connected domain $D\in\mathfrak D$. The Gibbs-Markov property \eqref{e:1.9a},  the factorization over connected components \eqref{e:1.5}  and the stochastic absolute continuity with respect to the Lebesgue measure permit us to represent non-simply connected domains as restrictions, up to a set of vanishing Lebesgue measure, of simply connected ones.  (This is because we assumed that~$\partial D$ has only a finite number of internal components which can then be joined by smooth curves to the outer boundary of~$D$.)  So, in fact, the law of~$Z^{\D}$ determines the law of~$Z^D$ for all~$D\in\mathfrak D$.

\subsection{Tail behavior}
Key for our later developments will be the control of the joint law of the scaled position and centered value of the absolute maximum, conditioned on the maximum being very large:

\begin{theorem}
\label{thm:1.4}
Fix $D\in\mathfrak D$, recall that $h^D_N$ denotes the DGFF in~$D_N$ defined in \eqref{E:1.1} and~$m_N$ denotes the centering sequence from \eqref{E:1.3}. Set, as before, $\alpha:=2/\sqrt g$. Then for all open $A \subseteq  D$,
\begin{equation}
\label{E:1.13a}
	\lim_{t \to \infty} \,\lim_{N \to \infty} \,
		\frac1t \texte^{\alpha t}\,
			P \Bigl( N^{-1} \operatornamewithlimits{argmax}_{D_N}\, h^{ D}_N \in A \, ,\,\,
				\max_{x\in D_N} h_N^{ D} > m_N + t \Bigr) 
			= \int_A \psi^{\,D}(x) \,\textd x
\end{equation}
with
\begin{equation}
\label{E:1.14a}
\psi^{\,D}(x) :=  c_\star\,1_D(x)\,\exp \Bigl\{ 2 \int_{\partial D} \Pi^ D(x, \textd z)  \log |z-x|\,\Bigr\},
\end{equation}
where $c_\star\in(0,\infty)$ is a constant independent of~$D$. 
\end{theorem}

Theorem~\ref{thm:1.4} generalizes Proposition~2.2 in Bramson, Ding and Zeitouni~\cite{BDingZ} where the limit \eqref{E:1.13a} with \emph{some} (continuous) function~$\psi^D$ was shown to hold for $D:=(0,1)^2$. Our main contribution here is that we identify the limiting density~$\psi^{\,D}$ explicitly with the square of the conformal radius.  We will say more about the constant~$c_\star$ in Section~\ref{sec2.5}. 
The asymptotic \eqref{E:1.13a} translates for the $Z^D$-measure as follows:

\begin{corollary}
\label{cor-Z-measure}
Given $D\in\mathfrak D$, define the random probability measure $\widehat Z^D(A):={Z^D(A)}/{Z^D(D)}$. Then for any open set $A\subseteq D$,
\begin{equation}
\label{E:1.22c}
\lim_{\lambda\downarrow0}\frac{E\bigl(\,\widehat Z^D(A)\,(1-\texte^{-\lambda Z^D(D)})\bigr)}{\lambda\log(\ffrac1\lambda)}= \int_A \psi^{\,D}(x)\,\textd x
\end{equation}
and, consequently,
\begin{equation}
\label{E:1.23c}
\lim_{\lambda\downarrow0}\frac{E\bigl(Z^D(A)\,\texte^{-\lambda Z^D(D)}\bigr)}{\log(\ffrac1\lambda)}=\int_A \psi^{\,D}(x)\,\textd x.
\end{equation}
\end{corollary}

These asymptotics serve as an important ingredient for the proof of Theorem~\ref{thm:3}. Indeed, thanks to the explicit form \eqref{E:1.14a} we readily check that $x\mapsto\psi^D(x)$ is invariant under simultaneous rotations of~$D$ and~$x$ --- while, of course, keeping the underlying lattice in the same position. With the help of the Gibbs-Markov property, this implies \emph{rotation invariance} of the law of the~$Z^D$-measure (see Lemma~\ref{lemma-6.1} for a precise statement). In conjunction with the scaling law in Corollary~\ref{cor-shift-scale}, this is then upgraded to the conformal-transformation rule \eqref{E:1.10a}. See Theorem~\ref{thm:6.1} for an axiomatic formulation.

We note that \eqref{E:1.23c} implies \eqref{E:1.22c} via l'Hospital's Rule  and so the two statements are actually equivalent.  With some extra work, these expressions also yield an asymptotic form for the Laplace transform of the integral of any continuous function~$f\colon\overline D\to\R$ against~$Z^D(\textd x)$; see Proposition~\ref{prop-3}.

\subsection{Connection to Liouville Quantum Gravity}
As a by-product of the considerations underlying Theorem~\ref{thm:3}, we obtain a list of conditions that identify the law of $Z^D$-measures uniquely up to an overall multiplicative constant. We give these in an axiomatic form:

\begin{theorem}
\label{thm:1.6}
Suppose~$\{M^D\colon D\in\mathfrak D\}$ is a family of random Borel measures that obey:
\settowidth{\leftmargini}{(11)}
\begin{enumerate}
\item[(0)] (support and total mass restriction) $M^D$ is concentrated on~$D$ and $M^D(D)<\infty$ a.s.
\item[(1)] (stochastic absolute continuity) $P(M^D(A)>0)=0$ for any Borel $A\subseteq D$ with $\leb(A)=0$.
\item[(2)] (shift and dilation invariance) For any $a\in\C$ and any $\lambda>0$, 
\begin{equation}
\label{E:2.17ie}
M^{a+\lambda D}(a+\lambda\textd x)\laweq \lambda^4 M^D(\textd x).
\end{equation}
\item[(3)] (Gibbs-Markov property) If $D,\widetilde D\in\mathfrak D$ are disjoint then
\begin{equation}
\label{E:2.18ua}
M^{D\cup\wt D}(\textd x)\laweq M^D(\textd x)+M^{\wt D}(\textd x) \,,
\end{equation}
with $M^D$ and $M^{\wt D}$ on the right regarded as independent.
If $D,\wt D\in\mathfrak D$ instead
obey $\widetilde D\subseteq D$ and $\leb(D\smallsetminus\widetilde D)=0$, then (for $\alpha:=2/\sqrt g$)
\begin{equation}
\label{E:2.19ua}
M^D(\textd x) \laweq \texte^{\alpha  \Phi^{ D, \wt{ D}}(x)} \,M^{\wt{ D}}(\textd x)\,,
\end{equation}
where $\Phi^{ D, \wt{ D}}$ is a centered Gaussian field with covariance $C^{D,\wt D}$, independent of $M^{\wt{ D}}$.
\item[(4)] (Laplace transform tail) There is~$c\in(0,\infty)$ such that for any open set~$A$ with $\overline A\subset D$,
\begin{equation}
\label{E:1.25}
\lim_{\lambda\downarrow0}\frac{E(M^D(A)\texte^{-\lambda M^D(D)})}{\log(\ffrac1\lambda)}= c\int_A \rad_D(x)^2\,\textd x,
\end{equation}
 where $\rad_D(x)$ is as in \eqref{E:rad}.
\end{enumerate}
Then $M^D\laweq c c_\star^{-1} Z^D$ for all $D\in\mathfrak D$, where~$c_\star$ is the constant from Theorem~\ref{thm:1.4}.
\end{theorem}

We note that the behavior under dilation \eqref{E:2.17ie} need not be assumed provided we can show that the limit \eqref{E:1.25}, with $M^D(\cdot)$ replaced by $K^4 M^{D_K}(K^{-1}\cdot)$ for~$D_K:=K^{-1}D$, takes place uniformly in~$K\ge1$. See Remark~\ref{remark-for-Stefan} for precise formulation.  We note that an analogous characterization has recently been shown for the CGFF itself as well; see Berestycki, Powell and Ray~\cite{BPR}. 

\smallskip
Theorem~\ref{thm:1.6} permits us to connect our $Z^D$-measures to the critical Liouville Quantum Gravity associated with the CGFF. This object has a number of constructive definitions whose equivalence has been shown only recently. We will use the definition based on the white-noise approximation of the~CGFF and Seneta-Heyde norming. 

For~$\{B_t\colon t\ge0\}$ the standard Brownian motion started at~$x$ and killed upon exit from~$D$, let $y\mapsto p_t^D(x,y)$ denote the (sub)probability density of~$B_t$.
Writing~$W$ for the Gaussian white noise with respect to the Lebesgue measure on~$(0,\infty)\times D$, let
\begin{equation}
\varphi_t(x):=\int_{(\texte^{-2t},\infty)\times D}p_{s/2}^D(x,z)W(\textd s\,\textd z).
\end{equation}
A covariance computation (based on the Chapman-Kolmogorov conditions for the substochastic kernel $p_t$) shows that $\{\varphi_t(x)\colon x\in D\}$ is a (smooth) Gaussian field tending to the CGFF in the limit as~$t\to\infty$. Defining, for each~$t>0$, the random measure
\begin{equation}
\label{E:2.22}
M_t^D(\textd x):= 1_D(x)\sqrt{t}\,\texte^{\alpha\varphi_t(x)-\frac12\alpha^2\Var(\varphi_t(x))}\,\rad_D(x)^2\,\textd x\,,
\end{equation}
 where $\rad_D(x)$ is as in \eqref{E:rad}. 
Theorem~10 of Duplantier, Rhodes, Sheffield and Vargas~\cite{DRSV2} states that there is a non-trivial, a.s.-finite random Borel measure~$M_\infty^D$ on~$D$ such that, for each Borel~$A\subseteq D$,
\begin{equation}
\label{E:2.25ua}
M_t^D(A)\,\,\underset{t\to\infty}\longrightarrow\,\, M_\infty^D(A),\quad\text{in probability}.
\end{equation}
Disregarding an unimportant normalization constant, \cite[Definition~11]{DRSV2} calls $M_\infty^D$ the \emph{critical Liouville quantum gravity} in~$D$. We now claim:

\begin{theorem}[Identification of~$Z^D$ with critical LQG measure]
\label{thm:LQG}
The measure $M_\infty^D$ obeys conditions (0-4) in Theorem~\ref{thm:1.6}  with $c:=1/\sqrt{2\pi}$ in \eqref{E:1.25}.  In particular, 
\begin{equation}
Z^D(\textd x)\,\,\laweq\,\,  c_\star\sqrt{2\pi}\,\,M^D_\infty(\textd x)
\end{equation}
for each~$D\in\mathfrak D$,  where~$c_\star$ is the constant from Theorem~\ref{thm:1.4}. 
\end{theorem}

The most difficult condition to check in Theorem~\ref{thm:1.6} is \eqref{E:1.25}. This condition can be thought of as a restriction on the tail of the law of~$Z^D(A)$. Indeed, through the use of  Karamata's Tauberian theorem~\cite[Theorem~1.7.1]{BGT}, \eqref{E:1.23c}  is equivalent to\
\begin{equation}
\label{E:2.27ua}
\frac{1}{\log t}\, E\bigl(Z^D(D); \,Z^D(D)\le t\bigr) \,\underset{t \to \infty} \longrightarrow\, c_\star \int_D \rad_D(x)^2 \dd x \,.
\end{equation}
This is consistent with, but insufficient to conclude that~$Z^D(D)$ has a Cauchy tail. We believe (and, in an earlier version of this text, have erroneously claimed a proof of) the following conjecture:

\begin{conjecture}[Cauchy tails]
\label{cor-2.10}
For each~$D\in\mathfrak D$ and each open~$A\subseteq D$, the random variable $Z^D(A)$ has a Cauchy tail. More precisely, for all~$A\subseteq D$ open,
\begin{equation}
\label{E:2.28ua}
t\, P\bigl(Z^D(A)>t\bigr)\,\underset{t\to\infty}\longrightarrow\,\,c_\star\int_A\rad_D(x)^2\textd x,
\end{equation}
where~$c_\star$ is the constant from Theorem~\ref{thm:1.4}. 
\end{conjecture}

We note that a Cauchy tail has been observed previously for the derivative martingale obtained from scale-invariant kernels in the unit box (Barral, Kupiainen, Nikula, Saksman, Webb~\cite[Section~5]{BKNSW}). Incidentally, for this case Madaule~\cite{Madaule} already proved the connection between the asymptotic law of the maximum and the derivative martingale directly.  Wong~\cite{Wong} has recently established a power law tail for the subcritical LQG measures. (\textit{Update in final version}: Wong recently posted a preprint~\cite{Wong2} that proves Conjecture~\ref{cor-2.10}.)
We remark that, as our proof of condition \eqref{E:1.25} is based on the asymptotic form of the covariance kernel, it extends readily to critical multiplicative chaos for general log-correlated Gaussian fields.


\begin{figure}[t]
\vglue-0.5cm
\centerline{\includegraphics[width=4truein]{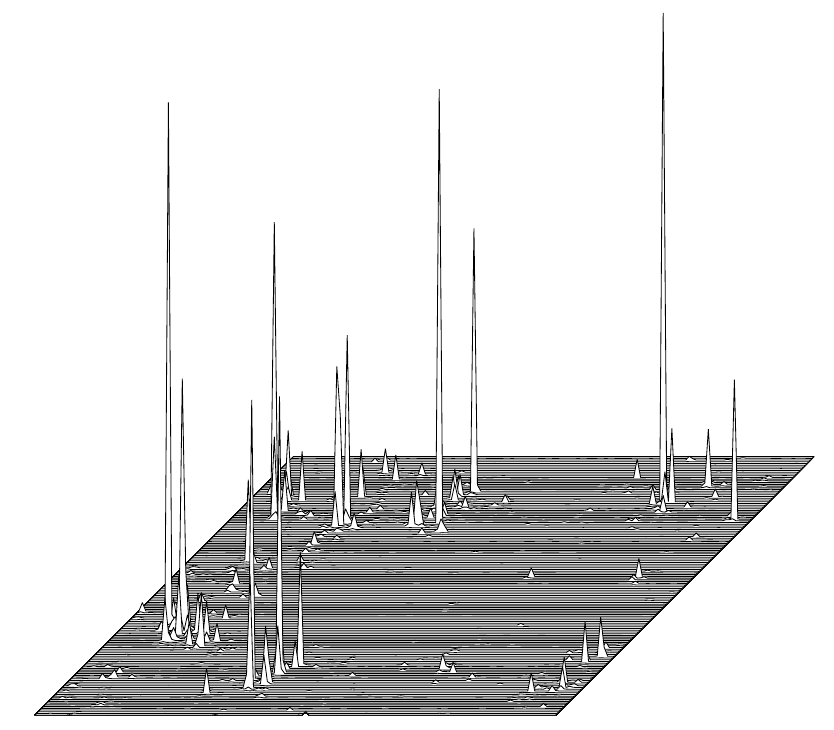}}
\begin{quote}
\small
{\sc Fig~3.\ }
\label{fig2}
An illustration of a  sample of the critical Liouville Quantum Gravity measure on a unit square obtained by simulating the exponential (with exponent $\alpha$) of the DGFF. 
\normalsize
\end{quote}
\end{figure}

\subsection{Remarks}
\label{sec2.5}\noindent
We proceed with some  additional  remarks concerning the above exposition.

\smallskip\noindent
(1) \textit{Domain choices}: Our restriction to domains with a finite number of non-degenerate boundary components comes from the need to ensure convergence of the harmonic measures in discrete approximations to a continuum domain (see Lemma~\ref{lemma-HM}). Although this explicitly excludes punctured domains, we remark that the proof still applies to at least some of them. In fact, the law of~$Z^D$, and thus the whole limit point process, remains unchanged when a polar set is removed from~$D$. (One can see this from the Gibbs-Markov property and the fact that $C^{D,\wt D}=0$ if~$\wt D\subseteq D$ with $D\smallsetminus\wt D$ polar.) \

Thus, for instance,~$Z^\D$ will have the same law as~$Z^{\D\smallsetminus\{0\}}$. An intuitive way to see this is that the DGFF in the lattice version of~$\D\smallsetminus\{0\}$ is that in~$\D$ but pinned at the origin. The pinning has a non-vanishing effect only up to lattice distances~$o(N)$ which, however, is also a region where no relevant local maxima would occur to begin with.

\smallskip\noindent
(2) \textit{Formulation of the Gibbs-Markov property  in Theorem~\ref{thm:1.6}}:
A consequence of the Gibbs-Markov property of the DGFF is that $Z^{D\cup\wt D}$, for $D,\wt D\in\mathfrak D$ disjoint, factors into the sum $Z^D+Z^{\wt D}$ of independent copies of the measures in~$D$ and~$\wt D$, respectively. This is guaranteed automatically for the measures arising in Theorem~\ref{thm:1} but has to be assumed explicitly in Theorem~\ref{thm:1.6}.
 
\smallskip\noindent
(3) \textit{Connection to critical LQG}: 
Theorem~\ref{thm:LQG} touches on the question of uniqueness of the critical LQG measure which has started to be  settled only quite recently. Indeed, as shown in Duplantier, Rhodes, Sheffield and Vargas~\cite{DRSV2}, the definition based on \eqref{E:2.22} agrees (up to overall normalization) with that based on the derivative martingale introduced in Duplantier, Rhodes, Sheffield and Vargas~\cite{DRSV1}.  Recently, Powell~\cite{Powell}, drawing on earlier work of Junilla and Saksman~\cite{Junilla-Saksman} and Huang, Rhodes and Vargas~\cite{HRV}, proved its equivalence with the various other versions  of  the critical LQG measure, including the one obtained by circle average approximation.
 
We note that for the subcritical cases --- corresponding to~$\alpha$ in \eqref{E:2.22} replaced by $\beta\in(0,\alpha)$  and $\sqrt{t}$ term dropped  --- the uniqueness is known in full generality from Shamov~\cite{Shamov}  (and in law already from Robert and Vargas~\cite{Robert-Vargas}).  There even the overall normalization is fixed by the fact that the expectation of the limit measure exists and agrees with that of the base measure (i.e., the Lebesgue measure on~$D$ in our case). This has proved to be quite useful in studying the ``intermediate'' level sets of the DGFF (Biskup and Louidor~\cite{BL-intermediate}). 

Characterizing the critical LQG measure by the conditions in Theorem~\ref{thm:1.6} goes part of the way of Shamov's strategy; a key question is what should replace the Gibbs-Markov property which is special to multiplicative chaos associated with the Gaussian Free Field. (Shamov's condition for behavior under Cameron-Martin shifts is a natural candidate although that requires the measures to be functions of the underlying field.) Note that \eqref{E:1.25} provides a convenient way to fix the normalization of the critical measure. The characterization of $Z^D$ measures proves to be quite useful in the description of (what we call) a near-critical process studied in a forthcoming preprint jointly with S.~Gufler~\cite{BGL}. In particular, it is proved there that~$Z^D$, albeit known to be non-atomic, is supported on a set of vanishing Hausdorff dimension. See Fig.~3.

\smallskip\noindent
(4) \textit{Conditions for uniqueness:}
The conditions in Theorem~\ref{thm:1.6} have been formulated with convenience in mind; no serious attempt was made to design a set of conditions that are minimal possible. We in fact believe that they are not:

\begin{conjecture}
Any family of random Borel measures $\{M^D
\colon D\in\mathfrak D\}$ satisfying conditions (0-3) of~Theorem~\ref{thm:1.6} obeys also condition~(4) thereof.  
\end{conjecture}

\noindent
Key to proving this conjecture seems to be the following statement: Conditional on~$Z^D(D)$ to be large, the measure $Z^D$ is close to a weighted point mass $Z^D(D)\delta_X$ at some random~$X\in D$ which is moreover independent of~$Z^D(D)$. Once this is proved, the Gibbs-Markov property reduces the conjecture to a ``smoothing transformation'' where the conclusions of Durrett and Liggett~\cite{Durrett-Liggett} apply. Still,  as far as the proof of Theorem~\ref{thm:LQG} is concerned, this strategy seems more involved than verifying all conditions for the critical LQG measure directly.

\smallskip\noindent
(5) \textit{Conformal transformation rule:}
In Duplantier, Rhodes, Sheffield and Vargas~\cite[Theorem~13]{DRSV2}, the critical LQG measure is stated to possess a transformation rule under conformal maps that implies, after some rewrites, the one in \eqref{E:1.10a} for maps between simply connected domains. The fourth-power of~$|f'(x)|$ comes by combining the factor $|f'(x)|^2$ coming from transforming the Lebesgue measure with another such factor coming from
\begin{equation}
(\psi^{f(D)}\circ f)(x)=\psi^D(x)\bigl|f'(x)\bigr|^2
\end{equation}
which is directly checked from \eqref{E:1.14a}.
In principle, we could thus refer Theorem~\ref{thm:3} to the combination of Theorem~\ref{thm:LQG} and \cite[Theorem~13]{DRSV2}. However, we still prefer to keep an independent proof of Theorem~\ref{thm:3} because we find it conceptually simpler (deducing it from the behavior of~$Z^D$-measures under dilations and rotations) and, moreover, because we anyway prove it along the way towards proving Theorems~\ref{thm:1.6} and~\ref{thm:LQG}.

\smallskip\noindent
(6) \textit{Constant~$c_\star$ from Theorem~\ref{thm:1.4}:}
Finally, we wish to make a note concerning the actual value of the constant~$c_\star$ appearing in \eqref{E:1.14a} which in light of \eqref{E:1.23c} and Theorem~\ref{thm:1.6} determines the normalization of the~$Z^D$-measure. For this let~$\phi$ denote the DGFF in~$\Z^2\smallsetminus\{0\}$, which is a centered Gaussian process on~$\Z^2$ with $\Cov(\phi_x,\phi_y)=\fraka(x)+\fraka(y)-\fraka(x-y)$, where $\fraka\colon\Z^2\to\R$ is the potential kernel for the simple symmetric random walk on~$\Z^2$. The asymptotic of~$\fraka$ is well known,
\begin{equation}
\label{E:3.23tt}
\fraka(x)=g\log|x|+ c_0+O\bigl(|x|^{-2}\bigr),\qquad |x|\to\infty,
\end{equation}
for~$c_0:=\pi^{-1}(2\gamma+\log8)$, where~$\gamma$ is the Euler constant; see, e.g., Lawler~\cite{Lawler} or Kozma and Schreiber~\cite{KozmaSchreiber04}. In terms of~$g=2/\pi$ and~$c_0$, we then get
\begin{equation}
\label{E:2.31}
c_\star = \texte^{2c_0/g}\,\sqrt{\frac\pi8}\,\,\lim_{r\to\infty} \sqrt{\log r}\,P\biggl(\,\,\bigcap_{x\in\Lambda_r(0)}\Bigl\{\phi_x+\frac2{\sqrt g}\fraka(x)\ge0\Bigr\}\biggr),
\end{equation}
where the limit is known to exist thanks to \cite[Theorem~2.4]{BL-new}. A crucial step towards this conclusion has been made already in \cite[Remark 6.7]{BL-new} albeit with a minor typo in the statement; the conclusion has then been completed in \cite[Theorem 10.7]{Biskup-PIMS}. The value of the limit in \eqref{E:2.31} is not known and it would be interesting to find out if it can be computed (more) explicitly.

\subsection{Outline}
The remainder of the paper is devoted to the proofs of the above results. In Section~\ref{sec2} we begin with some preliminaries concerning the Gibbs-Markov property of the DGFF, its convergence to to CGFF, etc. In particular, we prove Lemma~\ref{lemma-2.3a}. We also recall a couple of standard Gaussian inequalities that will be useful in our subsequent derivations. 

In Section~\ref{sec3}, we proceed to prove Theorem~\ref{thm:1}, dealing with existence and form of the limit of processes $\{\eta^D_{N,r_N}\}$. The crux of the proof is its reduction to the corresponding result for square boxes established in Biskup and Louidor~\cite{Biskup-Louidor}. This is achieved in Proposition~\ref{prop-3.2} whose formulation then permits a quick extraction of Theorem~\ref{thm:2} as well. Corollary~\ref{cor-shift-scale} then readily follows.

Our next task, in Section~\ref{sec4}, is to prove the asymptotic tail behavior in Theorem~\ref{thm:1.4}. Again, we will reduce this to a corresponding result for square boxes proved in  Bramson, Ding and Zeitouni~\cite{BDingZ}. The new non-trivial ingredient is a link between the \emph{location} of the maximizer in a set~$D_N$ and that in a square contained therein but of comparable diameter, assuming a coupling of the two fields via the Gibbs-Markov property. This is the content of (rather technical) Proposition~\ref{prop-4.2}. Worthy of attention is also our proof of the representation \eqref{E:1.14a}. This appears towards the end of proof of Theorem~\ref{thm:1.4} in Section~\ref{sec5.1ab}.

The proofs then proceed in a somewhat different order than the statements. Indeed, we first prove that any family of measures $\{M^D\colon D\in\mathfrak D\}$ satisfying the conditions of Theorem~\ref{thm:1.6} is a distributional limit of derivative-martingale like measures. Although the corresponding statement, Theorem~\ref{thm:5.1}, is close to Theorem~1.4 of Biskup and Louidor~\cite{Biskup-Louidor}, the method of proof is quite different. In particular, we rely on partitioning into triangles rather than squares as regularity of the ``binding'' field, $\Phi^{D,\wt D}$, with respect to such triangular partitions is easier to control (and, moreover, can be related to the DGFF on the triangular lattice). With the derivative-martingale like representation in hand, we then readily prove Theorem~\ref{thm:1.6} as well. 

In Section~\ref{sec6}, we in turn proceed to use the same representation to show that the laws of measures $\{Z^D\colon D\in\mathfrak D\}$ are rotation invariant. Then, in Theorem~\ref{thm:6.1} we give conditions for families of measures that guarantee the validity of the conformal transformation rule~\eqref{E:1.10a}. Not surprisingly, these boil down to invariance with respect to shifts, rotations and dilations and the validity of the Gibbs-Markov property. In Section~\ref{sec8} we then verify that the critical LQG measures~$M^D_\infty$ fulfill the conditions of Theorem~\ref{thm:1.6}. Kahane's convexity inequality is of much help here (this is the reason for our focus on Seneta-Heyde normalization) as well as a kind of concentric decomposition akin to that used in Biskup and Louidor~\cite{BL-new}.

The Appendix discusses the behavior of the harmonic measure under discrete and continuous approximations of our class of underlying domains. These underpin our requirements on the underlying domain as well as their discrete approximations.

\section{Preliminary considerations}
\label{sec2}\noindent
We begin by reviewing and further developing some tools from the theory of Gaussian processes and, particularly, the DGFF. An informed reader may consider just skimming through the statements below and then moving to the next section. 

\subsection{The Gibbs-Markov property}
As already noted above, an important technical input for many of our arguments is the Gibbs-Markov decomposition of the DGFF. Recall that~$h^V$ denotes the DGFF on $V$ with zero values in~$V^\cc$. Given non-empty finite sets $U \subseteq V\subset\Z^2$, denote
\begin{equation}
\label{E:2.1}
	\varphi^{V,U}(x) := E \bigl( h^V(x) \,\big|\, \sigma(h^V(y) ,\, y \in V \smallsetminus U)\bigr).
\end{equation}
Then $h^V-\varphi^{V,U}$ and $\varphi^{V,U}$ are independent with $h^V-\varphi^{V,U}\laweq h^U$. We will often write this as
\begin{equation}
\label{e:2.3}
	h^V\laweq h^U+\varphi^{V,U},
\end{equation}
where on the right-hand side $h^U$ is regarded as independent of $\varphi^{V,U}$.

The Gibbs-Markov decomposition \eqref{e:2.3} permits us to extend certain tightness results known for square boxes to general domains. Let $D\in\mathfrak D$ and write
\begin{equation}
	\Gamma_N^ D(t) := \bigl\{ x \in  D_N \colon  h_N^ D(x) \geq m_N - t \bigr\},
\end{equation}
where we recall that $m_N+O(1)$ is the scale of the absolute maximum in~$V_N$.
Then we have:

\begin{proposition}
\label{lem:4}
For all $D\in\mathfrak D$ and all $t \in \R$, the family $\{|\Gamma_N^ D(t)|\colon\, N \geq 1\}$ of random variables is tight. Moreover, for any $A\subseteq D$ non-empty and open,
\begin{equation}
\label{E:3.4qw}
\lim_{t\to\infty}\,\liminf_{N\to\infty}\,P\bigl(\,\exists x \in \Gamma_N^ D(t)\colon x/N\in A\bigr)=1.
\end{equation}
\end{proposition}

\begin{proposition}
\label{lem:5}
For all $D\in\mathfrak D$ and all $t \in \R$, 
\begin{equation}
	\lim_{r \to \infty}\, \limsup_{N \to \infty}\,
		P \bigl( \exists x,y \in \Gamma_N^D(t)\colon r<|x-y| < N/r\bigr) = 0 \,.
\end{equation}
\end{proposition}

\begin{proposition}
\label{lem:6}
For each $D\in\mathfrak D$ there is a constant $c=c(D)$ such that for all $t\ge1$, all $N \geq 1$ and all sets $A \subseteq  D_N$, 
\begin{equation}
		P \bigl( \exists x \in A \colon  h^D_N(x) \geq m_N + t \bigr)
			\leq c\,\Bigl(\frac{|A|}{N^2}\Bigr)^{1/2}\, 
			t\,\texte^{-\alpha t} \,.
\end{equation}
\end{proposition}

Versions of these results (with explicit estimates on $|\Gamma_N^D(t)|$ in the first proposition) for a square domain $D:=(0,1)^2$ have been proved in Theorem~1.2 and Theorem~1.1 in Ding and Zeitouni~\cite{DZ} (for Propositions~\ref{lem:4} and~\ref{lem:5}, respectively) and Lemma~3.8 of Bramson, Ding and Zeitouni~\cite{BDingZ} (for Proposition~\ref{lem:6}). The additional input required to extend these to general domains is supplied by the following comparison bounds:

\begin{lemma}
\label{lemma-2.4}
Let $U\subseteq V\subseteq W$ be finite subsets of~$\Z^2$. Then for all $c\in\R$ and all $R\ge0$,
\begin{equation}
\label{E:1.15}
P\bigl(\exists x\in U\colon h^V(x)\ge c+R\bigr)\le (1-\tfrac12\texte^{-\frac12 R^2/\sigma^2})^{-1} P\bigl(\exists x\in U\colon h^W(x)\ge c\bigr)
\end{equation}
where $\sigma^2:=\max_{x\in U}\Var(\varphi^{W,V}(x))$. Similarly, for all $r_1<r_2$,
\begin{multline}
\label{E:2.7}
\qquad
P\bigl(\exists x,y\in U\colon h^V(x),h^V(y)\ge c,\, r_1<|x-y|< r_2 \bigr)
\\ \le
4 P\bigl(\exists x,y\in U\colon h^W(x),h^W(y)\ge c,\, r_1<|x-y|< r_2 \bigr).
\qquad
\end{multline}
Finally, for all $c\in\R$ and all integers $R>0$,
\begin{equation}
\label{E:1.16}
P\bigl(|\{x\in U\colon h^V(x)\ge c\}|\ge 2R\bigr)
\le 2 P\bigl(|\{x\in U\colon h^W(x)\ge c\}|\ge R\bigr).
\end{equation}
\end{lemma}

\begin{proofsect}{Proof}
Recall that we may set $h^W:= h^V+\varphi^{W,V}$ with $h^V$ and $\varphi^{W,V}$ independent. Then
\begin{equation}
\bigl\{\exists x\in U\colon h^V(x)\ge c+R,\,\varphi^{W,V}(x)>-R\bigr\}
\subseteq \bigl\{\exists x\in U\colon h^W(x)\ge c\bigr\}.
\end{equation}
Fix some complete ordering of the vertices in~$U$ and decrease the event on the left by requiring that~$\varphi^{W,V}(x)>-R$ at the ``first''~$x\in U$ where $h^V(x)\ge c$. Since $h^V$ and~$\varphi^{W,V}$ are independent, we get \eqref{E:1.15} by a standard Gaussian estimate.

The same argument (modulo duplication of variables) with $R:=0$ works for \eqref{E:2.7} except that to get the factor~$4$ on the right hand side, we need
\begin{equation}
\label{E:2.5}
P\bigl(\varphi^{W,V}(x)>0,\,\varphi^{W,V}(y)>0\bigr)\ge\frac14.
\end{equation}
This follows from the fact that the covariance of $\varphi^{W,V}(x)$ and $\varphi^{W,V}(y)$ is non-negative  (which enables the FKG inequality; see~\cite[Proposition~5.22]{Biskup-PIMS} for a general statement). 

Also in the proof of \eqref{E:1.16} we proceed similarly as for \eqref{E:1.15}. First we note that
\begin{equation}
\Bigl\{\bigl|\{x\in U\colon h^V(x)\ge c,\,\varphi^{W,V}(x)>0\}\bigr|\ge R\Bigr\}
\subseteq\Bigl\{\bigl|\{x\in U\colon h^W(x)\ge c\}\bigr|\ge R\Bigr\}.
\end{equation}
Now the event on the left contains the event that the number of vertices in~$U$ where $h^V(x)\ge c$ is at least~$2R$ and that at least~$R$ of those vertices have~$\varphi^{W,V}(x)>0$. In light of symmetry and independence of~$h^V$ and~$\varphi^{W,V}$, conditional on the ``first''~$2R$ vertices where~$h^V(x)\ge c$, the latter event has probability at least a half. The inequality \eqref{E:1.16} then follows.
\end{proofsect}

Now we can move on to the proofs of the above propositions. In these proofs we will frequently use the notation $S_K$ for the continuum box~$(0,K)^2$ and, abusing the notation somewhat, also its discrete counterpart $(0,K)^2\cap\Z^2$.

\begin{proofsect}{Proof of Proposition~\ref{lem:4}}
First we address the tightness of $|\Gamma_N^D(t)|$.
Since $D$ is bounded, there exist $a\in\R$ and~$K\in\N$ such that $D\subseteq (a,a+K)^2$. By shift invariance of the DGFF, we may assume that~$a:=0$. Then \eqref{E:1.16} reduces the tightness of $|\Gamma_N^D(t)|$ to the tightness of~$|\Gamma_N^{S_K}(t)|$. The definitions also give
\begin{equation}
\bigl|\Gamma_N^{S_K}(t)\bigr|\laweq \bigl|\Gamma_{NK}^{S_1}(t+m_N-m_{NK})\bigr|.
\end{equation}
Since $t\mapsto\Gamma_N^{S_1}(t)$ is non-increasing and $m_{KN}-m_N=g\log K+o(1)=O(1)$ as $N\to\infty$, the tightness of $|\Gamma_N^{S_K}(t)|$ follows from Theorem~1.2 of Ding and Zeitouni~\cite{DZ}.

Moving on to the proof of \eqref{E:3.4qw}, we use that~$A$ is open to find an open square~$S$ such that $\overline S\subseteq A$. The main result of Bramson and Zeitouni~\cite{BZ} implies $P(\max h_N^S\ge m_N-t)\to1$ as $N\to\infty$ and~$t\to\infty$. (Strictly speaking, the result if formulated for a unit square but slight adjustments to the size of~$S$ and the fact that $m_{NK}-m_K=O(1)$ as~$N\to\infty$ make it hold for all open squares.) The claim will then follow from \eqref{E:1.15} by taking $c:=t$ and letting $R\to\infty$ afterwards, provided we can check that
\begin{equation}
\sup_{N\ge1}\,\max_{x\in S_N}\,\Var\bigl(\varphi^{D_N,A_N}(x)\bigr)<\infty.
\end{equation}
To this end, we first recall the representation using the discrete Green functions
\begin{equation}
\label{E:3.16uai}
\Var\bigl(\varphi^{D_N,A_N}(x)\bigr) = G_{D_N}(x,x)-G_{A_N}(x,x).
\end{equation}
Now let $S'$ and~$S''$ be open squares such that $S''\subseteq S\subseteq D\subseteq S'$ and that~$x$ is in the center of their discrete counterparts~$S_N'$ and~$S_N''$. The natural monotonicity of the Green function in the underlying set bounds \eqref{E:3.16uai} by $G_{S_N'}(x,x)-G_{S_N''}(x,x)$,
In light of the logarithmic asymptotic for the Green function, this is bounded by a quantity that depends only on the ratio of side-lengths of~$S'$ and~$S''$. As~$\overline S\subseteq A$, this ratio is bounded uniformly in~$x\in S_N$.
\end{proofsect}

\begin{proofsect}{Proof of Proposition~\ref{lem:5}}
The argument here is fairly analogous to those used in the previous proof: For a square~$S$ of integer side-length, the claim boils down to Theorem~1.1 of Ding and Zeitouni~\cite{DZ}. Thanks to \eqref{E:2.7}, it then extends to any~$D\subseteq S$ as well.
\end{proofsect}

\begin{proofsect}{Proof of Proposition~\ref{lem:6}}
We first recall that, by Lemma~3.8 of Bramson, Ding and Zeitouni~\cite{BDingZ}, for unit-square domain $S_1$ and any $A\subset(0,N)^2\cap\Z^2$, we have
\begin{equation}
		P \bigl( \exists x \in A \colon  h^{S_1}_N(x) \geq m_N + t - s \bigr)
			\leq c\,\Bigl(\frac{|A|}{N^2}\Bigr)^{1/2}\, t\, \texte^{-\alpha (t-s)},\qquad s\ge 0,\, t\ge1.
\end{equation}
In light of $h_N^{S_K}(x)=h^{S_{KN}}(x)=h_{NK}^{S_1}(x)$ and the fact that $m_{NK}-m_K=O(1)$ as~$N\to\infty$, this yields the same claim for domain $S_K:=(0,K)^2$ and any $A\subseteq NS_K$, provided~$c$ is adjusted by a~$K$-dependent factor. The bound \eqref{E:1.15} with $R:=0$ then extends this to all $D\subseteq S_K$ as well.
\end{proofsect}


\subsection{Continuum limit}
\noindent
Our next task is the extraction of the continuum limit for the ``binding field'' $\varphi^{U,V}$ provided both domains properly scale with~$N$. First off, let us check that the purported limit object, the continuum Gaussian field $\Phi^{D,\wt D}$, is well defined:

\begin{proofsect}{Proof of Lemma~\ref{lemma-2.3a}}
As is readily checked, the continuum Green function, i.e., the integral kernel of the operator $(-\frac14\Delta)^{-1}$ in~$D$, admits the representation
\begin{equation}
\label{E:3.19ww}
G^D(x,y)=-g\log|y-x|+g\int_{\partial D}\Pi^D(x,\textd z)\log|z-y|
\end{equation}
for all~$x,y\in D$ with $x\ne y$. Assuming~$\wt D\subseteq D$, this immediately gives
\begin{equation}
\label{E:1.7}
C^{D,\widetilde D}(x,y)=G^D(x,y)-G^{\widetilde D}(x,y),\qquad x,y\in\wt D,\, x\ne y.
\end{equation}
By the properties of the Poisson kernel, $x\mapsto \int_{\partial D}\Pi^D(x,\textd z)\log|z-y|$ is harmonic on~$D$ for all $y\in D$. Hence, so is $x\mapsto C^{D,\wt D}(x,y)$ on~$\wt D$ for all $y\in\wt D$. The Green function is  symmetric, since the Laplacian is a symmetric operator, and thus~$C^{D,\wt D}$ is symmetric as well.

The difference in \eqref{E:1.7} is defined by \eqref{E:3.19ww} only for~$x\ne y$ but it extends continuously to~$x=y$ by the aforementioned harmonicity in~$x$. In particular, $C^{D,\widetilde D}$ is symmetric and harmonic on~$\widetilde D$ in both variables. To see that $C^{D,\widetilde D}$ is positive semidefinite, we again invoke \eqref{E:1.7} and the fact that the Laplacian on~$D$ dominates the Laplacian on~$\widetilde D$ in the sense of comparison of quadratic forms.

It follows that $C^{D,\widetilde D}$ is a covariance kernel on~$\wt D\times\wt D$ and so there is a mean-zero Gaussian process $\Phi^{ D, \wt{ D}}$ on $\wt{ D}$ with covariance function~$C^{ D, \wt{ D}}$. As $C^{D,\widetilde D}$ is smooth, so are the sample paths of~$ \Phi^{ D, \wt{ D}}$ a.s., by Kolmogorov-\v Censtov criteria (or the Fernique inequality discussed below). To get path-wise harmonicity of~$\Phi^{ D, \wt{ D}}$ we just note that the mean and the variance of~$\Delta \Phi^{ D, \wt{ D}}$ vanish.
\end{proofsect}

Next let us address the passage to continuum limit. Given $\wt{D},D\in\mathfrak D$ with $\wt D\subseteq D$, we will write~$\varphi^{ D, \wt{ D}}_N$ as a shorthand for~$\varphi^{ D_N, \wt{ D}_N}$. Then we have:

\begin{proposition}
\label{lem:2}
Let $\wt{D},D\in\mathfrak D$ obey $\wt D\subseteq D$. Then for all $x,y \in \wt{ D}$
\begin{equation}
\label{E:3.20}
\text{\rm Cov}\bigl(
	\varphi^{ D, \wt{ D}}_{N}(\lf xN \rf) \,,\,
	\varphi^{ D, \wt{ D}}_{N}(\lf yN \rf) \bigr)
\,\underset{N\to\infty}\longrightarrow\, 
C^{D,\wt D}(x,y),
\end{equation}
with the convergence uniform over closed subsets of~$\wt D\times\wt D$. In particular, $\varphi_N^{D,\wt D}(\lfloor  N\cdot\rfloor)$ tends to~$\Phi^{D,\wt D}(\cdot)$ in the sense of finite dimensional distributions.
\end{proposition}

For the proof we will need to work with the discrete harmonic measure  $H^D_N(x,y)$ on~$D_N$ which is the probability that the simple symmetric random walk started from~$x\in D_N$ exits~$D_N$ at the point~$y\in\partial D_N$. By Lemma~\ref{lemma-HM}, $\sum_{z\in\partial D}H^D(\lfloor Nx\rfloor,z)\delta_{z/N}$ tends to~$\Pi^D(x,\cdot)$ in the vague (or even weak) topology on the space of probability measures. We also introduce the shorthand
\begin{equation}
\label{E:3.21ww}
D^\delta:=\bigl\{x\in D\colon\textd(x,D^\cc)>\delta\bigr\},
\end{equation}
where~$\delta>0$ is small enough so that~$D^\delta$ is non-empty. Notice that if~$D\in\mathfrak D$, then also $D^\delta\in\mathfrak D$ provided $\delta$ is small enough.

\begin{proofsect}{Proof of Proposition~\ref{lem:2}}
We use the well-known fact that the discrete Green function $G_{D_N}$ also admits a representation of the form \eqref{E:3.19ww}
\begin{equation}
G_{D_N}(x,y)=- \fraka(x-y)+\sum_{z\in\partial D_N}H^D_N(x,z) \fraka(y-z),
\end{equation}
 where~$\fraka$ is the potential kernel on~$\Z^2$. 
Assuming $x,y\in(\wt D^\delta)_N$ for some small~$\delta>0$, we thus get
\begin{multline}
\qquad
\text{\rm Cov}\bigl(\varphi^{D,\wt D}_N(x),\varphi^{D,\wt D}_N(y)\bigr)=
g\sum_{z\in\partial D_N}H^D_N(x,z)\log|y-z|
\\-g\sum_{z\in\partial \wt D_N}H^{\wt D}_N(x,z)\log|y-z|+o(1),
\qquad
\end{multline}
where we used that $H_N^D(x,z)$ is a probability measure to cancel the constant term in \eqref{E:3.23tt} and bound the error uniformly in~$x,y\in\wt D^\delta$.

The functions $z\mapsto\log|y-z|$ are bounded and continuous on a neighborhood of~$\partial D$, resp.,~$\partial\wt D$, and so the convergence in \eqref{E:3.20} follows from \eqref{E:3.21a}. Equicontinuity in~$y\in\wt D^\delta$ then shows that the limit is uniform in~$y\in\wt D^\delta$, for every~$x\in\wt D^\delta$. The uniformity clause in Lemma~\ref{lemma-HM} then applies the same with~$x$ and~$y$ interchanged, from which the claimed joint uniformity follows.
As all processes are centered Gaussian, the convergence of the covariances implies convergence in the sense of finite-dimensional distributions.
\end{proofsect}

\subsection{Some Gaussian inequalities}
Next we will apply some standard Gaussian inequalities to derive much needed tightness of the maximum and Gaussian tails for the fluctuations of both $\varphi_N^{D,\wt D}$ and~$\Phi^{D,\wt D}$. This will be done by invoking the Fernique inequality and the Borell-Tsirelson inequality. We only remind the reader of the statements of these general bounds; detailed proofs can be found in references below and/or the lecture notes by Biskup~\cite{Biskup-PIMS}.

Consider a centered (and separable) Gaussian field $\{X_t\colon t\in\mathfrak X\}$ indexed by points in the totally-bounded (pseudo)metric space $(\mathfrak X,\rho)$, where~$\rho$ is linked to~$X$ via $\rho(t,t'):=[E((X_t-X_{t'})^2)]^{1/2}$. Given any Borel probability measure~$m$ on~$\mathfrak X$ and writing $B_\rho(t,r):=\{t'\in\mathfrak X\colon \rho(t,t')\le r\}$, the Fernique inequality states that
\begin{equation}
\label{E:3.25}
E\bigl(\,\sup_{t\in\mathfrak X}\,X_t\bigr)\le K\sup_{t\in\mathfrak X}\,\int_0^\infty \sqrt{\log\frac1{m(B_\rho(t,r))}}\,\textd r
\end{equation}
for some universal constant~$K<\infty$. See, e.g., Adler~\cite[Theorem~4.1]{Adler}. The Borell-Tsirelson inequality in turn controls the fluctuation of $\sup_{t\in\mathfrak X}X_t$ from its mean by stating that these have tails of a centered Gaussian with variance $\sigma^2:=\sup_{t\in\mathfrak X}E(X_t^2)$. Based on these, we now claim:

\begin{proposition}
\label{lem:3}
Let $D,\wt{D}\in\mathfrak D$ obey $\wt D\subseteq D$. Then for all $\delta>0$ small,
\begin{equation}
\label{E:2.15}
\sup_{N\ge1} \,\,E\Bigl(\,\, \max_{x \in \wt{ D}_{N}^{\delta}} \varphi^{ D, \wt{ D}}_N(x)\Bigr)<\infty.
\end{equation}
Moreover, there are $c,c'\in(0,\infty)$ --- which depend on $D$, $\wt D$ and $\delta$ --- such that for all $N \geq 1$ and all~$t\ge0$,
\begin{equation}
\label{E:2.16}
\BbbP \biggl( \Bigl|\,
	\max_{x \in \wt{D}_N^\delta} \varphi^{ D, \wt{ D}}_N(x) - 
E \max_{x \in \wt{D}_N^\delta} \varphi^{ D, \wt{ D}}_N(x) \Bigr| > t \biggr)
	\leq c \texte^{-c' t^2} \,.
\end{equation}
Completely analogous bounds --- with $x\in\wt D^\delta_N$ replaced by $x\in\wt D^\delta$ --- hold for~$\Phi^{D,\wt D}$.
\end{proposition}

\begin{proof}
As $\wt D^\delta$ can be covered by a finite number of open squares whose closure is still contained in~$\wt D$, it suffices to prove this for one such a square~$S$.
Consider the random fields~$\varphi_N^{D,\wt D}(\lfloor \cdot N\rfloor)$ and~$\Phi^{D,\wt D}(\cdot)$ on~$S$. As observed before, the covariances of these are bounded throughout~$S$ uniformly in~$N$ while the harmonicity ensures that they are Lipschitz functions in both arguments, uniformly in~$N$ over~$S$. Letting~$L$ denote the requisite Lipschitz constant we have $\rho(x,y)\le\sqrt{L|x-y|}$. In particular, the ball $B_\rho(x,r)$ in $\rho$-metric thus contains the part of the Euclidean ball of radius $(r/L)^2$ that is contained in~$S$. 

Using the normalized Lebesgue measure on~$S$ for~$m$ in \eqref{E:3.25}, the regularity of the boundary of~$S$ implies that $m(B_\rho(t,r))\ge c (r/L)^4$ for~$r>0$ small where $c>0$ is a constant. It follows that the integral in \eqref{E:3.25} is finite, uniformly in~$N$, thus proving \eqref{E:2.15}. For~\eqref{E:2.16} we just apply the Borell-Tsirellson inequality with the fact that the variances of the two fields are uniformly bounded on~$\wt D^\delta$.
\end{proof}

Similar arguments also permit us to prove the following claim that will be useful in our approximation arguments in Section~\ref{sec3}.

\begin{proposition}
\label{prop-inner-apprx}
Let $D\in\mathfrak D$ and recall the set $D^\delta$ from \eqref{E:3.21ww}. Then for each $x\in D$,
\begin{equation}
\label{E:3.28rr}
\sup_{\begin{subarray}{c}
D'\colon x\in D'\\D^\delta\subseteq D'\subseteq D
\end{subarray}}
\text{\rm Var}\bigl(\Phi^{D,D'}(x)\bigr)\,\underset{\delta\downarrow0}\longrightarrow\,0
\end{equation}
locally uniformly in~$D$ and, 
for any open set $A$ with $\overline A\subset D$, also
\begin{equation}
\label{E:3.26wq}
\sup_{\begin{subarray}{c}
D'\colon A\subseteq D'\\D^\delta\subseteq D'\subseteq D
\end{subarray}} P\Bigl(\,\,\sup_{x\in A}\,\,\bigl|\Phi^{D,D'}(x)\bigr|>\epsilon\Bigr)\,\underset{\delta\downarrow0}\longrightarrow\,0
\end{equation}
for any $\epsilon>0$.
\end{proposition}

\begin{proofsect}{Proof}
Let $D'$ be such that~$D^\delta\subseteq D'\subseteq D$. Then
\begin{equation}
\Phi^{D,D^\delta}\laweq \Phi^{D,D'}+\Phi^{D',D^\delta}
\end{equation}
with the two fields on the right independent. This shows that $\text{\rm Var}(\Phi^{D,D'}(x))\le\text{\rm Var}(\Phi^{D,D^\delta}(x))$ for all $x\in D^\delta$. By Lemma~\ref{lemma-cont-h.m.} we have $C^{D,D^\delta}(\cdot,\cdot)\to0$ locally uniformly in~$D$ and, using Cauchy-Schwarz, the same thus applies to the supremum of $|C^{D,D'}(\cdot,\cdot)|$ over all~$D'$ as above.
 
This observation immediately yields \eqref{E:3.28rr}. For \eqref{E:3.26wq} we again notice that, by a covering argument it suffices to prove this for~$A$ equal to a square~$S$ with $\overline S\subset D$. Thanks to harmonicity, also $|C^{D,D'}(x,y)-C^{D,D'}(x,y')|/|y-y'|\to0$ as $\delta\downarrow0$ uniformly in~$x,y,y'\in S$ with $y\ne y'$ and in~$D'$ squeezed between~$D^\delta\cup S$ and~$D$. This implies that the (pseudo)metric $\rho$ induced on~$S$ by the Gaussian field $\Phi^{D,D'}$ obeys $\rho(x,y)\le \sqrt{L_n|x-y|}$, where $L_n\to0$ as $n\to\infty$. Using again the normalized Lebesgue measure on~$S$ for~$m$ in \eqref{E:3.25}, we find out (as in the previous proof) that $m(B_\rho(x,r))\ge c (r/L_n)^4$ for some $c>0$. As the $\rho$-diameter of~$S$ is at most order~$L_n$, we get
\begin{equation}
\sup_{\begin{subarray}{c}
D'\colon A\subseteq D'\\D^\delta\subseteq D'\subseteq D
\end{subarray}} E\Bigl(\,\,\sup_{x\in A}\,\,\bigl|\Phi^{D,D'}(x)\bigr|\Bigr)\,\underset{\delta\downarrow0}\longrightarrow\,0\,.
\end{equation} 
However, by the Borell-Tsirelson inequality, $\sup_{x\in A}\,\,\Phi^{D,D'}(x)$ has uniformly small Gaussian tails whose ``variance'' tends to zero uniformly in~$D'$ as above. This yields \eqref{E:3.26wq} without the absolute value. The symmetry $\Phi^{D,D'}\laweq-\Phi^{D,D'}$ then implies the full statement by a union bound.
\end{proofsect}

To make the list of our preliminaries complete, we also recall the Kahane convexity inequality in the form that will be suitable to our needs:

\begin{proposition}[Kahane's inequality]
\label{lemma-Kahane}
Let $\Phi(x)$ and $\wt\Phi(x)$ be continuous Gaussian fields indexed by points in an open set~$D\subset\R^2$. Assume
\begin{equation}
\text{\rm Cov}\bigl(\Phi(x),\Phi(y)\bigr)\ge\text{\rm Cov}\bigl(\wt\Phi(x),\wt\Phi(y)\bigr),\qquad x,y\in D.
\end{equation}
Then for any finite Borel measure $\sigma$ on~$D$,
\begin{equation}
E\biggl(\exp\Bigl\{-\int\sigma(\textd x)\texte^{\Phi(x)-\frac12\text{\rm Var}(\Phi(x))}\Bigr\}\biggr)
\ge E\biggl(\exp\Bigl\{-\int\sigma(\textd x)\texte^{\wt\Phi(x)-\frac12\text{\rm Var}(\wt\Phi(x))}\Bigr\}\biggr).
\end{equation}
\end{proposition}

\begin{proofsect}{Proof}
This goes back to Kahane~\cite{Kahane}. The proof boils down to a computation of the $t$-derivative of the expectation with~$\Phi$ replaced by $\sqrt t\,\Phi+\sqrt{1-t}\,\wt\Phi$, with $\Phi$ and~$\wt\Phi$ regarded as independent. The same argument applies even when the exponential (of the integral) is replaced by any convex function (of the integral) with at most a polynomial growth at infinity.
\end{proofsect}

\section{Convergence and Gibbs-Markov property}
\label{sec3}\noindent
Here we will prove Theorems~\ref{thm:1} and Theorem~\ref{thm:2}. For Theorem~\ref{thm:1} the key idea is to approximate general domains from within by families of disjoint squares for which the result can be drawn from Biskup and Louidor~\cite{Biskup-Louidor}. Notwithstanding, various limit statements need to be invoked and this is where the propositions from the previous section will come useful. A bonus point is that the  arguments then directly yield Theorem~\ref{thm:2} as well.

\subsection{Reduction to subdomains}
We first address the reduction to a subset of the underlying domain. Pick $D\in\mathfrak D$. We will generally consider approximations~$D^n\in\mathfrak D$ of~$D$ such that
\begin{equation}
\label{E:3.1}
D^n\subseteq D,\,\,\,\forall n\ge1,\quad\text{and}\quad \lim_{n\to\infty}\leb(D\smallsetminus D^n)=0.
\end{equation}
Next we introduce a (continuous) mollifier that will help us deal with various boundary issues.
For~$\epsilon>0$, let $\chi_{n,\epsilon}:D\to[0,1]$ be given by
\begin{equation}
\chi_{n,\epsilon}(x)=1\wedge\frac1{\epsilon}\bigl(\textd(x,(D^n)^\cc)-\epsilon\bigr)_+
\end{equation}
Obviously, $\chi_{n,\epsilon}$ is continuous with $\chi_{n,\epsilon}(x)=0$ when $\textd(x,(D^n)^\cc)\le\epsilon$ and $\chi_{n,\epsilon}(x)=1$ when $\textd(x,(D^n)^\cc)\ge2\epsilon$. In addition, $\epsilon\mapsto\chi_{n,\epsilon}(x)$ is non-increasing for all~$x$. 

The first step of the proof is a passage from the process in~$D$ to that in~$D^n$. This will naturally lead us to consider the continuum Gaussian field $\{\Phi^{D,D^n}(x)\colon x\in D^n\}$ which, we recall, has zero mean and covariance $C^{D,D^n}$ as defined in~\eqref{E:1.7}. Given a realization of~$\Phi^{D,D^n}$ and a continuous, compactly supported function $f\colon\overline D\times\R\to[0,\infty)$, define
\begin{equation}
\label{E:3.3}
f_{\epsilon,\Phi}^{D,D^n}(x,h):=f\bigl(x,\,h+\Phi^{D, D^n}(x)\bigr)\chi_{n,\epsilon}(x),\qquad (x,h)\in D^n\times\R.
\end{equation}
Note that, for each fixed~$n$,~$\epsilon$ and $\Phi^{D,D^n}$, this is still a bounded, continuous, non-negative function with compact support. The main reduction step is now the content of:

\begin{proposition}
\label{prop-3.1}
Suppose that $D$ and~$D^n$ satisfy \eqref{E:3.1} and let~$r_N$ be such that $r_N\to\infty$ with $r_N/N\to0$. Then for any continuous function $f\colon\overline D\times\R\to[0,\infty)$ with compact support,
\begin{equation}
\lim_{n\to\infty}\,\limsup_{\epsilon\downarrow0}\,\limsup_{r\to\infty}\,\limsup_{N\to\infty}\,\Bigl|\,E \texte^{-\langle\eta_{N,r_N}^D \,,\, f \rangle}-E \texte^{-\langle\eta_{N,r}^{D^n} \,,\, f^{D,D^n}_{\epsilon,\Phi}\rangle}\Bigr|=0.
\end{equation}
Here the second expectation is over~$\eta_{N,r}^{D^n}$ and~$\Phi^{D,D^n}$, regarded as independent.
\end{proposition}

We postpone the proof temporarily as it will require a lot of technical steps, and instead demonstrate how this can be used to establish Theorems~\ref{thm:1} and~\ref{thm:2}. For that we need to prove: 

\begin{proposition}
\label{prop-3.2}
Suppose that $D$ and~$D^n$ satisfy \eqref{E:3.1} and that, for each~$n\ge1$, there is an a.s.\ finite random Borel measure $Z^{D^n}$ on~$D^n$ such that, for any $r'_N\to\infty$ with $r_N'/N\to0$,
\begin{equation}
\label{E:3.5a}
\eta_{N,r_N'}^{D^n}\,\underset{N\to\infty}\Lawlongarrow\,\text{\rm PPP}\bigl(Z^{D^n}(\textd x)\otimes\texte^{-\alpha h}\textd h\bigr),
\end{equation}
where $\alpha:=2/\sqrt g$. Then there is an a.s.\ finite random measure~$Z^D(\textd x)$ on~$\overline D$ such that, for any $r_N\to\infty$ with~$r_N/N\to0$,
\begin{equation}
\eta_{N,r_N}^{D}\,\underset{N\to\infty}\Lawlongarrow\,\text{\rm PPP}\bigl(Z^D(\textd x)\otimes\texte^{-\alpha h}\textd h\bigr).
\end{equation}
Moreover, we have
\begin{equation}
\label{E:3.7b}
Z^{D^n}(\textd x)\,\texte^{\alpha\Phi^{D,D^n}(x)}\,\underset{n\to\infty}\Lawlongarrow\,Z^D(\textd x),
\end{equation}
where on the left $\Phi^{D,D^n}$ is a mean-zero Gaussian field on~$D^n$ with covariance~$C^{D,D^n}$, regarded as  independent of~$Z^{D^n}$.
\end{proposition}

\begin{proofsect}{Proof}
Let~$r_N$ be any sequence with $r_N\to\infty$ with $r_N/N\to0$.
By Proposition~\ref{prop-3.1}, we may find a sequence $r_N'$ satisfying $r_N'\to\infty$ with $r_N'/N\to0$ such that, for any bounded and continuous function $f\colon\overline D\times\R\to[0,\infty)$,
\begin{equation}
\label{E:3.8a}
\lim_{n\to\infty}\,\limsup_{\epsilon\downarrow0}\,\limsup_{N\to\infty}\,\Bigl|\,E \texte^{-\langle\eta_{N,r_N}^D \,,\, f \rangle}-E \texte^{-\langle\eta_{N,r_N'}^{D^n} \,,\, f^{D,D^n}_{\epsilon,\Phi}\rangle}\Bigr|=0.
\end{equation}  
Now $f^{D,D^n}_{\epsilon,\Phi}$ is continuous with compact support for a.e.\ realization of~$\Phi^{D,D^n}$, so conditioning on~$\Phi^{D,D^n}$ and applying \eqref{E:3.5a} along with the Bounded Convergence Theorem yields
\begin{equation}
\label{E:3.9b}
E \texte^{-\langle\eta_{N,r_N'}^{D^n} \,,\, f^{D,D^n}_{\epsilon,\Phi}\rangle}
\,\underset{N\to\infty}\longrightarrow\,E\biggl(\exp\Bigl\{-\int(1-\texte^{-f^{D,D^n}_{\epsilon,\Phi}})\,Z^{D^n}(\textd x)\texte^{-\alpha h}\textd h\Bigr\}\biggr),
\end{equation}
where $Z^{D^n}$ is independent of~$\Phi^{D,D^n}$. Since $\langle\eta_{N,r_N}^D \,,\, f \rangle$ is independent of~$n$ or~$\epsilon$, this and \eqref{E:3.8a} ensure that the limits
\begin{equation}
\lim_{N\to\infty}E \texte^{-\langle\eta_{N,r_N}^D \,,\, f \rangle}\quad\text{and}\quad \lim_{n\to\infty}\,\lim_{\epsilon\downarrow0}\,\lim_{N\to\infty}
E \texte^{-\langle\eta_{N,r_N'}^{D^n} \,,\, f^{D,D^n}_{\epsilon,\Phi}\rangle}
\end{equation}
exist and are equal. It thus suffices to pass to the limit $\epsilon\downarrow0$ followed by~$n\to\infty$ on the right-hand side of \eqref{E:3.9b}.

Plugging in the explicit form of $f^{D,D^n}_{\epsilon,\Phi}$ and substituting~$h$ for~$h+\Phi^{D,D^n}(x)$ yields
\begin{equation}
\text{r.h.s.\ of \eqref{E:3.9b}}=
E\biggl(\exp\Bigl\{-\int(1-\texte^{-f(x,h)\chi_{n,\epsilon}(x)})\,\texte^{\alpha\Phi^{D,D^n}(x)}\,Z^{D^n}(\textd x)\texte^{-\alpha h}\textd h\Bigr\}\biggr).
\end{equation}
Since $\chi_{n,\epsilon}$ increases pointwise to the indicator of~$D^n$ as $\epsilon\downarrow0$, the fact that~$Z^{D^n}$ puts all mass on~$D^n$ allows us to take~$\epsilon\downarrow0$ and get, with the help of the Monotone Convergence Theorem in the exponent and the Bounded Convergence Theorem overall, that
\begin{equation}
\label{E:3.12b}
\lim_{N\to\infty}E \texte^{-\langle\eta_{N,r_N}^D \,,\, f \rangle}
=\lim_{n\to\infty}
E\biggl(\exp\Bigl\{-\int(1-\texte^{-f(x,h)})\,\texte^{\alpha\Phi^{D,D^n}(x)}Z^{D^n}(\textd x)\texte^{-\alpha h}\textd h\Bigr\}\biggr).
\end{equation}
But the processes $\{\eta^D_{N,r_N}\}$ are tight as measures on~$\overline D\times\R$ thanks to Proposition~\ref{lem:4}. So taking appropriate~$f$'s with small supremum norm shows that the random variables
\begin{equation}
\Bigl\{\int Z^{D^n}(\textd x)\texte^{\alpha\Phi^{D,D^n}(x)}\colon n\ge1\Bigr\}
\end{equation}
are tight. We can thus find a sequence $n_k\to\infty$ such that $Z^{D_{n_k}}(\textd x)\texte^{\alpha\Phi^{D,D^{n_k}}(x)}$ converge weakly to a random measure~$Z^D(\textd x)$ concentrated on~$\overline D$. This permits passing the limit~$n\to\infty$ inside in \eqref{E:3.12b} and thus proving the claim.
\end{proofsect}

\subsection{Proofs of Theorems~\ref{thm:1} and~\ref{thm:2}}
We can now prove convergence to the Cox process as stated in Theorem~\ref{thm:1} and the Gibbs-Markov in Theorem~\ref{thm:2}. This is still conditional on the validity of Proposition~\ref{prop-3.1}, whose proof is deferred to the next subsection.

\begin{proofsect}{Proof of Theorem~\ref{thm:1}, convergence to Cox process}
We begin by recalling the conclusion of Theorem~1.1 of~\cite{Biskup-Louidor}. This theorem was stated for~$D$ being a unit square but simple scaling arguments allow us to extend it to squares of rational sizes.  Let~$K\in(0,\infty)\cap\Q$ and define $S_K:=(0,K)^2$. Then for any sequence~$r_N$ as given,
\begin{equation}
\label{e:2.30}
	\eta^{S_K}_{N,r_N}\,\underset{N\to\infty}\Lawlongarrow\, \text{\rm PPP}\bigl(Z^{S_K}(\textd x)\otimes\texte^{-\alpha h}\textd h\bigr),
\end{equation}
where $Z^{S_K}$ is a random measure obtained from that for the unit square by
\begin{equation}
\label{E:4.15t}
Z^{S_K}(\textd x)\laweq K^{4}\,Z^{S_1}(K^{-1}\textd x).
\end{equation}
Here the exponent~$4$ arises from the observation that $m_{KN}-m_N=2\sqrt{g}\log K+o(1)$ as $N\to\infty$ combined with $2\sqrt{g}\,\alpha=4$.

Now consider any $D\in\mathfrak D$. For any integer~$n\ge1$, tile~$\R^2$ by disjoint translates of $(0,2^{-n})^2$ by vectors from~$(2^{-n}\Z)^2$ and let $x_1,\dots,x_{\ell(n)}\in (2^{-n}\Z)^2$ enumerate the lower-left corners of those squares that are entirely contained in~$D$. Define
\begin{equation}
D^n:=\bigcup_{i=1}^{\ell(n)}\bigl(x_i+S^n\bigr),\qquad\text{where}\quad S^n:=(2^{-2n},2^{-n}-2^{-2n})^2.
\end{equation}
Since~$D$ is open and bounded, $D^n\in\mathfrak D$ and the assumptions \eqref{E:3.1} hold.

Since $D^n$ is a union of disjoint squares separated by distance at least~$2^{-2n}$, as soon as~$N>r2^{2n}$ holds, $\eta_{N,r}^{D^n}$ is the sum of $\ell(n)$ independent shifts of the process in $S^n$. By \eqref{e:2.30} we have \eqref{E:3.5a} for any sequence $r_N'$ such that $r_N'\to\infty$ but $r_N'/N\to0$ with $Z^{D^n}$ given by
as the sum of independent copies of~$Z^{S^n}$ translated by $x_1,\dots,x_{\ell(n)}$.
The convergence of $\eta_{N,r_N}^D$ to a Cox process with intensity $Z^D(\textd x)\otimes\texte^{-\alpha h}\textd h$ then follows from Proposition~\ref{prop-3.2}. 
\end{proofsect}

The above argument also shows \eqref{e:1.5}:

\begin{proofsect}{Proof of Theorem~\ref{thm:1}, disjoint partitions}
Suppose that $D$ is the disjoint union $D_1\cup\dots\cup D_m$. The approximation by squares then naturally splits into~$m$ independent processes in $D_1^n,\dots,D_m^n$. This gives $Z^{D^n}=Z^{D_1^n}+\dots+Z^{D_m^n}$
where the measures on the right are regarded as independent. Similarly,~$\Phi^{D,D^n}$ restricted to $D_i^n$ has the law of $\Phi^{D,D_i^n}$, with the fields in the disjoint components independent. Hence we get
\begin{equation}
Z^{D^n}(\textd x)\texte^{\alpha\Phi^{D,D^n}(x)}\laweq Z^{D_1^n}(\textd x)\texte^{\alpha\Phi^{D,D_1^n}(x)}+\dots+Z^{D_m^n}(\textd x)\texte^{\alpha\Phi^{D,D_m^n}(x)},
\end{equation}
with the measures on the right independent and concentrated on disjoint sets. Applying \eqref{E:3.7b} on both sides, the decomposition \eqref{e:1.5} follows.
\end{proofsect}

Next we check that the~$Z^D$ measure is stochastically absolutely continuous with respect to the Lebesgue measure and, in particular, is thus concentrated on~$D$. 

\begin{proofsect}{Proof of Theorem~\ref{thm:1}, stochastic absolute continuity}
Fix a Borel set~$A\subset\overline D$ with $\leb(A)=0$, let $A_n:=\{x\in\overline D\colon\dist(x,A)<2^{-n}\}$ and consider a sequence~$f_n\colon\overline D\to[0,1]$ of continuous functions supported in $A_n\times[1,4]$ such that~$f_n=1$ on~$A\times[2,3]$ and such that $f_n\downarrow0$ on the complement thereof. By the convergence we just proved and the Monotone Convergence Theorem,  $\langle \eta^D_{N,r_N},f_n\rangle$ tends in law to a Poisson random variable with (random) parameter $\alpha^{-1}\texte^{-2\alpha}(1-\texte^{-\alpha})Z^D(A)$ in the limit $N\to\infty$ followed by~$n\to\infty$. But Proposition~\ref{lem:6} tells us
\begin{equation}
\label{E:4.17t}
P\bigl(\langle \eta^D_{N,r_N},f_n\rangle>0\bigr)
\le P\bigl(\exists x\in NA_n\cap\Z^2\colon h^D_N(x)\ge m_N+1\bigr)
\le c\Bigl(\frac{|NA_n\cap\Z^2|}{N^2}\Bigr)^{1/2}\texte^{-\alpha}
\end{equation}
and, since $A_n$ is open, $|NA_n\cap\Z^2|/N^2\to\leb(A_n)$ as~$N\to\infty$. As $\leb(A_n)\downarrow\leb(A)=0$, this vanishes in the limits~$N\to\infty$ and~$n\to\infty$, and so $\langle \eta^D_{N,r_N},f_n\rangle$ in fact tends to zero in probability. It follows that $Z^D(A)=0$ a.s.\ and, in particular, $Z^D(\partial D)=0$ a.s.
\end{proofsect}

Proposition~\ref{prop-3.2} now helps us to get:

\begin{proofsect}{Proof of Corollary~\ref{cor-shift-scale}}
Shifts~$a\in\C$ and dilations~$\lambda>0$ taking integer values can directly be implemented on the lattice and so there \eqref{E:2.6t} follows immediately from the same calculation as that which led to \eqref{E:4.15t}. By the existence of the~$N\to\infty$ limit, this readily extends to all rational-valued shifts and dilations. 

To cover all shifts $a\in\C$ and dilations $\lambda>0$, define $D^n:=\{x\in D\colon\dist(x,D^\cc)>2^{-n}\}$ and note that~$D^n\uparrow D$. Since small shifts and dilations of~$D^n$ still lie in~$D^{n+1}$, one can find rational~$a_n\in\C$ and $\lambda_n>0$ so that $a_n\to a$, $\lambda_n\to\lambda$ and $a_n+\lambda_n D_n\uparrow a+\lambda D$. Fix an open set $A$ with $\overline A\subset D$. Proposition~\ref{prop-inner-apprx} shows that $\max_{x\in A}|\Phi^{D,D^n}(x)|\to0$ in probability and thus
\begin{equation}
\1_A(x) Z^{D^n}(\textd x)\,\underset{n\to\infty}\Lawlongarrow\,\1_A(x)Z^D(\textd x)
\end{equation}
by Proposition~\ref{prop-3.2}. Similarly we get
\begin{equation}
\1_A(a_n+\lambda_nx) Z^{a_n+\lambda_nD^n}(a_n+\lambda_n\textd x)\,\underset{n\to\infty}\Lawlongarrow\,\1_A(a+\lambda x)Z^{a+\lambda D}(a+\lambda\textd x).
\end{equation}
Since the laws of the measures on the left are related by a rational shift and dilation, we already know that they are~$\lambda_n^4$ multiples of each other (in law). From~$\lambda_n\to\lambda$ we then infer
\begin{equation}
\1_A(a+\lambda x)Z^{a+\lambda D}(a+\lambda\textd x)\,\laweq\,\lambda^41_A(x)Z^D(\textd x).
\end{equation}
Taking $A\uparrow D$ and using that $Z^D$ is concentrated on~$D$ a.s. then yields the claim.
\end{proofsect}

Invoking Proposition~\ref{prop-3.2}, we also conclude the Gibbs-Markov property:

\begin{proofsect}{Proof of Theorem~\ref{thm:3}}
Let~$D,\wt D\in\mathfrak D$ with $\leb(D\smallsetminus\wt D)=0$. Set~$D^n:=\wt D$ for all $n\ge1$ and note that \eqref{E:3.1} trivially applies. (Note that we had to first prove concentration of~$Z^{\wt D}$ on~$\wt D$ to make the measure on the left of \eqref{E:3.7b} well defined.) By Theorem~\ref{thm:1}, we know that \eqref{E:3.5a} holds for any sequence~$r_N'$ with the stated properties. Using Proposition~\ref{prop-3.2}, \eqref{E:3.7b} then trivializes as \eqref{e:1.9a}.
\end{proofsect}

Finally, we check all of the remaining properties of the $Z^D$-measures stated in Theorem~\ref{thm:1}:

\begin{proofsect}{Proof of Theorem~\ref{thm:1}, remaining properties of~$Z^D$-measure}
Almost-sure non-atomicity of~$Z^D$ and the fact that, a.s., $Z^D(A)>0$ for any non-empty open $A\subseteq D$ are known for $D:=(0,1)^2$ from Theorem~1.3 of Biskup and Louidor~\cite{Biskup-Louidor}. Invoking the Gibbs-Markov property, this extends to any superset $D\in\mathfrak D$ thereof. The scale and dilation invariance from Corollary~\ref{cor-shift-scale} then extends this to all~$D\in\mathfrak D$.
\end{proofsect}

\subsection{Proof of main reduction step}
For the proofs of Theorems~\ref{thm:1} and \ref{thm:3} to be really complete, we need to provide a proof of Proposition~\ref{prop-3.1}. This will be done via a sequence of lemmas. Throughout, we will think of~$h^D_N$ as defined by $h^D_N:=h^{D^n}_N+\varphi_N^{D,D^n}$, where $h_N^{D^n}$ and $\varphi_N^{D,D^n}$ are independent. Let~$\Theta_{N,r}^D$ denote the $r$-local maxima of $h_N^D$ in~$D_N$ and let $\Theta_{N,r}^{D^n}$ denote the $r$-local maxima of~$h_N^{D^n}$ in~$D^n_N$. Note that
\begin{equation}
\label{E:3.4}
\langle\eta_{N,r}^D,f\rangle=\sum_{x\in\Theta_{N,r}^D}f\bigl(x/N,\,h_N^D(x)-m_N\bigr).
\end{equation}
The argument begins by replacing~$r_N$ by an independent variable~$r$:

\begin{lemma}
\label{lemma-3.2}
Let~$r_N$ be such that $r_N\to\infty$ with $r_N/N\to0$. Then for any continuous function $f\colon\overline D\times\R\to[0,\infty)$ with compact support,
\begin{equation}
\lim_{r\to\infty}\,\,\limsup_{N\to\infty}P\bigl(\langle\eta_{N,r_N}^D,f\rangle\ne\langle\eta_{N,r}^D,f\rangle\bigr)=0.
\end{equation}
\end{lemma}

\begin{proofsect}{Proof}
We follow the proof of Lemma~4.4 in Biskup and Louidor~\cite{Biskup-Louidor}: Assume $r_N\ge r$ and note that $\langle\eta_{N,r_N}^D,f\rangle\ne\langle\eta_{N,r}^D,f\rangle$ implies the existence of a point in $\Theta_{N,r}^D\smallsetminus\Theta_{N,r_N}^D$ where $f(x/N,h_N^D-m_N)>0$. If~$f$ is supported on~$\overline D\times[-t,\infty)$, we thus have a point in $\Gamma_N^D(t)\cap(\Theta_{N,r}^D\smallsetminus\Theta_{N,r_N}^D)$. But that implies existence of two points $x,y\in\Gamma_N^D(t)$ such that $r\le|x-y|\le r_N$ which, since $r_N\le N/r$ for~$N$ large, has probability tending to zero in the stated limits by Proposition~\ref{lem:5}.
\end{proofsect}

Next we will insert the mollifier~$\chi_{n,\epsilon}$ next to~$f$ inside expectation:

\begin{lemma}
\label{lemma-3.3}
For any continuous function $f\colon\overline D\times\R\to[0,\infty)$ with compact support,
\begin{equation}
\lim_{n\to\infty}\,\limsup_{\epsilon\downarrow0}\,\limsup_{r\to\infty}\,\limsup_{N\to\infty}\,\Bigl|\,E \texte^{-\langle\eta_{N,r}^D \,,\, f \rangle}-E \texte^{-\langle\eta_{N,r}^D \,,\, f\chi_{n,\epsilon}\rangle}\Bigr|=0.
\end{equation}
\end{lemma}

\begin{proofsect}{Proof}
Suppose $f$ is supported in $\overline D\times [-t,\infty)$. Then
\begin{equation}
\label{E:3.6}
\texte^{-\langle\eta_{N,r}^D \,,\, f \rangle}-\texte^{-\langle\eta_{N,r}^D \,,\, f\chi_{n,\epsilon}\rangle}
=\texte^{-\langle\eta_{N,r}^D \,,\, f \chi_{n,\epsilon}\rangle}\bigl(1-\texte^{-\langle\eta_{N,r}^D \,,\, f(1-\chi_{n,\epsilon})\rangle}\bigr).
\end{equation}
By the restriction on the support of~$f$, we have $\langle\eta_{N,r}^D \,,\, f(1-\chi_{n,\epsilon})\rangle\ne0$ if and only if there is a point in $\Gamma_N^D(t)\smallsetminus D_{N,2\epsilon}^n$, where (abusing out earlier notations) we write
\begin{equation}
\label{E:3.27c}
D^n_{N,\delta}:=ND_\delta^n\cap\Z^2\quad\text{for}\quad
D^n_{\delta}:=\bigl\{x\in D^n\colon\textd(x,(D^n)^\cc)>\delta\bigr\}.
\end{equation} 
>From Proposition~\ref{lem:6} we thus get
\begin{equation}
\label{E:3.7}
P\bigl(\langle\eta_{N,r}^D \,,\, f(1-\chi_{n,\epsilon})\rangle\ne0\bigr)\le c\,\frac{\,\,|D_N\smallsetminus D_{N,2\epsilon}^n|^{1/2}\!\!}N.
\end{equation}
For large-enough~$N$, for every~$x\in D_N\smallsetminus D_{N,2\epsilon}^n$ the set $N(D\smallsetminus D^n_{3\epsilon})$ will contain the open square of  side-length one centered at~$x$. It follows that $|D_N\smallsetminus D_{N,2\epsilon}^n|/N^2\le \leb(D\smallsetminus D^n_{3\epsilon})$. The assumptions on~$D^n$ ensure that $\leb(D\smallsetminus D^n_{3\epsilon})$, and thus also the right-hand side of \eqref{E:3.7}, tend to zero as~$\epsilon\downarrow0$ followed by~$n\to\infty$. In conjunction with \eqref{E:3.6}, this proves the claim.
\end{proofsect}

Our next step will consist of application of the Gibbs-Markov property. 

\begin{lemma}
\label{lemma-3.4}
For any $n\ge1$, any~$\epsilon>0$ and any continuous function $f\colon\overline D\times\R\to[0,\infty)$ with compact support,
\begin{equation}
\label{E:3.10}
\lim_{r\to\infty}\,\limsup_{N\to\infty}\,\Bigl|\,E \texte^{-\langle\eta_{N,r}^D \,,\, f \chi_{n,\epsilon}\rangle}-E \texte^{-\langle\eta_{N,r}^{D^n} \,,\, f^{D,D^n}_{\epsilon,\Phi}\rangle}\Bigr|=0.
\end{equation}
\end{lemma}

Before we can start the formal proof, we need to make  some preparatory steps. Consider the finite-$N$ version of the function in \eqref{E:3.3}:
\begin{equation}
\label{E:3.11}
f_{N,\epsilon,\varphi}^{D,D^n}(x,h):=f\bigl(x,\,h+\varphi_N^{D, D^n}(\lfloor Nx\rfloor)\bigr)\chi_{n,\epsilon}(x),\qquad (x,h)\in D^n\times\R.
\end{equation}
Then
\begin{multline}
\label{E:3.12a}
\qquad
\Bigl|\,E \texte^{-\langle\eta_{N,r}^D \,,\, f \chi_{n,\epsilon}\rangle}-E \texte^{-\langle\eta_{N,r}^{D^n} \,,\, f^{D,D^n}_{\epsilon,\Phi}\rangle}\Bigr|
\\
\le
\Bigl|\,E \texte^{-\langle\eta_{N,r}^D \,,\, f \chi_{n,\epsilon}\rangle}-E \texte^{-\langle\eta_{N,r}^{D^n} \,,\, f^{D,D^n}_{N,\epsilon,\varphi}\rangle}\Bigr|+\Bigl|\,E \texte^{-\langle\eta_{N,r}^{D^n} \,,\, f^{D,D^n}_{N,\epsilon,\varphi}\rangle}-E \texte^{-\langle\eta_{N,r}^{D^n} \,,\, f^{D,D^n}_{\epsilon,\Phi}\rangle}\Bigr|.
\end{multline}
Focusing on the first term on the right, by \eqref{E:3.4}, \eqref{E:3.11} and Jensen's inequality,
\begin{equation}
\label{E:3.17}
\begin{aligned}
\Bigl|\,E &\texte^{-\langle\eta_{N,r}^D \,,\, f \chi_{n,\epsilon}\rangle}-E \texte^{-\langle\eta_{N,r}^{D^n} \,,\, f^{D,D^n}_{N,\epsilon,\varphi}\rangle}\Bigr|\\
&\qquad\le E\Bigl|\,\prod_{x\in \Theta_{N,r}^D}\texte^{-f\left(x/N,\,h^D_N(x)-m_N\right)\chi_{n,\epsilon}(x/N)}-\prod_{x\in \Theta_{N,r}^{D^n}}\texte^{-f\left(x/N,\,h^D_N(x)-m_N\right)\chi_{n,\epsilon}(x/N)}\Bigr|.
\end{aligned}
\end{equation}
Our strategy is to identify a set of asymptotically full measure on which we can derive a uniform estimate on the term in absolute value. That will in turn require establishing a one-to-one correspondence between the terms in the two products. This will then yield a good bound on the first term on the right of \eqref{E:3.12a} as well.

We begin by identifying some exceptional events.
Let the function~$f$ be as given and let~$t_0>0$ be such that $f$ is supported on~$\overline D\cap[-t_0,\infty)$. Let~$n\ge1$ and~$\epsilon>0$ be fixed. Recall the set $D^n_{N,\delta}$ from \eqref{E:3.27c}.
For~$t>0$ and $M>0$ let
\begin{equation}
\label{E:3.14b}
A_{N,t}^{(0)}:=\bigl\{\max_{x\in D^n_{N,\epsilon/2}}|\varphi_N^{D,D^n}(x)|\ge t\bigr\}
\quad\text{and}\quad
A_{N,M,t}^{(1)}:=\bigl\{|\Gamma_N^{D^n}(t_0+2t)|>M\bigr\}.
\end{equation}
Define also
\begin{equation}
A_{N,r,t}^{(2)}:=\bigl\{\exists x,y\in\Gamma_N^{D^n}(t_0+2t)\colon r/2<|x-y|<N/r\bigr\}
\end{equation}
and, recalling that $\Lambda_r(x):=\{y\in\Z^2\colon |x-y|\le r\}$, set
\begin{equation}
A_{N,r,t}^{(3)}:=\bigl\{\exists x\in\Gamma_N^{D^n}(t_0+2t)\cap D^n_{N,\epsilon/2},\,\exists y\in \Lambda_{2r}(x)\,\colon|\varphi_N^{D,D^n}(x)-\varphi_N^{D,D^n}(y)|>N^{-1/3}\bigr\}.
\end{equation}
Abbreviate $A_{N,M,r,t}:=A_{N,t}^{(0)}\cup A_{N,M,t}^{(1)}\cup A_{N,r,t}^{(2)}\cup A_{N,r,t}^{(3)}$. 

\begin{lemma}
\label{lemma-3.5}
For~$f$ and~$t_0$ as above,
\begin{equation}
\lim_{t\to\infty}\,\,\limsup_{r\to\infty}\,\,\limsup_{M\to\infty}\,\limsup_{N\to\infty} \,P(A_{N,M,r,t})=0.
\end{equation}
\end{lemma}

\begin{proofsect}{Proof}
By Propositions~\ref{lem:3},~\ref{lem:4}, and~\ref{lem:5}, we immediately have
\begin{equation}
\label{E:3.16a}
\lim_{t\to\infty}\,\,\limsup_{r\to\infty}\,\,\limsup_{M\to\infty}\,\limsup_{N\to\infty}
\,P\bigl(A_{N,t}^{(0)}\cup A_{N,M,t}^{(1)}\cup A_{N,r,t}^{(2)}\bigr)=0.
\end{equation}
Concerning the remaining event we note that, thanks to discrete harmonicity of $C^{D,D^n}$, the (discrete) Harnack inequality implies for any $r\ge1$ once $N$ is sufficiently large,
\begin{equation}
\max_{\begin{subarray}{c}
x\in D^n_{N,\epsilon/2}\\ y\in\Lambda_{2r}(x)
\end{subarray}}
\text{\rm Var}\bigl(\varphi_N^{D,D^n}(x)-\varphi_N^{D,D^n}(y)\bigr)\le c(\epsilon)r/N
\end{equation}
for some~$c(\epsilon)<\infty$.
By Proposition~\ref{lem:2}, the independence of $\varphi_N^{D,D^n}$ of~$h^B_N$, the fact that $A_{N,M,t}^{(1)}$ depends only on~$h^B_N$, the Markov inequality and a union bound yield
\begin{equation}
\lim_{N\to\infty}P\bigl(A_{N,r,t}^{(3)}\big|(A_{N,M,t}^{(1)})^\cc\bigr)=0,
\end{equation}
for any $M$, $r$ and~$t$. Combining this with \eqref{E:3.16a}, the claim follows.
\end{proofsect}

Next we develop the correspondence between the relevant local maxima of~$\eta_{N,r}^D$ and $\eta_{N,r}^{D^n}$. Thanks to our assumption on the support of~$f$, only points in~$\Gamma_N^D(t_0)\cap D^n_{N,\epsilon/2}$ can possibly contribute to the products in \eqref{E:3.17}. Hence it suffices to show:

\begin{lemma}
\label{lemma-3.6}
Let~$M>0$, $r\ge1$ and $t>1$ be fixed. Then for all~$N$ large, on the event~$(A_{N,M,r,t})^\cc$, there is a map~$q\colon\Theta_{N,r}^{D^n}\cap\Gamma_N^D(t_0+1)\cap D^n_{N,\epsilon/2}\to\Theta_{N,r}^D$ such that
\begin{enumerate}
\item[(1)] $q$ is injective,
\item[(2)] $q(\Theta_{N,r}^{D^n}\cap\Gamma_N^D(t_0+1)\cap D^n_{N,\epsilon/2})\supseteq\Theta_{N,r}^D\cap\Gamma_N^D(t_0)\cap D^n_{N,\epsilon}$,
\end{enumerate}
and, for all $x$ for which~$q(x)$ is defined,
\begin{enumerate}
\item[(3)] $|q(x)-x|\le r$,
\item[(4)] $|h_N^D(q(x))-h_N^D(x)|\le N^{-1/3}$.
\end{enumerate}
\end{lemma}

\begin{proofsect}{Proof}
We will need~$N$ so large that $N/r>2r$, $\epsilon N>r$ and $N^{-1/3}<t$. Let us assume that~$(A_{N,M,r,t})^\cc$ occurs. We claim that then
\begin{equation}
\label{E:3.18}
x\in\Gamma_N^D(t_0+1)\cap D^n_{N,\epsilon/2}
\quad\Rightarrow\quad 
\begin{cases}
\Lambda_{r/2}(x)\cap\Theta_{N,r}^D\ne\emptyset,
\\
\Lambda_{r/2}(x)\cap\Theta_{N,r}^{D^n}\ne\emptyset.
\end{cases}
\end{equation}
Indeed, the containment in~$(A_{N,M,r,t})^\cc$ ensures the following:
\begin{enumerate}
\item[(a)] $h_N^D(x)\ge m_N-t_0-1$ and so $h_N^{D^n}(x)\ge m_N-(t_0+t+1)$,
\item[(b)] $h_N^D(y)<m_N-(t_0+t)$ for $y$ such that $r/2<|x-y|<N/r$,
\item[(c)] $|\varphi_N^{D,D^n}(y)-\varphi_N^{D,D^n}(x)|\le N^{-1/3}$ for $y\in\Lambda_{2r}(x)$.
\end{enumerate}
Then (a) and (b) imply that $h_N^D(y)\le h_N^D(x)-t+1<h_N^D(x)$ for $y\in\Lambda_{2r}(x)\smallsetminus\Lambda_{r/2}(x)$ and so the maximizer of $h_N^D$ on~$\Lambda_{2r}(x)$ lies in~$\Lambda_{r/2}(x)$. This gives the top line in \eqref{E:3.18}; for the bottom line we instead note that (a-c) give
\begin{equation}
h_N^{D^n}(y)\le h_N^{D^n}(x)+1-t+N^{-1/3}<h_N^{D^n}(x),\qquad y\in\Lambda_{2r}(x)\smallsetminus\Lambda_{r/2}(x),
\end{equation}
and the same conclusion holds for~$h_N^{D^n}$ as well. 

By the definition of $r$-local maximum, the sets on the right of \eqref{E:3.18} are singletons almost surely. Moreover, if~$x\in\Gamma_N^D(t_0+1)\cap D^n_{N,\epsilon/2}$ also lies in $\Theta_{N,r}^{D^n}$, then~$x$ is the (a.s.) unique element of $\Lambda_{r/2}(x)\cap\Theta_{N,r}^{D^n}$. We then define~$q(x)$ to be the (a.s.) unique element of $\Lambda_{r/2}(x)\cap\Theta_{N,r}^D$. This implies (3) in the statement and shows~(1) as well. To get~(4), we note that $h_N^D(q(x))-h_N^D(x)$ and $h_N^{D^n}(q(x))-h_N^{D^n}(x)$ differ by at most~$N^{-1/3}$ but, by the definition of~$r$-local maximum and~(3), they have opposite signs. So the bound in~(4) must hold. To get also~(2) we note that \eqref{E:3.18} ensures the existence of a point $x\in\Lambda_{r/2}(y)\cap\Theta_{N,r}^{D^n}$ for every $y\in\Theta_{N,r}^D\cap\Gamma_N^D(t_0)\cap D^n_{N,\epsilon}$. Since~$\epsilon N>r$ and (4) holds, this~$x$ lies in $\Theta_{N,r}^{D^n}\cap\Gamma_N^D(t_0+1)\cap D^n_{N,\epsilon/2}$ and so $q(x)$ is defined and~$y=q(x)$ holds.
\end{proofsect}

We are now finally ready to prove the claim in~\eqref{E:3.10}:
 
\begin{proofsect}{Proof of Lemma~\ref{lemma-3.4}}
Recall the calculation in \twoeqref{E:3.12a}{E:3.17}. We will focus on the expression in \eqref{E:3.17} first. Abbreviate
\begin{equation}
\Sigma_{N,r,\epsilon}^{D^n}:=\Theta_{N,r}^{D^n}\cap\Gamma_N^D(t_0+1)\cap D^n_{N,\epsilon/2}.
\end{equation}
Thanks to the restriction on the support of~$f$ and the presence of~$\chi_{n,\epsilon}$, we can freely restrict the second product to~$x\in\Sigma_{N,r,\epsilon}^{D^n}$. By Lemma~\ref{lemma-3.6}, once~$N$ is large enough and $(A_{N,M,r,t})^\cc$ occurs, all terms that effectively contribute to the first product lie in~$q(\Sigma_{N,r,\epsilon}^{D^n})$. Since~$q$ is also injective, we can write the first product over $q(\Sigma_{N,r,\epsilon}^{D^n})$ and use the telescoping trick to write, on~$(A_{N,M,r,t})^\cc$,
\begin{equation}
\begin{aligned}
\Bigl|\,\prod_{x\in \Theta_{N,r}^D}&\texte^{-f\left(x/N,\,h^D_N(x)-m_N\right)\chi_{n,\epsilon}(x/N)}-\prod_{x\in \Theta_{N,r}^{D^n}}\texte^{-f\left(x/N,\,h^D_N(x)-m_N\right)\chi_{n,\epsilon}(x/N)}\Bigr|
\\
&\qquad\le
\sum_{x\in\Sigma_{N,r,\epsilon}^{D^n}}\Bigl|\texte^{-f\left(q(x)/N,\,h^D_N(q(x))-m_N\right)\chi_{n,\epsilon}(q(x)/N)}-\texte^{-f\left(x/N,\,h^D_N(x)-m_N\right)\chi_{n,\epsilon}(x/N)}\Bigr|.
\end{aligned}
\end{equation}
Let $\text{osc}_{f\chi_{n,\epsilon}}(\delta)$ denote the maximal oscillation of~$f(x,h)\chi_{n,\epsilon}(x)$ as either argument varies over an interval of length~$\delta$. On $(A_{N,M,r,t})^\cc$ (and assuming~$t\ge1$) we have $|\Sigma_{N,r,\epsilon}^{D^n}|\le M$ and, using parts (3-4) of~Lemma~\ref{lemma-3.6}, the sum is at most $M \,\text{osc}_{f\chi_{n,\epsilon}}(r/N+N^{-1/3})$. This tends to zero as~$N\to\infty$; the expectation in \eqref{E:3.17} then tends to zero as~$N\to\infty$ and~$r\to\infty$ by Lemma~\ref{lemma-3.5}.

It remains to address the second term in \eqref{E:3.12a}. Applying the definitions \eqref{E:3.4}, \eqref{E:3.3}, \eqref{E:3.11} we can rewrite this term as
\begin{equation}
\begin{aligned}
\Bigl|\,E &\texte^{-\langle\eta_{N,r}^{D^n} \,,\, f^{D,D^n}_{N,\epsilon,\varphi}\rangle}-E \texte^{-\langle\eta_{N,r}^{D^n} \,,\, f^{D,D^n}_{\epsilon,\Phi}\rangle}\Bigr|
\\
&\qquad\qquad= 
\biggl|\,E\prod_{x\in \Theta_{N,r}^{D^n}}\texte^{-f\left(x/N,\,h^{D^n}_N(x)-m_N+\varphi_N^{D,D^n}(x)\right)\chi_{n,\epsilon}(x/N)}
\\*[-3mm]
&\qquad\qquad\qquad\qquad\qquad\qquad-E\prod_{x\in \Theta_{N,r}^{D^n}}\texte^{-f\left(x/N,\,h^{D^n}_N(x)-m_N+\Phi^{D,D^n}(x/N)\right)\chi_{n,\epsilon}(x/N)}\biggr|
\end{aligned}
\end{equation}
Now recall the event $A_{N,t}^{(0)}$ from \eqref{E:3.14b}. Due to our restriction on the support of~$f$, on $(A_{N,t}^{(0)})^\cc$ we can restrict the first product to $x\in\Theta_{N,r}^{D^n}\cap\Gamma^{D^n}(t_0+t)$. Similarly, on the complement of
\begin{equation}
\label{E:3.26}
\tilde A_{N,t}^{(0)}:=\bigl\{\max_{x\in D^n_{N,\epsilon/2}}\Phi^{D,D^n}(x/N)>t\bigr\}
\end{equation}
the second product can also be restricted to $x\in\Theta_{N,r}^{D^n}\cap\Gamma_N^{D^n}(t_0+t)$ as well. Introducing  an additional restriction on the size of $\Gamma_N^{D^n}(t_0+t)$, we use the telescopic trick again to get
\begin{multline}
\label{E:3.27}
\Bigl|\,E \texte^{-\langle\eta_{N,r}^{D^n} \,,\, f^{D,D^n}_{N,\epsilon,\varphi}\rangle}-E \texte^{-\langle\eta_{N,r}^{D^n} \,,\, f^{D,D^n}_{\epsilon,\Phi}\rangle}\Bigr|\le P\bigl(A_{N,t}^{(0)}\bigr)+P\bigl(\tilde A_{N,t}^{(0)}\bigr)+P\bigl(|\Gamma_N^{D^n}(t)|> M\bigr)
\\
\!\!\!\!+ E\biggl(\,\1_{\{|\Gamma_N^{D^n}(t_0+t)|\le M\}}
\sum_{x\in \Theta_{N,r}^{D^n}\cap\Gamma_N^{D^n}(t_0+t)}
\wt E\Bigl|\,\texte^{-f\left(x/N,\,h^{D^n}_N(x)-m_N+\varphi_N^{D,D^n}(x)\right)\chi_{n,\epsilon}(x/N)}
\\-
\texte^{-f\left(x/N,\,h^{D^n}_N(x)-m_N+\Phi^{D,D^n}(x/N)\right)\chi_{n,\epsilon}(x/N)}\Bigr|\biggr),
\end{multline}
where the inner expectation~$\wt E$ is over any coupling of the random variables
\begin{equation}
\bigl\{(\varphi_N^{D,D^n}(x),\Phi^{D,D^n}(x/N))\colon x\in \Theta_{N,r}^{D^n}\cap\Gamma_N^{D^n}(t_0+t)\bigr\}
\end{equation}
with the correct marginal laws, conditioned on~$h_N^{D^n}$. Since there are at most~$M$ random variables in this collection, for any~$\delta>0$, Proposition~\ref{lem:2} ensures the existence of a coupling such that $|\varphi_N^{D,D^n}(x)-\Phi^{D,D^n}(x/N)|<\delta$ for all~$x\in \Theta_{N,r}^{D^n}\cap\Gamma_N^{D^n}(t_0+t)\cap D_{N,\epsilon}^n$ with probability at least~$1-\delta$ once~$N$ is large. Hence, the $N\to\infty$ limit of the last term on the right of \eqref{E:3.27} is again at most $M(\delta+\text{osc}_{f\chi_{n,\epsilon}}(\delta))$ for any $\delta>0$. As the three probabilities on the right of \eqref{E:3.27} tend to zero as $N\to\infty$, $t\to\infty$ and $M\to\infty$, the claim is proved.
\end{proofsect}

\begin{proofsect}{Proof of Proposition~\ref{prop-3.1}}
The claim follows readily from Lemmas~\ref{lemma-3.2}, \ref{lemma-3.3} and \ref{lemma-3.4}.
\end{proofsect}


\section{Maximizer conditioned on large maximum}
\label{sec4}\noindent
The goal of this section is to establish Theorem~\ref{thm:1.4} dealing with the asymptotic distribution of the maximizer of~$h_N^D$ conditional on the maximum being atypically large. Here and henceforth, the term ``square'' designates a set of the form $(a,a+r)\times(b,b+r)$ with~$a,b\in\R$ and~$r>0$.

\subsection{Reduction to squares}
\label{sec5.1ab}\noindent
Our strategy is to deduce Theorem~\ref{thm:1.4} from the following statement (adapted to our notation) that was proved for square domains in Bramson, Ding and Zeitouni~\cite{BDingZ}:

\begin{proposition}[Proposition~2.2 of~\cite{BDingZ}]
\label{thm-asymp}
Consider the square $S:=(0,1)^2$ and recall that we have $\alpha:=2/\sqrt g=\sqrt{2\pi}$ in our normalization.
There exists a constant $a_\star\in(0,\infty)$ such that
\begin{equation}
\label{E:2.4q}
\lim_{t\to\infty}\,\limsup_{N\to\infty}\,
\biggl|
\frac1t\,\texte^{\alpha t}P\Bigl(\,\max_{x\in S_N}h_N^S(x)> m_N+t\Bigr) -
a_\star \biggr| = 0 \,.
\end{equation}
Moreover, there exists a probability density $\psi$ on $[0,1]^2$ such that for any open set $A\subset[0,1]^2$,
\begin{equation}
\label{E:2.5q}
\lim_{t\to\infty}\,\limsup_{N\to\infty}
\biggl|\, P\Bigl(N^{-1}\operatornamewithlimits{argmax}_{S_N}\,h_N^S\in A\,\Big|\,\max_{x\in S_N}h_N^S(x)> m_N+t\Bigr)
- \int_A\psi(x)\textd x \biggr| = 0 \,.
\end{equation}
\end{proposition}

Thanks to the results of Bramson, Ding and Zeitouni~\cite{BDingZ} and Biskup and Louidor~\cite{Biskup-Louidor}, we in fact know that the $N\to\infty$ limits in \twoeqref{E:2.4q}{E:2.5q} in fact exist. Moreover, the conclusions readily generalize to squares of all sizes. Explicitly, for~$S:=a+(0,K)^2$ with $K>0$ rational and~$a\in\Q^2$, reducing to a subsequence of~$N$'s for which $KN\in\N$ and noting that $m_{KN}=m_N+2\sqrt g\log K+o(1)$ as $N\to\infty$ yields (for any $A\subseteq D$ open) that
\begin{equation}
\label{E:4.3b}
\lim_{t\to\infty}\lim_{N\to\infty}\,\frac1t\texte^{\alpha t}\,P\Bigl(N^{-1}\operatornamewithlimits{argmax}_{S_N}\,h_N^S\in A\,,\,\,\max_{x\in S_N}h_N^S(x)> m_N+t\Bigr)
=\int_A\psi^S(x)\textd x,
\end{equation}
where $\psi^S(x):=a_\star K^{2}\psi((x-a)/K)$. Thus, \eqref{E:1.13a} holds for squares of rational sizes. Squares of irrational sizes are handled by straightforward approximations and Proposition~\ref{lem:6}. Note that that~$\psi^S(x)$ is invariant under a simultaneous shift of both~$x$ and~$S$.

Our goal is to extend \eqref{E:4.3b} to all~$D\in\mathfrak D$. More importantly, we also aim to establish the formula \eqref{E:1.14a} for the corresponding density~$\psi^D$. As it turns out, the key is to show that, under the coupling of the DGFF in two nested domains, the maximizer of the field in the larger domain conditioned on the maximum to be very large will (with high probability) coincide with the maximizer of the field in the subdomain. This is the content of the following~claim:

\begin{proposition}[Coincidence of maximizers]
\label{prop-4.2}
Let $D\in\mathfrak D$ and suppose that~$A$ is a non-empty open set and~$S,S'$ open squares such that $\overline A\subset S'$ and $\overline{S'}\subset S\subseteq D$. Regard $h_N^D$ as defined by $h_N^D:=h_N^S+\varphi_N^{D,S}$ with $h_N^S$ and~$\varphi_N^{D,S}$ independent and let $\hat{X}_N^D$, resp., $\hat{X}_N^S$ be the (a.s.-unique) maximizers of~$h_N^D$, resp., $h_N^S$ on~$D_N$, resp.,~$S_N$. Then
\begin{equation}
\label{E:4.4a}
\begin{aligned}
\lim_{M\to\infty}\limsup_{t \to \infty}\,\limsup_{N \to \infty}\,
\biggl|&P\Bigl( \hat{X}_N^ D \in A_N,\, \max_{x\in D_N}h_N^D > m_N + t\Bigr)
\\
&-	 P\Bigl( \hat{X}_N^S \in A_N,\, h_N^D(\hat{X}_N^S) > m_N + t,\,\max_{x\in S'_N}|\varphi_N^{D,S}(x)|\le M\Bigr)
\biggr|\,\frac 1t \texte^{\alpha t}  \ =\  0 \,.
\end{aligned}
\end{equation}
\end{proposition}

Postponing the proof of this proposition, we move to:

\begin{proofsect}{Proof of Theorem~\ref{thm:1.4}}
Let~$D\in\mathfrak D$.
We begin by showing the existence of a function~$\psi^D$ such that \eqref{E:1.13a} holds. In light of Proposition~\ref{lem:6} and the additivity in~$A$ of both sides of \eqref{E:1.13a}, it suffices to prove the claim for all open sets~$A\subseteq D$ so small that there are open squares $S,S'$ for which $\overline A\subset S'$ and $\overline{S'}\subset S\subseteq D$ hold. We will assume these sets to be fixed for the duration of the argument.

Consider a tiling of $\R^2$ by open squares of side~$1/K$ and let $S^i\colon i=1,\dots,n_K$ denote those squares entirely contained in~$A$. We assume that~$K$ is so large that $n_K\ge1$. Let $S_N^i$ denote the scaled up version of~$S^i$ by~$N$ and note that $S_N^i\cap S_N^j=\emptyset$ for $i\ne j$. Denote
\begin{equation}
\varphi_{N,i}^{\text{min}}:=\min_{z\in S^i_N}\varphi_N^{D,S}(z),\qquad i=1,\dots,n_K.
\end{equation}
Since the maximizer of~$h_N^S$ is a.s.\ unique, the Gibbs-Markov property shows
\begin{equation}
\begin{aligned}
P\Bigl( \hat{X}_N^S \in A_N,\, &h_N^D(\hat{X}_N^S) > m_N + t,\,\max_{x\in S'_N}|\varphi_N^{D,S}(x)|\le M\Bigr)
\\
&\ge\sum_{i=1}^{n_K}
P\Bigl( \hat{X}_N^S \in A_N\cap S^i_N,\, h_N^D(\hat{X}_N^S) > m_N + t,\,\max_{x\in S'_N}|\varphi_N^{D,S}(x)|\le M\Bigr)
\\
&\ge
\sum_{i=1}^{n_K}
P\Bigl( \hat{X}_N^S \in A_N\cap S^i_N,\, h_N^S(\hat{X}_N^S) > m_N + t-\varphi_{N,i}^{\text{min}},\,\max_{x\in S'_N}|\varphi_N^{D,S}(x)|\le M\Bigr).
\end{aligned}
\end{equation}
As $n_K$ does not depend on~$N$ or $t$, we may assume that the corresponding version of \eqref{E:4.3b} holds for all~$n_K$ squares uniformly. Moreover, the distributional convergence~$\varphi_N^{D,S}\Lawarrow\Phi^{D,S}$ takes place uniformly on~$S'$ and so we can couple $\varphi_N^{D,S}$ and~$\Phi^{D,S}$ on the same probability space such that $\max_{x\in S'_N}|\varphi_N^{D,S}(x)-\Phi^{D,S}(x/N)|=o(1)$ as~$N\to\infty$. Using these facts, we derive
\begin{equation}
\begin{aligned}
\liminf_{t\to\infty}\,\liminf_{N\to\infty}\,\frac1t\texte^{\alpha t}
P\Bigl( \hat{X}_N^S \in A_N,\, &h_N^D(\hat{X}_N^S) > m_N + t,\,\max_{x\in S'_N}|\varphi_N^{D,S}(x)|\le M\Bigr)
\\
&\ge\sum_{i=1}^{n_K}\Bigl(\int_{A\cap S^i}\psi^{S}(x)\textd x\Bigr)E\Bigl(\texte^{\alpha\Phi^{D,S}_i}\1_{\{\sup_{S'}|\Phi^{D,S}|< M\}}\Bigr),
\end{aligned}
\end{equation}
where $\Phi_i^{D,S}:=\inf_{x\in S^i}\Phi^{D,S}(x)$. Writing the right-hand side as one integral and using the Bounded Convergence Theorem and the  a.s.\ continuity of~$\Phi^{D,S}$ on~$S'$ to take~$K\to\infty$, we get  the integral of~$\psi^S(x)$ against the expectation of $\texte^{\alpha\Phi^{D,S}(x)}$ on the event where $\sup_{S'}|\Phi^{D,S}|<M$. In the limit~$M\to\infty$, the Monotone Convergence Theorem yields
\begin{equation}
\begin{aligned}
\liminf_{M\to\infty}\,\liminf_{t\to\infty}\,\liminf_{N\to\infty}\,\frac1t\texte^{\alpha t}\,
P\Bigl( \hat{X}_N^S \in A_N,\, &h_N^D(\hat{X}_N^S) > m_N + t,\,\max_{x\in S'_N}|\varphi_N^{D,S}(x)|\le M\Bigr)
\\
&\ge\int_{A}\psi^{S}(x)\texte^{\frac12\alpha^2\text{Var}(\Phi^{D,S}(x))}\textd x.
\end{aligned}
\end{equation}
A completely analogous argument using squares of side-length~$1/K$ that overlap slightly and cover~$A$ gives us, in conjunction with Proposition~\ref{lem:6} and the fact that~$\partial A$ has zero Lebesgue measure, an identical upper bound with \emph{limes superior} replacing \emph{limes inferior} in all three cases. Based on \eqref{E:4.4a} we then conclude \eqref{E:1.13a} with
\begin{equation}
\psi^D(x):=\psi^S(x)\texte^{\frac12\alpha^2\text{Var}(\Phi^{D,S}(x))},\qquad x\in S.
\end{equation}
It remains to identify~$\psi^D$ explicitly.

Pick~$D\in\mathfrak D$ and recall that \eqref{E:1.6} gives $\text{Var}(\Phi^{D,S}(x))$ as the difference $F^D(x)-F^S(x)$, where we have temporarily denoted
\begin{equation}
\label{E-F^D}
F^D(x):=g\int_{\partial  D}\Pi^{ D}(x,\textd z)\log|z-y|.
\end{equation}
It follows that
\begin{equation}
\label{E:4.11c}
\psi^D(x)\texte^{-\frac12\alpha^2g \,F^D(x)}
=\psi^S(x)\texte^{-\frac12\alpha^2g\, F^S(x)},\qquad x\in S.
\end{equation}
We claim that the quantity on the left is the same for all~$x\in D$. Indeed, if~$x,x'\in D $ are two points, choose~$S$ to be a square centered at~$x$ and~$S'$ a square of the same size centered at~$x'$ such that $S,S'\subseteq D$. The joint shift invariance in~$x$ and~$S$ of both~$\psi^S(x)$ and~$F^S(x)$ implies that the right-hand side of \eqref{E:4.11c} is the same for both~$x$ and~$x'$. Hence, so is the left-hand side. 

Writing~$c_\star$ for the common value of the quantity in \eqref{E:4.11c}, we infer
\begin{equation}
\label{E:4.12qq}
\psi^D(x)=c_\star\,\texte^{\frac12\alpha^2g \,F^D(x)},\qquad x\in D.
\end{equation}
Since $\alpha^2 g=4$ (regardless of the value of~$g$), this gives us \eqref{E:1.14a}. 
\end{proofsect}

We prove also the corresponding conclusion for the $Z^D$-measure:

\begin{proofsect}{Proof of Corollary~\ref{cor-Z-measure}}
Let~$A\subseteq D$ be open and let~$t$ be related to~$\lambda$ via $\lambda:=\frac1\alpha\texte^{-\alpha t}$. Let $\eta^D$ be the limit process in domain~$D$ and let $(x^\star,h^\star)$ be the point with the largest value of the field coordinate. Then, as observed in Biskup and Louidor~\cite{Biskup-Louidor}, 
\begin{equation}
P\bigl(\,x^\star\in A,\,h^\star\le t\bigr)=E\bigl(\,\widehat Z^D(A)(1-\texte^{-\lambda Z^D(D)})\bigr).
\end{equation}
Noting that $\lambda\log(\ffrac1\lambda)=t\texte^{-\alpha t}$, the asymptotic \eqref{E:1.22c} follows from Theorem~\ref{thm:1.4}, the distributional convergence in Theorem~\ref{thm:1} and some simple approximation arguments.

Concerning \eqref{E:1.23c}, we pick $\theta\in(0,1)$ and observe
\begin{equation}
E\bigl(\,\widehat Z^D(A)(\texte^{-\theta\lambda Z^D(D)}-\texte^{-\lambda Z^D(D)})\bigr)=\lambda \int_\theta^1E\bigl(\ Z^D(A)\texte^{-s\lambda Z^D(D)}\bigr)\textd s.
\end{equation}
Straightforward estimates of the exponent on the right hand side then yield
\begin{equation}
\begin{aligned}
\lambda(1-\theta)E\bigl(\,Z^D(A)&\texte^{-\lambda Z^D(D)}\bigr)\\
&\le
E\bigl(\,\widehat Z^D(A)(\texte^{-\theta\lambda Z^D(D)}-\texte^{-\lambda Z^D(D)})\bigr)
\\&\qquad\qquad\qquad\le\lambda(1-\theta)E\bigl(\,Z^D(A)\texte^{-\theta\lambda Z^D(D)}\bigr).
\end{aligned}
\end{equation}
Thanks to \eqref{E:1.22c}, in the limit as~$\lambda\downarrow0$ the middle expression is asymptotic to $(1-\theta)\lambda\log(\ffrac1\lambda)$ times the integral of~$\psi^D$ over~$A$. The claim  follows by taking~$\theta\downarrow0$.
\end{proofsect}

\subsection{Coincidence of maximizers}
We now move to the proof of Proposition~\ref{prop-4.2}.
We begin by noting that inserting the restriction on the size of the field $\varphi_N^{D,S}$ on~$S'_N$ comes at no significant cost even under the small-probability event that the maximum is large:

\begin{lemma}
\label{lemma-4.3}
For sets $S,S',D$ as in Proposition~\ref{prop-4.2} above,
\begin{equation}
\lim_{M\to\infty}\limsup_{t \to \infty}\,\limsup_{N \to \infty}\,\frac1t\texte^{\alpha t}\,
P\Bigl( \,\hat X_N^D\in S'_N\cap\Gamma_N^D(-t),\,\max_{x\in S'_N}|\varphi_N^{D,S}(x)|>M\Bigr)=0.
\end{equation}
\end{lemma}

\begin{proofsect}{Proof}
Thanks to the fact that~$S'$ is separated from~$S^\cc$ by a positive distance, Proposition~\ref{lem:3} yields
\begin{equation}
\label{E:4.6}
P\Bigl(\,\max_{x\in S'_N}|\varphi_N^{D,S}(x)|>M\Bigr)\le c\texte^{-c' M^2}
\end{equation}
for some~$c,c'\in(0,\infty)$ and all~$N\ge1$. This means we can immediately discard the event that the maximum $|\varphi_N^{D,S}(x)|$ exceeds, say,~$t^{2/3}$. For the complementary event, note that the Gibbs-Markov decomposition $h_N^D=h_N^S+\varphi_N^{D,S}$ on~$S$ implies
\begin{equation}
\begin{aligned}
P\Bigl( \,&\hat X_N^D\in S'_N\cap\Gamma_N^D(-t),\,M<\max_{x\in S'_N}|\varphi_N^{D,S}(x)|<t^{2/3}\Bigr)
\\
&\quad\le\sum_{M\le L\le t^{2/3}}
P\Bigl(\,\hat X_N^D\in S'_N,\, \max_{x\in D_N}h_N^D(x) > m_N + t,\,\,L<\max_{x\in S'_N}|\varphi_N^{D,S}(x)|\le L+1\Bigr)
\\
&\quad\le\sum_{M\le L\le t^{2/3}}
P\Bigl(\,\max_{x\in S'_N}h_N^S(x)> m_N+t-L-1,\,\max_{x\in S'_N}|\varphi_N^{D,S}(x)|> L\Bigr).
\end{aligned}
\end{equation}
Using that the two events in the last line are independent and applying \eqref{E:2.4q} and \eqref{E:4.6}, the \emph{limes superior} as~$N\to\infty$ of the last probability is bounded by $t\texte^{-\alpha t}$ times $c''\texte^{\alpha L-c' L^2}$ for some~$c',c''\in(0,\infty)$. The latter is summable on~$L\ge M$ uniformly in~$t$. The claim follows.
\end{proofsect}

Our proof of Proposition~\ref{prop-4.2} will make a repeated use of the fact that, conditional on an exceptionally high value for the global maximum of~$h_N^D$, all other extreme local maxima of~$h_N^D$ are considerably~smaller.

\begin{lemma}
\label{lem:5.1}
Let $ D \in\mathfrak D$. For any function $s\colon[0,\infty)\to[0,\infty)$ obeying $s(t) \leq t$ and $t - s(t) = o(t)$ as $t \to \infty$,
\begin{equation}
\label{e:3.1}
\lim_{t \to \infty}\,\lim_{r \to \infty}\, \limsup_{N \to \infty} \,\frac1t \texte^{\alpha t}
P\Bigl(\exists x\in\Gamma_N^D(-t),\,\exists y\in\Gamma_N^D(-s(t))\colon|x-y|>r\Bigr)
 = 0 \,.
\end{equation}
\end{lemma}

\begin{proofsect}{Proof}
After some straighforward arguments, the proof reduces to a question about asymptotic for the Laplace transform of the total mass of the $Z^D$-measure. However, since the relevant facts we need from~$Z^D$ have not yet been established for general~$D$, but are known for squares, we first reduce the whole problem to a square domain. 

Let~$S$ be an open square containing~$D$. Let $H_{N,r,t}^D$ denote the event in \eqref{e:3.1} for domain~$D$ and field~$h_N^D$ and let~$H_{N,r,t}^S$ denote the event for domain~$S$ with the field $h_N^S:=h_N^D+\varphi_N^{S,D}$. On~$H_{N,r,t}^D$, let $(x^\star,y^\star)$ be the vertices realizing the event that are chosen according to some \emph{a priori} ordering of all pairs; otherwise, let $(x^\star,y^\star)$ be just the first pair. Then
\begin{equation}
H_{N,r,t}^S\supseteq H_{N,r,t}^D\cap\bigl\{\varphi_N^{S,D}(x^\star)\ge0,\,\varphi_N^{S,D}(y^\star)\ge0\bigr\}.
\end{equation}
Now observe (similarly as in the proof of Lemma~\ref{lemma-2.4}) that the second event is independent of the first and has probability at least~$1/4$. Hence we get $P(H_{N,r,t}^D)\le 4 P(H_{N,r,t}^S)$ and it thus suffices to prove the result for open squares.

So suppose that~$D$ is an open square. For any $r_N\to\infty$ with $r_N/N\to0$, 
\begin{multline}
\qquad
H_{N,r,t}^D
\subseteq \bigl\{\exists x,y\in\Gamma_N(s)\colon r\le |x-y|<r_N\bigr\}
\\
\cup\bigl\{\eta_{N,r_N}^D(D\times[t,\infty))\ge 1,\,\eta_{N,r_N}^D(D\times[s(t),\infty))\ge2\bigr\}.
\qquad
\end{multline}
holds once~$N$ is sufficiently large. In the limit as~$N\to\infty$ and~$r\to\infty$, the probability of the first event on the right tends to zero by Proposition~\ref{lem:5} while the second event has an interpretation in terms of the limit~$\eta^D$ of point processes~$\{\eta_{N,r_N}^D\}$. In particular, it suffices to show that
\begin{equation}
\label{E:4.19a}
\lim_{t\to\infty}\,\frac1t\texte^{\alpha t}\,P\Bigl(\eta^D\bigl(D\times[t,\infty)\bigr)\ge1,\eta^D\bigl(D\times[s(t),\infty)\bigr)\ge2\Bigr)=0
\end{equation}
for any square domain~$D$. 

We will first write the event in \eqref{E:4.19a} as the set difference
\begin{equation}
\label{E:4.20a}
\Bigl\{\eta^D\bigl(D\times[t,\infty)\bigr)\ge1\Bigr\}\smallsetminus\Bigl\{\eta^D\bigl(D\times[t,\infty)\bigr)=1,\,\eta^D\bigl(D\times[s(t),t)\bigr)=0\Bigr\}
\end{equation}
and note that the second event is a subset of the first and so the probability of their set-difference is the difference of their probabilities.
Thanks to Theorem~1.1 of Biskup and Louidor~\cite{Biskup-Louidor}, the process~$\eta^D$ is Poisson with intensity $Z(\textd x)\otimes\texte^{-\alpha h}\textd h$. As the spatial coordinate of the points is not restrained in the above events, only the total mass $Z:=Z^D(D)$ matters. Abbreviating $\lambda:=\frac1\alpha\texte^{-\alpha t}$ and $\lambda':=\frac1\alpha\texte^{-\alpha s(t)}$, \eqref{E:4.20a} and the fact that $\{\eta^D(D\times[t,\infty))=1\}$ and $\{\eta^D(D\times[s(t),t))=0\}$ are independent allows us to write
\begin{equation}
P\Bigl(\eta^D\bigl(D\times[t,\infty)\bigr)\ge1,\eta^D\bigl(D\times[s(t),\infty)\bigr)\ge2\Bigr)
=E\bigl(1-\texte^{-\lambda Z}-\lambda Z\texte^{-\lambda'Z}\bigr).
\end{equation}
The key point is that, for square domains, the asymptotic  for~$\lambda,\lambda'$ small of the expression on the right is known. Indeed, Lemma~3.5 from~\cite{Biskup-Louidor} gives 
\begin{equation}
E\bigl(1-\texte^{-\lambda Z}\bigr)=\bigl[1+o(1)\bigr] c\lambda\log(1/\lambda),\qquad\lambda\downarrow0,
\end{equation}
and
\begin{equation}
E\bigl(Z\texte^{-\lambda Z}\bigr)=\bigl[1+o(1)\bigr] c\log(1/\lambda),\qquad\lambda\downarrow0,
\end{equation}
for some constant~$c\in(0,\infty)$. (In fact, for unit square we get $c:=a_\star$ where $a_\star$ is as in Proposition~\ref{thm-asymp}.) Interpreting this back using~$t$ and~$s(t)$, we get
\begin{multline}
\qquad
P\Bigl(\eta^D\bigl(D\times[t,\infty)\bigr)\ge1,\eta^D\bigl(D\times[s(t),\infty)\bigr)\ge2\Bigr)
\\
=\bigl[c+o(1)\bigr]\, t\texte^{-\alpha t}-\bigl[c+o(1)\bigr]\, s(t)\texte^{-\alpha t},\qquad t,s(t)\to\infty.
\qquad
\end{multline}
Under our assumptions on $s(t)$, this tends to zero as~$t\to\infty$ even after multiplication by~$t^{-1}\texte^{\alpha t}$.
\end{proofsect}

Notwithstanding our previous reduction to square domains in the previous proof, Proposition~\ref{prop-4.2} still requires at least a bound on the Laplace transform of the total mass of~$Z^D$. 

\begin{lemma}
\label{lem:5.2}
For each~$D\in\mathfrak D$, there is~$c>0$ such that for all $\lambda>0$ small enough, 
\begin{equation}
E\bigl(1 - \texte^{-\lambda Z^D(D)}\bigr) \leq c \lambda \log \tfrac{1}{\lambda} \,.
\end{equation}
\end{lemma}

\begin{proofsect}{Proof}
By Theorem~\ref{thm:1} and some simple approximation arguments, the expectation on the left is the limit of the probability that $\max_{x\in D_N}h_N^D(x)>m_N+t$, where $\lambda:=\frac1\alpha\texte^{-\alpha t}$. The claim follows from, e.g., Proposition~\ref{lem:6}.
\end{proofsect}

\begin{proofsect}{Proof of Proposition~\ref{prop-4.2}}
For general~$A\subseteq D$ define
\begin{equation}
\hat X_N^{D,A}:=\operatornamewithlimits{argmax}_{A_N}\,h_N^D
\end{equation}
and note that $\hat X_N^{D,D}=\hat X_N^D$. Consider sets $A\subseteq S\subseteq S'\subseteq D$ as in the statement. Writing the event in the first probability on the left-hand side of \eqref{E:4.4a} as $\{\hat X_N^D\in A_N\cap\Gamma_N^D(-t)\}$, our goal is to simply replace $\hat X_N^D$ by $\hat X_N^S$. We will do this by the sequence of replacements
\begin{equation}
\hat X_N^D\longrightarrow \hat X_N^{D,S'}\longrightarrow \hat X_N^{S,S'} \longrightarrow \hat X_N^S.
\end{equation}
Each of the arrows will constitute a STEP as designated below. Note that we can freely intersect the primary event with
\begin{equation}
F_{N,M}:=\bigl\{\,\max_{x\in S_N'}|\varphi_N^{D,S}(x)|\le M\bigr\},
\end{equation}
as the complementary event can be dismissed (in the limit~$M\to\infty$) thanks to Lemma~\ref{lemma-4.3}. To keep our expressions short, we introduce the (sub-probability) measure $P_{N,M}$ by
\begin{equation}
P_{N,M}(E):=P(E\cap F_{N,M}).
\end{equation}
We now proceed with the actual proof.

\smallskip\noindent
\textsl{STEP 1}: Our task to replace $\hat X_N^D$ by $\hat X_N^{D,S'}$ in the above event. Here we note that $\hat X_N^D\in A_N\cap\Gamma_N^D(-t)$ implies $\hat X_N^{D,S'}=X_N^D$ (because $A\subseteq S'$) and thus 
\begin{equation}
\bigl\{\hat X_N^D\in A_N\cap\Gamma_N^D(-t)\bigr\}\subseteq\bigl\{\hat X_N^{D,S'}\in A_N\cap\Gamma_N^D(-t)\bigr\}.
\end{equation}
On the difference of the two events, the maximizer of~$h_N^D$ on~$S_N'$ lies in~$A_N$ but the absolute maximizer lies in $D_N\setminus S_N'$. As $\textd(A,(S')^\cc)>0$, the difference is for large enough~$N$ contained in the event from \eqref{e:3.1} with $s(t):=t$. Hence
\begin{equation}
\label{E:4.31}
\lim_{t\to\infty}\,\limsup_{N\to\infty}\,\frac1t\texte^{\alpha t}\,\Bigl|\,P_{N,M}\bigl(\hat X_N^D\in A_N\cap\Gamma_N^D(-t)\bigr)-P_{N,M}\bigl(\hat X_N^{D,S'}\in A_N\cap\Gamma_N^D(-t)\bigr)\Bigr|=0
\end{equation}
by Lemma~\ref{lem:5.1}.

\smallskip\noindent
\textsl{STEP 2}: Our next goal is to substitute $\hat X_N^{D,S'}$ by~$\hat X_N^{S,S'}$ in $\{\hat X_N^{D,S'}\in A_N\cap\Gamma_N^D(-t)\}\cap F_{N,M}$. First we observe that these two points cannot be too far from each other. Indeed, on $F_{N,M}$
\begin{equation}
h_N^D(\hat X_N^{S,S'})
\ge h_N^S(\hat X_N^{S,S'})-M
\ge h_N^S(\hat X_N^{D,S'})-M\ge h_N^D(\hat X_N^{D,S'})-2M
\end{equation}
and so $\hat X_N^{D,S'}\in\Gamma_N^D(-t)$ and $|\hat X_N^{D,S'}-\hat X_N^{S,S'}|>r$ imply the occurrence of the event in \eqref{e:3.1} with $s(t):=t-2M$. Hence,
\begin{equation}
\label{E:4.34}
\lim_{t\to\infty}\,\lim_{r\to\infty}\,\limsup_{N\to\infty}\,\frac1t\texte^{\alpha t}\,P_{N,M}\Bigl(\hat X_N^{D,S'}\in\Gamma_N^D(-t),\,|\hat X_N^{D,S'}-\hat X_N^{S,S'}|>r\Bigr)=0.
\end{equation}
Moving to the cases when $|\hat X_N^{D,S'}-\hat X_N^{S,S'}|\le r$, the fact that $\varphi_N^{D,S}$ is discrete harmonic on the square~$S_N$ and bounded by~$M$ on~$S_N'$ implies
\begin{equation}
\bigl|\varphi_N^{D,S}(x)-\varphi_N^{D,S}(y)\bigr|\le cM\frac{|x-y|}N,\qquad x,y\in S_N',
\end{equation}
for some constant~$c\in(0,\infty)$ independent of $N$ and~$M$. So once $|\hat X_N^{D,S'}-\hat X_N^{S,S'}|\le r$, 
\begin{equation}
0\le\,h_N^D(\hat X_N^{D,S'})-h_N^D(\hat X_N^{S,S'})\le cM\frac rN.
\end{equation}
It follows that the part of the symmetric difference
\begin{equation}
\bigl\{\hat X_N^{D,S'}\in A_N\cap\Gamma_N^D(-t)\bigr\}\triangle \bigl\{\hat X_N^{S,S'}\in A_N\cap\Gamma_N^D(-t)\bigr\}
\end{equation}
that belongs to $F_{N,M}\cap\{|\hat X_N^{D,S'}-\hat X_N^{S,S'}|\le r\}$ is contained in the union
\begin{equation}
\bigl\{\exists x\in A_N\cap\Gamma_N^D(-t)\colon\textd(x,A_N^\cc)\le r\bigr\}
\cup\bigl\{\,\max_{x\in S_N'}h_N^D(x)\in[m_N+t,m_N+t+cMr/N]\bigr\}.
\end{equation}
Since~$|\partial A|$ has zero Lebesgue measure, Proposition~\ref{lem:6} ensures 
\begin{equation}
\lim_{t\to\infty}\,\lim_{r\to\infty}\,\limsup_{N\to\infty}\,\frac1t\texte^{\alpha t}\,P\bigl(\exists x\in A_N\cap\Gamma_N^D(-t)\colon\textd(x,A_N^\cc)\le r\bigr)=0.
\end{equation}
Similarly, the convergence to the Cox process established in Theorem~\ref{thm:1} along with the fact that $Z^D(\partial S')=0$ a.s.\ implies, for any~$\epsilon>0$,
\begin{equation}
\label{E:4.40a}
\limsup_{N\to\infty}\,P\bigl(\,\max_{x\in S_N'}h_N^D(x)\in[m_N+t,m_N+t+\epsilon]\bigr)
\le E\bigl(\texte^{-\lambda\texte^{-\alpha\epsilon}Z}-\texte^{-\lambda Z}\bigr),
\end{equation}
where we used the shorthands $Z:=Z^D(S')$ and $\lambda:=\frac1\alpha\texte^{-\alpha t}$. The right-hand side is further bounded by $E(1-\texte^{-\lambda(1-\texte^{-\alpha\epsilon})Z})$. Applying Lemma~\ref{lem:5.2} (and $Z^D(S')\le Z^D(D)$), for~$\epsilon>0$ fixed this decays as a constant times $t\texte^{-\alpha t}(1-\texte^{-\alpha\epsilon})$ as $t\to\infty$. Taking~$\epsilon\downarrow0$, the probability in \eqref{E:4.40a} thus tends to zero even after multiplication by $t^{-1}\texte^{\alpha t}$. Combining with \eqref{E:4.34}, we have derived
\begin{equation}
\label{E:4.40}
\lim_{t\to\infty}\,\lim_{r\to\infty}\,\limsup_{N\to\infty}\,\frac1t\texte^{\alpha t}\,\Bigl|\,P_{N,M}\bigl(\hat X_N^{D,S'}\in A_N\cap\Gamma_N^D(-t)\bigr)
-P_{N,M}\bigl(\hat X_N^{S,S'}\in A_N\cap\Gamma_N^D(-t)\bigr)\Bigr|=0
\end{equation}
and thus completed STEP 2.

\smallskip\noindent
\textsl{STEP 3}: Our final task is to replace $\hat X_N^{S,S'}$ by~$\hat X_N^{S}$. Here we note that if $\hat X_N^S\in S_N'$ then $\hat X_N^S=\hat X_N^{S,S'}$ and so the difference of events $\{X_N^S\in A_N\cap\Gamma_N^D(-t)\}$ and $\{X_N^{S,S'}\in A_N\cap\Gamma_N^D(-t)\}$ implies, on~$F_{N,M}$, the event in \eqref{E:3.1} with~$s(t):=t-M$. Lemma~\ref{lem:5.1} yields
\begin{equation}
\label{E:4.41}
\lim_{t\to\infty}\,\lim_{r\to\infty}\,\limsup_{N\to\infty}\,\frac1t\texte^{\alpha t}\,\Bigl|\,P_{N,M}\bigl(\hat X_N^{S,S'}\in A_N\cap\Gamma_N^D(-t)\bigr)
-P_{N,M}\bigl(\hat X_N^{S}\in A_N\cap\Gamma_N^D(-t)\bigr)\Bigr|=0.
\end{equation}
The desired claim now follows by combining \eqref{E:4.31}, \eqref{E:4.40} and \eqref{E:4.41} with Lemma~\ref{lemma-4.3}.
\end{proofsect}

\section{Representation via a derivative martingale}
\label{sec5}\noindent
As noted before, our proofs proceed in a slightly different order from this point on than the statements of the results. Indeed, our next goal is to prove Theorem~\ref{thm:1.6} deferring the proof of conformal-transformation rule from Theorem~\ref{thm:3} to the next section. 

\subsection{Uniqueness characterization}
\label{sec:5.1}\noindent
The proof of Theorem~\ref{thm:1.6} will be based on a representation of measure~$M^D$ in the statement as a weak limit of derivative-martingale like expressions that depend, as far as random objects are concerned, only on a single Gaussian field. A representation of this kind was put forward already in Theorem~1.4 of Biskup and Louidor~\cite{Biskup-Louidor} but the derivations there were based on specific facts about the DGFF on~$\Z^2$ while here we deal with a continuum problem from the outset. It turns out that, instead of approximating our domains by unions of squares, it is more convenient to work with regular triangulations.

Given a domain~$D\in\mathfrak D$ and an integer $K\ge1$, consider a tiling of~$\C$ by equilateral triangles of side-length~$K^{-1}$ with vertices lying on a triangular lattice of mesh~$K^{-1}$. Let $D^1,\dots,D^{m_K}$ denote the (open) triangles that lie entirely in~$D$ and abbreviate
\begin{equation}
\wt D:=\bigcup_{i=1}^{m_K}D^i.
\end{equation}
Note that $\wt D$ depends on~$D$ and~$K$, although we do not make this notationally explicit.
Given a~$\delta\in(0,1)$, let us assume that indices $i=1,\dots,n_K$, for some $n_K\le m_K$, enumerate those triangles that are at least distance~$\delta$ away from~$\C\smallsetminus D$ and let $D^1_\delta,\dots,D^{n_K}_\delta$ denote the equilateral triangles of side length $(1-\delta)K^{-1}$ with the same orientation and centers as $D^1,\dots,D^{n_K}$, respectively. 
Recall the definition
\begin{equation}
\text{osc}_A f:=\sup_{x\in A}f(x)-\inf_{x\in A}f(x)
\end{equation}
of oscillation of a function~$f$ on a set~$A$.
Then we have:

\begin{figure}[t]
\vglue0.2cm
\centerline{\includegraphics[width=3.6truein]{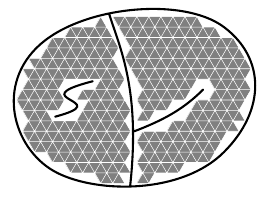}}
\begin{quote}
\small
\vglue0.2cm
{\sc Fig~4.\ }
\label{fig3b}
An illustration of domain $\wt D$ (the shaded triangles) for~$D$ as in Fig.~1. The interior of the union of the closures of these triangles is domain~$D'$ used later in this section.
\normalsize
\end{quote}
\end{figure}

\begin{theorem}
\label{thm:5.1}
Recall that $\alpha:=2/\sqrt g$ and consider the events
\begin{equation}
A^i_{K,R}:=\bigl\{\text{\rm osc}_{D^i_\delta}\Phi^{D,\wt D}\le R\bigr\}\cap\bigl\{\max_{{D^i_\delta}}\Phi^{D,\wt D}\le2\sqrt g\log K-R\bigr\},\qquad i=1,\dots,n_K.
\end{equation}
Then for any family $\{M^D\colon D\in\mathfrak D\}$ of random Borel measures satisfying conditions (0-4) of Theorem~\ref{thm:1.6} with~$c:=1$ in~(4), the random measure
\begin{equation}
\label{E:5.2}
\alpha\psi^D(x)\sum_{i=1}^{n_K}\1_{A_{K,R}^i}\bigl(\alpha\Var(\Phi^{D,\wt D}(x))-\Phi^{D,\wt D}(x)\bigr)\texte^{\alpha \Phi^{D,\wt D}(x)-\frac12\alpha^2\Var(\Phi^{D,\wt D}(x))}\,\1_{D_\delta^i}(x)\,\textd x
\end{equation}
tends in law to $M^D$ in the limit as $K\to\infty$, $R\to\infty$ and~$\delta\downarrow0$ (in this order). This holds irrespective of the orientation of the triangular grid.
\end{theorem}

The representation \eqref{E:5.2} immediately yields:

\begin{proofsect}{Proof of Theorem~\ref{thm:1.6}}
In light of Theorem~\ref{thm:2} and Corollary~\ref{cor-Z-measure}, the family $\{Z^D\colon D\in\mathfrak D\}$ constructed in Theorem~\ref{thm:1} obeys conditions (0-4) of Theorem~\ref{thm:1.6}. By Theorem~\ref{thm:5.1}, these properties determine the law of $\{Z^D\colon D\in\mathfrak D\}$ uniquely.
\end{proofsect}

\begin{remark}
As is easy to check, $\Phi^{D,\wt D}$ is the projection of the CGFF on~$D$ onto the set of functions that are piece-wise harmonic on each of the triangles that falls entirely inside~$D$. Hence, if it were not for the truncations and other restrictions, the measure in \eqref{E:5.2} would already have the desired form of (an approximation to) the critical LQG measure. Unfortunately, we do not know how to (cleanly) remove the truncations to make the connection precise through this route.
\end{remark}

Assume~$\{M^D\colon D\in\mathfrak D\}$ and~$\Phi^{D,\wt D}$ are statistically-independent realizations of the corresponding processes and suppose throughout the rest of this section that the conditions (0-4) in Theorem~\ref{thm:1.6} hold. The proof of Theorem~\ref{thm:5.1} opens up by noting that
\begin{equation}
\label{E:5.4a}
\1_{\wt D}(x)M^D(\textd x)\laweq \sum_{i=1}^{ m_K}\texte^{\alpha\Phi^{D,\wt D}(x)}M^{D^i}(\textd x)
\end{equation}
holds. This is thanks to the Gibbs-Markov property (property (3) in Theorem~\ref{thm:1.6}).
As the next proposition shows, the restriction of the measure on the left to~$\wt D$ is immaterial:

\begin{proposition}
\label{prop-slither}
For each $\epsilon>0$,
\begin{equation}
\lim_{\delta\downarrow0}\,\limsup_{K\to\infty}\,P\biggl(M^D\Bigl(D\smallsetminus\bigcup_{i=1}^{n_K} D_\delta^i\Bigr)>\epsilon\biggr)=0.
\end{equation}
\end{proposition}

Focussing now on the right-hand side of the expression in \eqref{E:5.4a}, our derivations will require restrictions on irregularity of the field $\Phi^{D,\wt D}$; in particular, we need a bound on the maximal value and oscillation of the field in individual triangles. This is the content of the following propositions (recall again that $\wt D$ depends on~$K$ and~$\delta$):

\begin{proposition}
\label{prop-0}
Let~$x_0^i$ denote the center of the triangle~$D^i$. There is~$\zeta>0$ such that, for each~$\delta>0$
\begin{equation}
\limsup_{K\to\infty}\,P\Bigl(\,\max_{i=1,\dots,n_K}\Phi^{D,\wt D}(x_0^i)>2\sqrt g\log K-\zeta\log\log K\Bigr)=0.
\end{equation}
\end{proposition}

\begin{proposition}
\label{prop-1}
For any~$\delta\in(0,1)$ and any $\epsilon>0$,
\begin{equation}
\label{EE:1.2}
\lim_{R\to\infty}\,\limsup_{K\to\infty}\,
P\Bigl(\,\,\sum_{i=1}^{n_K}\int_{D^i_\delta}M^{D^i}(\textd x)\texte^{\alpha\Phi^{D,\wt D}(x)}\1_{\{\text{\rm osc}_{D^i_\delta}\Phi^{D,\wt D}>R\}}>\epsilon\Bigr)=0.
\end{equation}
\end{proposition}

As a final ingredient, we will need to control the small-argument asymptotic of the Laplace transform of the integral of~$M^D$ against an equicontinuous class functions. This will be based on the tail assumption in part~(4) of Theorem~\ref{thm:1.6}. Fortunately, we will only need this for one domain; namely, the equilateral triangle~$T$ of side-length 1 centered at~$0$ and oriented consistently with the above triangular cover of~$\C$. 

Let~$T_\delta$ be the triangle of side-length~$1-\delta$ centered and oriented same as~$T$. Given~$\beta>0$ and $R>0$ let~$\FF_{R,\beta,\delta}$ denote the class of continuous functions $f\colon\overline T\to\R$ such that,
\begin{equation}
\label{E:5.8c}
f(x)\ge\beta\quad\text{and}\quad |f(x)-f(y)|\le R|x-y|,\qquad x,y\in\overline T_\delta.
\end{equation}
Then we have:

\begin{proposition}
\label{prop-3}
Fix~$\beta>0$ and~$R>0$. For each $\epsilon>0$ there are $\delta_0>0$ and $\lambda_0>0$ such that, for all $\lambda\in(0,\lambda_0)$, all $\delta\in(0,\delta_0)$ and all~$f\in\FF_{R,\beta,\delta}$,
\begin{equation}
(1-\epsilon)\int_{T_\delta}f(x)\psi^T(x)\,\textd x\le
\frac{\log E(\texte^{-\lambda M^T(f\1_{T_\delta})})}{\lambda\log\lambda}\le(1+\epsilon)\int_{T_\delta}f(x)\psi^T(x)\,\textd x,
\end{equation}
where we abbreviated $M^T(f\1_{T_\delta}):=\int_{T_\delta}M^T(\textd x)\,f(x)$.
\end{proposition}

Deferring temporarily the proof of all four propositions, we can now prove the derivative-martingale like representation in \eqref{E:5.2}:

\begin{proofsect}{Proof of Theorem~\ref{thm:5.1}}
In light of \eqref{E:5.4a} and the above propositions, the random measure
\begin{equation}
\label{E:5.10}
M^D_{K,R,\delta}(\textd x):=\sum_{i=1}^{n_K}\1_{A_{K,R}^i}\texte^{\alpha\Phi^{D,\wt D}(x)}\1_{D_\delta^i}(x)\,M^{D^i}(\textd x)
\end{equation}
tends weakly to $M^D$ in the stated limits. Pick some~$f\colon\overline D\to[0,\infty)$ bounded and continuous and assume~$f\ge \beta$ on~$D$ for some~$\beta>0$. Writing $M^D_{K,R,\delta}(f)$ for the integral of~$f$ with respect to $M^D_{K,R,\delta}$, \eqref{E:5.10} reads
\begin{equation}
E\bigl(\texte^{-M^D_{K,R,\delta}(f)}\bigr)=E\biggl(\,\,\prod_{\begin{subarray}{c}
i=1,\dots,n_K\\ A_{K,R}^i\text{ occurs}
\end{subarray}}
\exp\Bigl\{-\texte^{\alpha\Phi^{D,\wt D}(x_0^i)}M^{D^i}(f\1_{D_\delta^i}\texte^{\alpha(\Phi^{D,\wt D}-\Phi^{D,\wt D}(x_0^i))})\Bigr\}\biggr).
\end{equation}
Whenever $A_{K,R}^i$ occurs, we have
\begin{equation}
\label{E:6.13ww}
\texte^{\alpha\Phi^{D,\wt D}(x_0^i)}\le\texte^{\alpha 2\sqrt g\,\log K-\alpha\zeta\log\log K}=K^4(\log K)^{-\alpha \zeta}
\end{equation}
and
\begin{equation}
\label{E:5.13}
\bigl|\Phi^{D,\wt D}(x)-\Phi^{D,\wt D}(x_0^i))\bigr|\le R,\qquad x\in D_\delta^i.
\end{equation}
Since $x\mapsto \Phi^{D,\wt D}(x)$ is harmonic on~$D^i$ and~$f$ is, being a member of~$\FF_{R,\beta,\delta}$, uniformly continuous and positive, for all realizations of~$\Phi^{D,\wt D}$ satisfying \twoeqref{E:6.13ww}{E:5.13} the function
\begin{equation}
g(z):=f\bigl(x_0^i+K^{-1}z\bigr)\texte^{\alpha(\Phi^{D,\wt D}(x_0^i+z/K)-\Phi^{D,\wt D}(x_0^i))}
\end{equation}
lies in~$\FF_{R',\beta',\delta}$ for some suitable~$R'$ and~$\beta'$ that depend only on~$R$ and~$\beta$. The scaling property~(2) tells us that $K^4 M^{D^i}(K^{-1}\textd x)$ has the law of a shifted $M^T$. Proposition~\ref{prop-3} with $\lambda:=K^{-4}\texte^{\alpha\Phi^{D,\wt D}(x_0^i)}$ then implies
\begin{equation}
\label{E:5.15}
\begin{aligned}
\exp\biggl\{&(1+\epsilon)\,K^{-4}\log\bigl(K^{-4}\texte^{\alpha\Phi^{D,\wt D}(x_0^i)}\bigr)\int_{D_\delta^i}f(x)\,K^4\psi^{D^i}(x)\texte^{\alpha\Phi^{D,\wt D}(x)}\,\textd x\biggr\}
\\
&
\le E\biggl(\,\exp\Bigl\{-\texte^{\alpha\Phi^{D,\wt D}(x_0^i)}M^{D^i}(f\1_{D_\delta^i}\texte^{\alpha(\Phi^{D,\wt D}-\Phi^{D,\wt D}(x_0^i))})\Bigr\}\,\bigg|\,\Phi^{D,\wt D}\biggr)
\\
&\qquad\le
\exp\biggl\{(1-\epsilon)\,K^{-4}\log\bigl(K^{-4}\texte^{\alpha\Phi^{D,\wt D}(x_0^i)}\bigr)\int_{D_\delta^i}f(x)\,K^4\psi^{D^i}(x)\texte^{\alpha\Phi^{D,\wt D}(x)}\,\textd x\biggr\}
\end{aligned}
\end{equation}
whenever~$A_{K,R}^i$ holds,~$K$ is sufficiently large and~$\delta$ is sufficiently small, uniformly in~$i=1,\dots,n_K$. Here the factor~$K^4$ inside the integral arises from scaling (and shifting) the measure~$\psi^T(x)\textd x$ on the unit triangle~$T$ to the triangle~$D^i$ of side-length~$K^{-1}$. 

In order to convert the expression in the exponentials on both sides of \eqref{E:5.15} into the desirable form, we will absorb the term $\log(K^{-4}\texte^{\alpha\Phi^{D,\wt D}(x_0^i)})$ into the integral and turn it into the expression $\alpha(\Phi^{D,\wt D}(x)-\alpha\Var(\Phi^{D,\wt D}(x)))$. For this we first note (cf Lemma~\ref{lemma-var-bd}) that
\begin{equation}
\label{E:5.16}
\Var\bigl(\Phi^{D,\wt D}(x)\bigr)=g\log K+O(1),\qquad x\in D_\delta^i.
\end{equation}
As~$\alpha^2g=4$, assuming that $A_{K,R}^i$ occurs, this and \eqref{E:5.13} permit us to write,
\begin{equation}
\log(K^{-4}\texte^{\alpha\Phi^{D,\wt D}(x_0^i)})
=\alpha\bigl(\Phi^{D,\wt D}(x)-\alpha\Var(\Phi^{D,\wt D}(x))\bigr)+O(1),
\end{equation}
where the~$O(1)$ term is bounded uniformly in~$x\in D_\delta^i$. But $\log(K^{-4}\texte^{\alpha\Phi^{D,\wt D}(x_0^i)})$ is order $\log\log K$ on~$A_{K,R}^i$ and so the $O(1)$ term can be absorbed at the cost of changing $(1\pm\epsilon)$ into~$(1\pm2\epsilon)$. Denoting the measure in \eqref{E:5.2} by $\wt Z^D_{K,R,\delta}$ and recalling that
\begin{equation}
\psi^{D^i}(x)=\psi^D(x)\texte^{-\frac12\alpha^2\Var(\Phi^{D,\wt D}(x))},\qquad x\in D^i,
\end{equation}
(see, e.g., \eqref{E:4.12qq}) we thus get
\begin{equation}
E\bigl(\texte^{-(1+\epsilon)\wt Z^D_{K,R,\delta}(f)}\bigr)\le
E\bigl(\texte^{-M^D_{K,R,\delta}(f)}\bigr)\le E\bigl(\texte^{-(1-\epsilon)\wt Z^D_{K,R,\delta}(f)}\bigr).
\end{equation}
In particular, since~$\epsilon$ was arbitrary, $\wt Z^D_{K,R,\delta}(f)\Lawarrow M^D(f)$ in the stated limits for any bounded and continuous function $f\colon\overline D\to\R$ with~$f\ge\beta$. But~$\beta$ was arbitrary and such functions are dense in the space of all bounded and continuous functions and so the claim follows.
\end{proofsect}

\subsection{Proofs of key propositions}
We are left with the proof of the above four propositions. We begin with those that are easier:

\begin{proofsect}{Proof of Proposition~\ref{prop-slither}}
Let us temporarily write~$\wt D_{K,\delta}$ instead of just~$\wt D$.
Suppose the claim fails. Then there are sequences $K_n\to\infty$ and~$\delta_n\downarrow0$, the latter exponentially decaying, such that $\liminf_{n\to\infty}P(M^D(D\smallsetminus\wt D_{K_n,\delta_n})>\epsilon)>\epsilon$ for some~$\epsilon>0$. Now $\leb(D\smallsetminus\wt D_{K_n,\delta_n})=O(\delta_n)$ and so $B_m:=\bigcup_{n\ge m}(D\smallsetminus D_{K_n,\delta_n})$ obeys $\leb(B_m)\to0$ as~$m\to\infty$. By $D\smallsetminus\wt D_{K_n,\delta_n}\subseteq B_n$ and~$B_n$ is decreasing and so $P(M^D(B_n)>\epsilon)>\epsilon$ for all~$n\ge1$. But this contradicts the stochastic absolute continuity assumption because $B_\infty:=\bigcap_{n\ge1}B_n$ has zero Lebesgue measure and yet non-zero $M^D$ measure with positive probability. The claim follows.
\end{proofsect}

For the second proposition, and also later reference, we need to control the variance of~$\Phi^{D,\wt D}$ and some linear combinations thereof:

\begin{lemma}
\label{lemma-var-bd}
For each $\delta\in(0,1)$ there is $c(\delta)<\infty$ such that
\begin{equation}
\label{E:5.20cc}
\sup_{K\ge1}\,\,\max_{i=1,\dots,n_K}\,\,\sup_{x\in D_\delta^i}\,
\Bigl|\Var\bigl(\Phi^{D,\wt D}(x)\bigr)-g\log K\Bigr|\le c(\delta)
\end{equation}
and
\begin{equation}
\label{EE:1.3}
\sup_{K\ge1}\,\,\max_{i,j=1,\dots,n_K}\,\sup_{x,y\in D_\delta^i}\,\sup_{z\in D_\delta^j}\,\,\text{\rm Cov}\bigl(\Phi^{D,\wt D}(z),\Phi^{D,\wt D}(y)-\Phi^{D,\wt D}(x))\bigr)\le c(\delta).
\end{equation}
In fact, we even have
\begin{equation}
\label{E:6.22r}
\sup_{K\ge1}\,\max_{i=1,\dots,n_K}\,\sup_{\begin{subarray}{c}
x,y\in D_\delta^i\\x\ne y
\end{subarray}}
\frac{E\bigl((\Phi^{D,\wt D}(x)-\Phi^{D,\wt D}(y))^2\bigr)}{ K |x-y|}\le c(\delta).
\end{equation}
\end{lemma}

\begin{proofsect}{Proof}
Recall that~$C^{D,\wt D}(x,y)$ denotes the covariance of the Gaussian field $\Phi^{D,\wt D}$. Using that the harmonic measure is normalized to one, we rewrite \eqref{E:1.6} as
\begin{equation}
C^{D,\wt D}(x,y)=\int_{\partial D\times\partial\wt D}\Pi^D(x,\textd z)\otimes\Pi^{\wt D}(x,\textd \tilde z)\,\log\frac{|y-z|}{|y-\tilde z|}.
\end{equation}
To get \eqref{E:5.20cc} we now evaluate this for~$y=x\in D_\delta^i$ and note that, given any $z\in\partial D$ and $\tilde z\in\partial D_\delta^i$, the ratio $\frac{|x-\tilde z|}{|x-z|}$ is bounded between a constant times $\delta/K$ and a constant times $1/(\delta K)$. (The constant depends on the diameter of~$D$ but not on~$K$.) 

For \eqref{EE:1.3} and \eqref{E:6.22r} we pick $x,y\in D^i_\delta$, $z\in D^j_\delta$ and rewrite \eqref{E:1.6}  as
\begin{multline}
\label{EE:1.6}
\qquad
\text{\rm Cov}\bigl(\Phi^{D,\wt D}(z),\Phi^{D,\wt D}(x)-\Phi^{D,\wt D}(y))\bigr)
\\
=\int_{\partial D}\Pi^D(z,\textd u)\log\frac{|x-u|}{|y-u|}-\int_{\partial\wt D}\Pi^{\wt D}(z,\textd u)\log\frac{|x-u|}{|y-u|}.
\qquad
\end{multline}
Given any~$u\in\C\smallsetminus\{x\}$, the triangle inequality shows
\begin{equation}
\biggl|\frac{|y-u|}{|x-u|}-1\biggr|\le\frac{|x-y|}{|x-u|}
\end{equation}
 which by the concavity of $x\mapsto\log(1+x)$ for~$x>-1$ gives
\begin{equation}
-\frac{|x-y|}{|y-u|}\le\log\frac{|x-u|}{|y-u|}\le \frac{|x-y|}{|x-u|},
\end{equation}
as soon as the ratios on the left and right are less than one. This happens for all~$u\in\partial D\cup\partial\wt D$ when~$|x-y|<\delta/K$ and, under this condition, the first integral in \eqref{EE:1.6} is at most a constant times~$\delta^{-1}|x-y|$ in absolute value while the second is a most a constant times $\delta^{-1}K|x-y|$ in absolute value. Writing the difference $\Phi^{D,\wt D}(x)-\Phi^{D,\wt D}(y)$ as the sum of differences for pairs of points at distance less than $\delta/K$ and using the convexity of~$\wt D_\delta^i$, this generalizes to all 
$x,y\in D^i_\delta$ and all~$z\in D^j_\delta$. Hereby \eqref{EE:1.3} and \eqref{E:6.22r} follows.
\end{proofsect}

\begin{proofsect}{Proof of Proposition~\ref{prop-0}}
Thanks to \eqref{E:5.20cc}, the standard Gaussian estimate and some simple algebra show
\begin{equation}
P\bigl(\,\Phi^{D,\wt D}(x_0^i)>2\sqrt g\log K-R\bigr)
\le \frac c{\sqrt{\log K}}\,K^{-2}\,\texte^{\alpha R}
\end{equation}
for some constant~$c>0$, uniformly in~$i=1,\dots,n_K$ and~$R$ such that $0\le R\le \sqrt g\,\log K$. As there are only order~$K^2$ points to control, the claim follows by a union bound.
\end{proofsect}

The proof of Proposition~\ref{prop-3}, dealing with the measure $M^T$ on the triangle~$T$ of unit length, is somewhat more involved. The argument naturally splits into two parts the first of which is the content of the following claim:

\begin{lemma}
\label{lemma-5.6}
Fix~$\beta>0$ and~$R>0$. For each $\epsilon>0$ there are $\delta_0>0$ and $\lambda_1>0$ such that, for all $\lambda\in(0,\lambda_1)$, all $\delta\in(0,\delta_0)$ and all~$f\in\FF_{R,\beta,\delta}$,
\begin{equation}
(1-\epsilon)\int_{T_\delta}f(x)\psi^T(x)\,\textd x\le
\frac{E(M^T(f\1_{T_\delta})\texte^{-\lambda M^T(T)})}{\log(\ffrac1\lambda)}\le(1+\epsilon)\int_{T_\delta}f(x)\psi^T(x)\,\textd x
\end{equation}
\end{lemma}

\begin{proofsect}{Proof}
Fix~$\beta>0$ and~$R>0$ and, given an~$\epsilon>0$, partition~$T$ into $K^2$ triangles $\{T_i\colon i=1,\dots,K^2\}$ of side~$K^{-1}$, where~$R/K\le\epsilon$. Thanks to our assumption that condition (2) in Theorem~\ref{thm:1.6} holds, we may find~$\lambda_1>0$ such that
\begin{equation}
\label{E:5.21a}
(1-\epsilon)\int_{T_i}\psi^T(x)\textd x\le
\frac{E\bigl(M^T(T_i)\texte^{-\lambda M^T(T)}\bigr)}{\log(\ffrac1\lambda)}
\le 
(1+\epsilon)\int_{T_i}\psi^T(x)\textd x
\end{equation}
is true for all~$\lambda\in(0,\lambda_1)$ and all~$i=1,\dots,K^2$.

Now let~$f$ be a function such that \eqref{E:5.8c} holds and observe that, since~$f\ge0$ we necessarily have~$f\le R$.
Thanks to the Lipschitz estimate in \eqref{E:5.8c}, we may then approximate~$f$ from above and below by a function $f_K$ that is constant on each~$T_i$ and $|f_K-f|\le\epsilon$. Then \eqref{E:5.21a} gives
\begin{equation}
\begin{aligned}
E\bigl(M^T(f\1_{T_\delta})\texte^{-\lambda M^T(T)}\bigr)
&\le E\bigl(M^T(f_K+\epsilon)\texte^{-\lambda M^T(T)}\bigr)\\
&\le\log(\ffrac1\lambda)\sum_{i=1}^{K^2}(1+\epsilon)
\int_{T_i}\psi^T(x)\bigl(f_K(x)+\epsilon\bigr)\textd x
\\
&\le \log(\ffrac1\lambda)(1+\epsilon)\int_T\psi^T(x)\bigl(f(x)+2\epsilon\bigr)\textd x
\end{aligned}
\end{equation}
as soon as~$\lambda\in(0,\lambda_1)$. For the corresponding lower bound we assume that~$2\epsilon\in(0,\beta)$; a similar derivation then yields
\begin{equation}
\frac{E\bigl(M^T(f\1_{T_\delta})\texte^{-\lambda M^T(T)}\bigr)}{\log(\ffrac1\lambda)}
\ge
(1-\epsilon)\int_{T_\delta'}\psi^T(x)\bigl(f(x)-2\epsilon\bigr)\textd x,
\end{equation}
where~$T_\delta'$ is the union of all~$T_i$'s that lie entirely in~$T_\delta$. Since the integral of~$\psi^T$ is finite and positive and~$f$ is bounded below and above by constants that depend only on~$\beta$ and~$R$, respectively, the claim follows by choosing~$\delta$ sufficiently small and~$K$ large enough.
\end{proofsect}

\begin{remark}
\label{remark-for-Stefan}
The bound \eqref{E:5.21a} is the only place where the Laplace transform asymptotics \eqref{E:1.25} of the $Z^D$-measure enters the proof of Theorem~\ref{thm:5.1}. Here we state this bound only for unit triangles because, in \eqref{E:5.15}, we use invariance under dilations from \eqref{E:2.17ie} to bring triangles of side-length~$1/K$ to unit scale. Should the dilation property in \eqref{E:2.17ie} not be \emph{a priori} available, the same can be achieved by enhancing \eqref{E:1.25} to:
\settowidth{\leftmargini}{(111)}
\begin{enumerate}
\item[(4')] (uniform Laplace transform tail) There is~$c\in(0,\infty)$ such that for any open set~$A$ with $\overline A\subset D$, and with
\begin{equation}
D_K:=\bigl\{K^{-1}x\colon x\in D\bigr\}\quad\text{and}\quad A_K:=\bigl\{K^{-1}x\colon x\in A\bigr\}
\end{equation}
we have
\begin{equation}
\label{E:1.25a}
\lim_{\lambda\downarrow0}\,\,\sup_{K\ge1}\,\,\biggl|\,\frac{K^4\,E(M^{D_K}(A_K)\texte^{-\lambda K^4 M^{D_K}(D_K)})}{\log(\ffrac1\lambda)}- c\int_A  r^{\,D}(x)^2\,\textd x\biggr|\,=\,0
\end{equation}
\end{enumerate}
As the proof shows, we in fact just need to prove this for~$D$ and~$A$ being the triangles~$T$ and~$T_\delta$.
\end{remark}

\begin{proofsect}{Proof of Proposition~\ref{prop-3}}
We will relate the quantity in the statement to that in Lemma~\ref{lemma-5.6}. Pick~$f\colon T\to[0,\infty)$ with~$f(x)\le R$. For the upper bound we note
\begin{equation}
\begin{aligned}
\log E\bigl(\texte^{-\lambda M^T(f)}\bigr)
&=-\lambda\int_0^1\frac{E(M^T(f)\texte^{-s\lambda M^T(f)})}{E(\texte^{-s\lambda M^T(f)})}\,\textd s
\\
&\le-\lambda E\bigl(M^T(f)\texte^{-\lambda R M^T(T)}\bigr),
\end{aligned}
\end{equation}
where we first bounded the denominator in the integrand by one and then estimated the exponent in the numerator from above by $\lambda M^T(T)$ times the absolute maximum of~$f$. Invoking the upper bound in Lemma~\ref{lemma-5.6} for~$f$ replaced by $f\1_{T_\delta}$, we get the upper bound in the claim as soon as $\lambda_0\le\lambda_1/R$.

For the corresponding lower bound, pick~$a>0$ and write
\begin{equation}
\begin{aligned}
E\bigl(\texte^{-\lambda M^T(f)}\bigr)
&\ge E\bigl(\texte^{-\lambda M^T(f)-\lambda a M^T(T)}\bigr)
\\
&\ge E\bigl(\texte^{-\lambda a M^T(T)}\bigr)\exp\biggl\{-\lambda \frac{E(M^T(f)\texte^{-\lambda a M^T(T)})}{E(\texte^{-\lambda a M^T(T)})}
\biggr\}
\end{aligned}
\end{equation}
where the second bound follows by Jensen's inequality for exponential function. Since the logarithm of the first term on the right is order $a\lambda\log(\ffrac1\lambda)$, the claim again follows from Lemma~\ref{lemma-5.6} with $\epsilon$ replaced by, say, $\epsilon/2$ provided~$a$ is chosen sufficiently small and $\lambda_0\le\lambda_1 a$.
\end{proofsect}

Proposition~\ref{prop-1} is hardest to prove and is the sole reason why we work with triangular partitions. In order to explain this, let $D'$ denote the interior of $\bigcup_{i=1}^{m_K}\overline{D^i}$. Then
\begin{equation}
\Phi^{D,\wt D}\laweq\Phi^{D,D'}+\Phi^{D',\wt D}
\end{equation}
with the last two fields independent. Next  we recall that the field~$\Phi^{D',\wt D}$ can be thought of as a projection of the CGFF on~$D'$ onto the space of (piece-wise) harmonic functions on~$\wt D$. This space of harmonic functions contains the linear subspace $\HH^\triangle$ of all continuous piece-wise linear functions with zero boundary conditions on~$\partial D'$. Thus, $\Phi^{D,D'}$ further splits into a sum of independent Gaussian fields, $\Phi^\triangle+\Phi^{\perp}$, where $\Phi^\triangle$ is the projection of~$\Phi^{D',\wt D}$ on~$\HH^\triangle$ while~$\Phi^{\perp}$ is the corresponding projection onto the orthogonal complement thereof. Combining with the above formula, we thus have
\begin{equation}
\label{E:6.36ww}
\Phi^{D,\wt D} \laweq \Phi^{D,D'}+ \Phi^\triangle+\Phi^{\perp},
\end{equation}
with all three fields on the right independent. An important technical point for us is that, thanks to an observation that goes back to Sheffield~\cite{Sheffield-review}, $\Phi^\triangle$ can be represented as a linear extension of the DGFF on the triangular grid.

\smallskip
We start with some estimates on the variances of the various involved fields. The first one of these, $\Phi^{D,D'}$, can be dealt with using Proposition~\ref{prop-inner-apprx} so let us move directly to the ``rough'' field~$\Phi^\perp$. Here we get:

\begin{lemma}
\label{lemma-2}
For each~$\delta>0$ there is $c(\delta)<\infty$ such that
\begin{equation}
\sup_{K\ge1}\,\max_{i=1,\dots,n_K}\,\sup_{x\in D_\delta^i}\,\text{\rm Var}\bigl(\Phi^{\perp}(x)\bigr)\le c(\delta)
\end{equation}
\end{lemma}

The proof of this claim requires somewhat unpleasant computations and so we postpone it temporarily in order to keep attention focussed on the main part of the argument. We will nonetheless record an immediate corollary:

\begin{corollary}
Denote
\begin{equation}
\label{E:6.41ww}
\epsilon_R:=\sup_{K\ge1}\,\,\max_{i=1,\dots,n_K}\,\,\sup_{x\in D_\delta^i}\,
E\bigl(\texte^{\alpha\Phi^\perp(x)}\1_{\{\text{\rm osc}_{D^i_\delta}\Phi^{\perp}>R\}}\bigr).
\end{equation}
Then $\lim_{R\to\infty}\epsilon_R=0$.
\end{corollary}

\begin{proofsect}{Proof}
Thanks to the independence of $\Phi^{D,D'}+\Phi^\triangle$ and $\Phi^\perp$, the bound \eqref{E:6.22r} applies equally well to $\Phi^\perp$ instead of~$\Phi^{D,\wt D}$. This makes the Fernique criterion available which, in conjunction with the Borell-Tsirelson inequality (enabled by Lemma~\ref{lemma-2}), then shows that $\sup_{x\in D_\delta^i}\,|\Phi^\perp(x)|$ has Gaussian tails uniformly in~$K\ge1$ and $i=1,\dots,n_K$. Bounding both $\Phi^\perp(x)$ and $\text{\rm osc}_{D^i_\delta}\Phi^{\perp}$ in terms of $\sup_{x\in D_\delta^i}\,|\Phi^\perp(x)|$ and applying the exponential Chebyshev inequality, we get the claim.
\end{proofsect}

The above estimates will be particularly useful thanks to the following observation:

\begin{lemma}
\label{cor-1}
Let $Y\laweq\NN(0,100c(\delta))$ be independent of $\Phi^\triangle$. Then for any $x,y\in\bigcup_{i=1}^n D_\delta^i$,
\begin{equation}
\text{\rm Cov}\bigl(\Phi^\triangle(x)+Y,\Phi^\triangle(y)+Y\bigr)
\ge \text{\rm Cov}\bigl(\Phi^{D,\wt D}(x),\Phi^{D,\wt D}(y)\bigr).
\end{equation}
Similarly, setting $\Psi^\triangle(x):=\Phi^\triangle(x)+\Phi^\triangle(x_0^i+a)-\Phi^\triangle(x_0^i+b)$ whenever~$x\in D_\delta^i$, once~$\delta$ is small enough, we have
\begin{equation}
\text{\rm Cov}\bigl(\Psi^\triangle(x)+Y,\Psi^\triangle(y)+Y\bigr)
\ge \text{\rm Cov}\bigl(\Phi^{D,\wt D}(x),\Phi^{D,\wt D}(y)\bigr),
\end{equation}
for any~$x,y\in\bigcup_{i=1}^n D_\delta^i$ and any~$a,b$ with~$|a|,|b|\le \frac14K^{-1}$.
\end{lemma}

\begin{proofsect}{Proof}
first note that by the first part of Proposition~\ref{prop-inner-apprx}, $\text{Var}(\Phi^{D,D'}(x))\le c(\delta)$ uniformly in $x\in\bigcup_{i=1}^n D_\delta^i$ and $K\ge1$. Using also Lemma~\ref{lemma-2}, the bounds \twoeqref{EE:1.3}{E:6.22r} in Lemma~\ref{lemma-var-bd} apply with~$\Phi^{D,\wt D}$ replaced by~$\Phi^\triangle$ (and $c(\delta)$ multiplied by a numerical constant). 

Now assume, without loss of generality that~$Y$, is independent of $\Phi^\triangle$, $\Phi^\perp$ and $\Phi^{D,D'}$. Then \eqref{E:6.36ww} and Lemma~\ref{lemma-2} imply
\begin{equation}
\text{\rm Cov}\bigl(\Phi^\triangle(x)+Y,\Phi^\triangle(y)+Y\bigr)
\ge \text{\rm Cov}\bigl(\Phi^{D,\wt D}(x),\Phi^{D,\wt D}(y)\bigr)+2c(\delta).
\end{equation}
This immediately gives the first part of the claim. For the second part, we write the covariance using covariances of~$\Phi^\triangle$ and then apply \twoeqref{EE:1.3}{E:6.22r}.
\end{proofsect}

\begin{proofsect}{Proof of Proposition~\ref{prop-1}}
In light of the second part of Proposition~\ref{prop-inner-apprx}, for any~$\delta>0$,
\begin{equation}
\text{\rm osc}_{\bigcup_{i=1}^{n_K}\,D_\delta^i}\Phi^{D,D'}\,\underset{K\to\infty}\longrightarrow\,0
\end{equation}
in probability. This means that (at the cost of modifying~$R$ slightly) we can replace the event $\{\text{\rm osc}_{D^i_\delta}\Phi^{D,\wt D}>R\}$ by the union of the corresponding events (with~$R$ replaced by~$R/2$) for the fields $\Phi^\triangle$ and~$\Phi^\perp$. Invoking a union bound, we may then study the contributions of triangles where these events occur separately.

Let us first address the contribution coming from high oscillations of the field $\Phi^\perp$. Define
\begin{equation}
M_{K,R}^{\perp}:=\sum_{i=1}^{n}\int_{D^i_\delta}M^{D^i}(\textd x)\texte^{\alpha\Phi^{D,\wt D}(x)}\1_{\{\text{\rm osc}_{D^i_\delta}\Phi^{\perp}>R\}}.
\end{equation}
Then, by the fact that $\Phi^{D,D'}+\Phi^\perp$ and~$\Phi^\triangle$ are independent and using Jensen's inequality,
\begin{equation}
E\texte^{-\lambda M_{K,R}^{\perp}}
\ge E\biggl(\exp\Bigl\{-\lambda\sum_{i=1}^{n}\int_{D^i_\delta}M^{D^i}(\textd x)\texte^{\alpha\Phi^\triangle(x)}E\bigl(\texte^{\alpha\Phi^\perp(x)}\1_{\{\text{\rm osc}_{D^i_\delta}\Phi^{\perp}>R\}}\bigr)\Bigr\}\biggr).
\end{equation}
Recall that the inner expectation is bounded by the quantity~$\epsilon_R$ in \eqref{E:6.41ww}.

Let now~$Y\laweq\NN(0,100c(\delta))$ be independent of $\Phi^\perp$, $\Phi^\triangle$ and the $M^{D^i}$'s. The first claim in Lemma~\ref{cor-1} permits us to use Kahane's convexity inequality (see Proposition~\ref{lemma-Kahane}) to conclude
\begin{equation}
\begin{aligned}
E\texte^{-\lambda \texte^{\alpha Y}M_{K,R}^{\perp}}
&\ge E\biggl(\exp\Bigl\{-\lambda\epsilon_R\sum_{i=1}^{n}\int_{D^i_\delta}M^{D^i}(\textd x)\texte^{\alpha\Phi^\triangle(x)+\alpha Y}\Bigr\}\biggr).
\\
&\ge
E\biggl(\exp\Bigl\{-\lambda\,\epsilon_R\,\texte^{\frac32\alpha^2 c(\delta)}\sum_{i=1}^{n}\int_{D^i_\delta}M^{D^i}(\textd x)\texte^{\alpha\Phi^{D,\wt D}(x)}\Bigr\}\biggr)
\\
&\ge E\biggl(\exp\Bigl\{-\lambda\,\epsilon_R\,\texte^{\frac32\alpha^2 c(\delta)} M^D(D)\Bigr\}\biggr),
\end{aligned}
\end{equation}
where we also used that $\text{Var}(\Phi^\triangle(x))\le\text{Var}(\Phi^{D,\wt D}(x))$. The right-hand side is independent of~$K$ and it tends to zero as~$R\to\infty$ since~$M^D(D)<\infty$ a.s.\ and~$\epsilon_R\to0$. As~$Y$ is a fixed random variable independent of~$M_{K,R}^\perp$, it follows that $M_{K,R}^\perp\to0$ in probability as~$K\to\infty$ and~$R\to\infty$.

Next let us move to contributions coming from high oscillations of~$\Phi^\triangle$.
Since~$\Phi^\triangle$ is linear on every triangle, to control its oscillation it suffices to control its value at three distinct points. For~$\delta$ small enough, the distance of~$x_0^i$ to the boundary of~$D_\delta^i$ is at least~$\frac14 K^{-1}$. For each~$a,b$ with~$|a|,|b|\le \frac14 K^{-1}$ we have $x_0^i+a,x_0^i+b\in D_\delta^i$ and so it makes sense to define
\begin{equation}
M_{K,R}^\triangle(a,b):=\sum_{i=1}^{n}\int_{D^i_\delta}M^{D^i}(\textd x)\texte^{\alpha\Phi^{D,\wt D}(x)}\1_{\{\Phi^\triangle(x_0^i+a)-\Phi^\perp(x_0^i+b)>R\}}
\end{equation}
Using Jensen's inequality to eliminate the field $\Phi^{D,D'}+\Phi^\perp$ as before, we obtain
\begin{multline}
\quad
E\texte^{-\lambda M_{K,R}^\triangle(a,b)}
\\
\ge E\biggl(\exp\Bigl\{-\lambda\texte^{\frac12\alpha^2 c(\delta)}\sum_{i=1}^{n}\int_{D^i_\delta}M^{D^i}(\textd x)\texte^{\alpha\Phi^\triangle(x)}\1_{\{\Phi^\triangle(x_0^i+a)-\Phi^\perp(x_0^i+b)>R\}}\Bigr\}\biggr).
\end{multline}
The Markov inequality then shows
\begin{equation}
\quad
E\texte^{-\lambda M_{K,R}^\triangle(a,b)}
\ge E\biggl(\exp\Bigl\{-\lambda\texte^{\frac14c(\delta)-\alpha R}\sum_{i=1}^{n}\int_{D^i_\delta}M^{D^i}(\textd x)\texte^{\alpha\Psi(x)}\Bigr\}\biggr).
\end{equation}
where we recalled $\Psi^\triangle(x):=\Phi^\triangle(x)+\Phi^\triangle(x_0^i+a)-\Phi^\triangle(x_0^i+b)$ once $x\in D^i_\delta$.
Applying again Kahane's estimate from Proposition~\ref{lemma-Kahane} with the help of the second claim in Lemma~\ref{cor-1}, we thus get
\begin{equation}
\label{EE:1.19}
E\texte^{-\lambda \texte^{\alpha Y}M_{K,R}^\triangle(a,b)}
\ge E\biggl(\exp\Bigl\{-\lambda c'(\delta)\texte^{-\alpha R} M^D(D)\Bigr\}\biggr)
\end{equation}
for some~$c'(\delta)<\infty$. Notice that this gives us the convergence $M_{K,R}^\triangle(a,b)\to0$ in probability as~$K\to\infty$ and~$R\to\infty$ uniformly in the allowed~$a,b$.

Now let~$a_1,a_2,a_3$ be the points marking the third-roots of unity. The intersection of the event in \eqref{EE:1.2} with the event $\{M_{K,R/2}^\perp\le\ffrac\epsilon2\}$ is then contained in
\begin{equation}
\bigcup_{1\le i<j\le3}\bigl\{M_{K,R/20}^\triangle(a_i2^{-K-2},a_j2^{-K-2})>\ffrac\epsilon6\bigr\}.
\end{equation}
By \eqref{EE:1.19}, each of these three events has probability tending to zero as $K\to\infty$ and $R\to\infty$. Since also $M_{K,R}^\perp\to0$ in probability in this limit, the claim follows.
\end{proofsect}

\subsection{Variance of~$\Phi^\perp$}
For the proof of Proposition~\ref{prop-3} to be complete, it remains to check the uniform boundedness of the variance of the field~$\Phi^\perp$.

\begin{proofsect}{Proof of Lemma~\ref{lemma-2}}
The independence of $\Phi^\triangle$ and~$\Phi^\perp$ implies
\begin{equation}
\label{e:1.8}
\Var \bigl( \Phi^\perp(x) \bigr) = \Var \bigl( \Phi^{D, \wt{D}}(x) \bigr) - \Var \bigl( \Phi^\triangle(x) \bigr)-\Var \bigl( \Phi^{D,D'}(x) \bigr) 
\end{equation}
and so, in light of the first part of Proposition~\ref{prop-inner-apprx}, it suffices to show that the first two variances on the right differ only by a universal constant. Lemma~\ref{lemma-var-bd} gives
\begin{equation}
\Var \bigl( \Phi^{D, \wt{D}}(x) \bigr)\le g\log K+c_1
\end{equation}
for all $x\in\bigcup_{i=1}^{n_K}D_\delta^i$. It thus suffices to show that
\begin{equation}
\label{E:5.48}
\Var \bigl( \Phi^\triangle(x) \bigr)\ge g\log K-c_2
\end{equation}
for some~$c_2<\infty$ and all $x\in\bigcup_{i=1}^{n_K}D_\delta^i$.

Let $\T_K$ denote the set of the vertices in the triangles constituting~$\wt D$; we think of~$\T_K$ as a subgraph of the triangular lattice of mesh size~$K^{-1}$. Let $\partial^\star\T_K$ denote \emph{inner} boundary of~$\T_K$; these are the vertices that lie in~$\partial D'$. Abusing our notations slightly, let~$\Phi^K$ denote the restriction of~$\Phi^\triangle$ to the points in~$\T_K$. As observed in Sheffield~\cite[Section~4.2]{Sheffield-review}, $\Phi^K$ is distributed as~$3^{1/4}$ times the DGFF on~$\T_K$ with $0$ boundary values on~$\partial^\star\T_K$. More precisely, in our normalization (which differs from Sheffield by a factor $1/4$) the density of~$\Phi^K$ with respect to the Lebesgue measure is proportional to
\begin{equation}
	\exp \Bigl\{ - \tfrac18 \tfrac{1}{\sqrt{3}} \|\nabla\varphi\|_2^2 \Bigr\} =
	\exp \Bigl\{ -\tfrac18  \tfrac{6}{\sqrt{3}} \big<\varphi, -\tfrac16\Delta\varphi \big> \Bigr\} =
	\exp \Bigl\{ -\tfrac12  \tfrac{\sqrt{3}}{2} \big<\varphi, G_K^{-1} \varphi \big> \Bigr\}\,.
\end{equation}
Here $\langle f,g\rangle$, resp., $\Vert f\Vert_2$ denotes the inner product, resp., the norm in $\ell^2(\T_K)$ with respect to the counting measure, $\Delta$ is the discrete Laplace operator on~$\T_K$, and $\frac16\Delta$ is thus the generator of the simple random walk on $\T_K$. Finally, $G_K:=(-\frac16\Delta)^{-1}$ is the Green operator with $0$ boundary conditions on $\partial^\star\T_K$.

Based on the above reasoning, it follows that
\begin{equation}
\label{E:5.50}
\text{Cov} \big( \Phi^K(x), \Phi^K(y) \big) = (2/\sqrt{3}) G_K(x,y),\qquad x,y\in\T_K.
\end{equation}
Letting $\fraka(x)$ denote the potential kernel associated with the simple random walk on the triangular lattice~$\T$ of unit mesh size and writing $H^{\T_K}(x,\cdot)$ to denote the harmonic measure on~$\partial^\star\T_K$ for the random walk on~$\T_K$ started from~$x$, standard discrete potential theory shows
\begin{equation}
\label{e:1.11}
G_K(x,y) = - \fraka\bigl(K(x-y) \bigr)+
\sum_{z\in\partial^\star\T_K} H^{\T_K} (x,z)\fraka\bigl(K(z-y) \bigr)  \,.
\end{equation}
For the harmonic potential on the triangular lattice (or other lattice as well), Theorem~1 from Kozma and Schreiber~\cite{KozmaSchreiber04} gives the asymptotic
\begin{equation}
\label{e:1.12}
\fraka(w) \underset{|w|\to\infty}\sim \tau \log |w| \ , \quad  \text{where}\quad
\tau := \frac{\text{vol}(\T)}{\pi |\det \,\Sigma|^{1/2}} \,,
\end{equation}
for $|w|$ denoting the Euclidean norm of~$w$,
\begin{equation}
\text{vol}(\T) := \lim_{r \to \infty} \frac{\text{Leb}( B(0,r) )}{| B(0,r) \cap \T|}
\end{equation}
expressing the comparison of volumes of the Euclidean ball~$B(0,r)$ of radius~$r$ centered at~$0$ with the number of vertices of~$\T$ contained therein
and $\Sigma$ denoting the covariance matrix of the single step distribution of the random walk (regarded as a random variable in~$\R^2$).
A straightforward calculation now shows
\begin{equation}
\label{e:1.13}
	\text{vol}(\T) = \frac{\sqrt3}2
	\quad\text{and}\quad
	\det \,\Sigma = \frac14 
\end{equation}
and so $\tau = \sqrt{3}/\pi \,.$

Now let $u,v \in \T_K\cap \overline{D^i}$ for some $i\in \{1, \dots, n_K\}$. Then the lattice distance of these points from~$\partial^\star\T_K$ is at least order~$\delta K$ and so, as~$K\to\infty$, the asymptotic \eqref{e:1.12} can be used. Combining \eqref{E:5.50} with \eqref{e:1.12} and \eqref{e:1.13}, we thus get
\begin{equation}
\label{e:1.14}
\text{Cov}\bigl( \Phi^K(u), \Phi^K(v) \bigr) \geq \frac{\sqrt{3}}{\pi} \frac{2}{\sqrt{3}} \log K-c
\end{equation}
for some~$c=c(\delta)<\infty$.
If $\{u,v,w\}$ are the vertices in $\T_K\cap \overline{D^i}$ for some $i\in \{1, \dots, n_K\}$, then every $x \in D^i$ can be written as $x=\alpha_1 u+\alpha_2v+\alpha_3w$ for some nonnegative numbers $\alpha_1, \alpha_2, \alpha_3$ that add up to~$1$. The piece-wise linearity of~$\Phi^\triangle$ then implies that
$\Phi^\triangle(x)$ can be written as 
\begin{equation}
	\Phi^\triangle(x) = \Phi^\triangle \big(\alpha_1 u + \alpha_2 v + \alpha_3 w\big)
	= \alpha_1 \Phi^K(u) + \alpha_2 \Phi^K(u) + \alpha_3 \Phi^K(u).
\end{equation}
Invoking~\eqref{e:1.14}, we thus get
\begin{equation}
\label{e:1.16}
	\Var \big(\Phi^\triangle(x) \big) \geq 
		\big(g\log K - c\big) ( \alpha_1 + \alpha_2 + \alpha_3 )^2
		 = g \log K-c.
\end{equation}
This is exactly the desired bound \eqref{E:5.48}.
\end{proofsect}

\section{Conformal invariance}
\label{sec6}\noindent
The next item to prove is the conformal-transformation rule from Theorem~\ref{thm:3}. We begin by a simple consequence of Theorem~\ref{thm:5.1}. Recall that a rotation is a map $f\colon\C\to\C$ of the form $f(z):=\lambda z$ for~$\lambda\in\C$ with~$|\lambda|=1$.

\begin{lemma}
\label{lemma-6.1}
The family of laws of $\{Z^D\colon D\in\mathfrak D\}$ is invariant under rotations. More precisely, for all~$\lambda\in\C$ with~$|\lambda|=1$ and all $D\in\mathfrak D$, and denoting $f(z):=\lambda z$,
\begin{equation}
(Z^{f(D)}\circ f)\,(\textd x)\laweq Z^D(\textd x).
\end{equation}

\end{lemma}

\begin{proofsect}{Proof}
As observed already in the proof of Theorem~\ref{thm:1.6}, the family $\{Z^D\colon D\in\mathfrak D\}$ obeys conditions (0-4) thereof and so, by Theorem~\ref{thm:5.1}, $Z^D$ is a weak limit of the measures in \eqref{E:5.2}. Thanks to the representation based on Poisson kernel and Euclidean distance, the function $\psi^D$ is clearly invariant under rotations of~$D$ and so is the covariance~$C^{D,\wt D}$ of the field~$\Phi^{D,\wt D}$, provided we rotate the underlying triangular grid --- which determines~$\wt D$ from~$D$ --- along with~$D$. It follows that also the law of the measure in \eqref{E:5.2} is rotation invariant and thus so must be~$Z^D$.
\end{proofsect}

We will now state general conditions on a family of random measures that ensures validity of the conformal-transformation rule \eqref{E:1.10a}. 

\begin{theorem}
\label{thm:6.1}
Suppose~$\{M^D\colon D\in\mathfrak D\}$ is a family of random Borel measures that obey:
\settowidth{\leftmargini}{(11)}
\begin{enumerate}
\item[(0)] (support and total mass restriction) $M^D$ is concentrated on~$D$ and $M^D(D)<\infty$ a.s.
\item[(1)] (stochastic absolute continuity) $P(M^D(A)>0)=0$ for any Borel $A\subseteq D$ with $\leb(A)=0$,
\item[(2)] (shift, dilation and rotation invariance) for $a,\lambda\in\C$,
\begin{equation}
M^{a+\lambda D}(a+\lambda\textd x)\laweq |\lambda|^4 M^D(\textd x),
\end{equation}
\item[(3)] (Gibbs-Markov property) if $D,\widetilde D\in\mathfrak D$ are disjoint then, for $M^D$ and $M^{\wt D}$ on the right are regarded as independent,
\begin{equation}
M^{D\cup\wt D}(\textd x)\laweq M^D(\textd x)+M^{\wt D}(\textd x),
\end{equation}
while if $D,\widetilde D\in\mathfrak D$ obey $\widetilde D\subseteq D$ and $\leb(D\smallsetminus\widetilde D)=0$, then
\begin{equation}
M^D(\textd x) \laweq M^{\wt{ D}}(\textd x)\, \texte^{\alpha  \Phi^{ D, \wt{ D}}(x)} \,,
\end{equation}
where $\Phi^{ D, \wt{ D}}$ and $M^{\wt{ D}}$ on the right-hand side are independent.
\end{enumerate}
Then for any conformal bijection $f\colon D\to D'$ with $D,D'\in\mathfrak D$
\begin{equation}
\label{E:6.2}
(M^{D'}\circ f)(\textd x)\laweq |f'(x)|^4 M^D(\textd x).
\end{equation}
\end{theorem}

We can apply this result immediately to the family of measures constructed in Theorem~\ref{thm:1}:

\begin{proofsect}{Proof of Theorem~\ref{thm:3}}
By Theorem~\ref{thm:1} and Lemma~\ref{lemma-6.1}, the family of measures $\{Z^D\colon D\in\mathfrak D\}$ obeys the conditions of Theorem~\ref{thm:6.1} and so we get \eqref{E:1.10a} from \eqref{E:6.2}. The relation \eqref{E:1.17c} is then a consequence of the fact that, for any conformal bijection $f\colon D\to\wt D$ between simply connected domains~$D$ and~$D'$, we have $|f'(x)|=\text{rad}_{\wt D}(f(x))/\text{rad}_D(x)$.
\end{proofsect}

The upshot of Theorem~\ref{thm:6.1} is that the Gibbs-Markov property is sufficient to turn invariance under ``rigid'' conformal maps --- shifts, rotations and dilations --- into invariance under all conformal maps. This is no big surprise as general conformal maps can be thought of as shifts, rotations and dilations locally. However,  full details still require some work. Fortunately, a good part of this work has already been done in Propositions~\ref{prop-slither}--\ref{prop-1} in Section~\ref{sec5}.

\smallskip
The crux of the proof will be to show the following one-way bound:

\begin{proposition}
\label{prop-6.2}
Let~$\{M^D\colon D\in\mathfrak D\}$ be as in Theorem~\ref{thm:6.1}. Pick $D,D'\in\mathfrak D$ and let~$f\colon D\to D'$ be a conformal bijection. Then for any bounded and continuous function $u\colon D'\to[0,\infty)$,
\begin{equation}
\label{E:6.4}
E\bigl(\texte^{-M^D(u\circ f)}\bigr)
\ge E\bigl(\texte^{-M^{D'}(u|f'\circ f^{-1}|^{-4})}\bigr).
\end{equation}
\end{proposition}

Indeed, this is more than enough to give us the desired result:

\begin{proofsect}{Proof of Theorem~\ref{thm:6.1}}
Just iterate \eqref{E:6.4} twice to see that equality must hold for all bounded and continuous~$u$.
\end{proofsect}

To get Proposition~\ref{prop-6.2}, we begin with the observation that the law of the ``binding'' field~$\Phi^{D,\wt D}$ is invariant under conformal maps. More precisely, we have:

\begin{lemma}
Let $D,\wt D\in\mathfrak D$ obey $\wt D\subseteq D$. Let~$f$ be a conformal bijection of~$D$ onto~$f(D)\in\mathfrak D$. Then we have
\begin{equation}
C^{\,D,\wt D}(x,y)=C^{\,f(D),f(\wt D)}\bigl(\,f(x),f(y)\bigr),\qquad x,y\in\wt D.
\end{equation}
In particular, $\Phi^{f(D),f(\wt D)}\circ f\laweq\Phi^{D,\wt D}$.
\end{lemma}

\begin{proofsect}{Proof}
By \eqref{E:1.7}, $C^{D,\wt D}$ is the difference of the Green functions in~$D$ and~$\wt D$. The claim then follows from the fact that $G^{\,f(D)}(f(x),f(y))=G^D(x,y)$ for any~$x,y\in D$. (This is itself proved, e.g., by noting that $y\mapsto G^D(x,y)$ is a fundamental solution of the Poisson equation in~$D$, i.e., it solves the equation $\Delta\varphi(y)=\delta_x(y)$ in~$D$ with zero boundary conditions on~$\partial D$. Under~$f$, this equation transforms to the corresponding equation in~$f(D)$.)
\end{proofsect}

\begin{proofsect}{Proof of Proposition~\ref{prop-6.2}}
Let~$D\in\mathfrak D$ and let~$\wt D$ be the union of triangles of side~$K^{-1}$ as detailed in Section~\ref{sec:5.1}. Then, as observed before, the measure $M^D_{K,R,\delta}$ defined in \eqref{E:5.10} tends, in the limit $K\to\infty$, $R\to\infty$ and~$\delta\downarrow0$, to~$M^D$. Suppose now that $f\colon D\to D'$ is a conformal bijection of~$D$ onto a domain~$D'\in\mathfrak D$ and pick a test function $u\colon D'\to[0,\infty)$ which we assume to be bounded and continuous. The composition $u\circ f$ induces a corresponding test function on~$D$. Then
\begin{equation}
M^D_{K,R,\delta}(u\circ f)=\sum_{i=1}^{n_K}\1_{A_{K,R}^i}\,M^{D^i}\bigl((u\circ f)\,\1_{D^i_\delta}\,\texte^{\alpha\Phi^{D,\wt D}}\bigr).
\end{equation}
Define
\begin{equation}
f_{i,K}(z):=f(x_0^i)+(1+K^{-1/2})f'(x_0^i)(z-x_0^i)
\end{equation}
and note that this is a slightly scaled-up linearization of~$f$ on~$D^i$.
Using the inverse of this function to change variables from~$z\in D$ to~$z\in D'$ we get
\begin{equation}
\label{E:6.7}
M^D_{K,R,\delta}(u\circ f)=\sum_{i=1}^{n_K}\1_{A_{K,R}^i}\,M^{D^i}\circ f_{i,K}^{-1}\Bigl(\,(u\circ f\circ f_{i,K}^{-1})\,(\1_{D^i_\delta}\circ f_{i,K}^{-1})\,\texte^{\alpha\Phi^{D,\wt D}\circ f_{i,K}^{-1}}\Bigr).
\end{equation}
But the fact that~$D^i$ are at least~$\delta$ away from the boundary of~$D$ implies a uniform bound on~$f''$ there and, as a consequence, the estimate
\begin{equation}
\label{E:7.11cc}
f-f_{i,K}=O(K^{-3/2}),\qquad\text{on }D^i,
\end{equation}
with the implicit constant uniform in~$i=1,\dots,n_K$. For any fixed~$\epsilon>0$, the following is true as soon as~$K$ is sufficiently large and~$c\in(0,\infty)$ is a constant independent of~$K$:
\begin{enumerate}
\item[(a)] $f_{i,K}(D^i)\supseteq f(D^i)$,
\item[(b)] for $\wt D^i_\delta:=\{z\in f(D^i)\colon\dist(z,f(D^i)^\cc)>(\ffrac\delta2)|f'(x_0^i)|/K\}$,
\begin{equation}
\1_{D^i_\delta}\circ f_{i,K}^{-1}\le 1_{\wt D^i_\delta},
\end{equation}
\item[(c)] whenever $A_{K,R}^i$ occurs,
\begin{equation}
\texte^{\alpha\Phi^{D,\wt D}\circ f_{i,K}^{-1}}\le
\texte^{\alpha\Phi^{D,\wt D}\circ f^{-1}+ cK^{-3/2}},\qquad\text{on }\wt D^i_\delta,
\end{equation}
\item[(d)] on~$f(D^i)$,
\begin{equation}
u\circ f\circ f_{i,K}^{-1}\le u+\epsilon.
\end{equation}
\item[(e)] for each $\epsilon>0$ and~$K$ sufficiently large, 
\begin{equation}
\label{E:6.15}
\texte^{\frac12\alpha^2\Var(\Phi^{f_{i,K}(D^i),f(D^i)}(x))}\le 1+\epsilon
\end{equation}
holds for all $x\in D_\delta^i$.
\end{enumerate}
Here (a-b) are consequences of the definition of $f_{i,K}$, (c) uses \eqref{E:7.11cc} the restriction on Lipschitz property of~$\Phi^{D,\wt D}$ imposed by event $A_{K,R}^i$ while (d) is in turn based on the uniform continuity of the test function~$u$. To get (e) we again invoke the ``closeness'' of domains $f_{i,K}(D^i)$ and $f(D^i)$ along with the invariance under scaling of both domains and Proposition~\ref{prop-inner-apprx}.

Invoking (b-d) in \eqref{E:6.7} gives
\begin{equation}
\label{E:6.12}
M^D_{K,R,\delta}(u\circ f)\le\texte^{c K^{-1/2}}\,\sum_{i=1}^{n_K}\1_{A_{K,R}^i}\,(M^{D^i}\circ f_{i,K}^{-1})\Bigl(\,(u+\epsilon)\,\1_{\wt D^i_\delta}\,\texte^{\alpha\Phi^{D,\wt D}\circ f^{-1}}\Bigr)
\end{equation}
while assumptions (2-3) of Theorem~\ref{thm:6.1} imply
\begin{equation}
(M^{D^i}\circ f_{i,K}^{-1})(w)\laweq |f'(x_0^i)|^{-4}\,M^{f(D^i)}\bigl(w\,\texte^{\alpha\Phi^{f_{i,K}(D^i),f(D^i)}}\bigr)
\end{equation}
for any measurable function~$w\colon f_{i,K}(D^i)\to[0,\infty)$ with $\supp(w)\subseteq f(D^i)$. (Here $\Phi^{f_{i,K}(D^i),f(D^i)}$ is independent of~$\Phi^{D,\wt D}$.) 
Using this inside the Laplace transform and applying Jensen's inequality, we get the estimate
\begin{equation}
\label{E:6.14}
E\bigl(\texte^{-(M^{D^i}\circ f_{i,K}^{-1})(w)}\bigr)
\ge E\biggl(\exp\Bigl\{-|f'(x_0^i)|^{-4}\,M^{f(D^i)}\bigl(w\,\texte^{\frac12\alpha^2\Var(\Phi^{f_{i,K}(D^i),f(D^i)})}\bigr)\Bigr\}\biggr).
\end{equation}

Now we are ready to put bits and pieces together.  First, since $\{M^{D^i}\colon i=1,\dots,n_K\}$ are independent, and all of them are independent of~$\Phi^{D,\wt D}$, \eqref{E:6.12}, \eqref{E:6.14} and \eqref{E:6.15} yield
\begin{multline}
E\bigl(\texte^{-M^D_{K,R,\delta}(u\circ f)}\bigr)
\\
\ge E\biggl(\,\exp\Bigl\{-\texte^{c K^{-1/2}}(1+\epsilon)\sum_{i=1}^{n_K}M^{f(D^i)}\bigl((u+\epsilon)(|f'\circ f^{-1})|^{-4}+\epsilon)\texte^{\alpha\Phi^{D,\wt D}\circ f^{-1}}\bigl)\Bigr\}\biggr),
\end{multline}
where we regard $\{M^{f(D^i)}\colon i=1,\dots,n_K\}$ as independent and independent of~$\Phi^{D,\wt D}$.
Notice that we also bounded $|f'(x_0^i)|^{-4}\le |(f'\circ f^{-1})(x)|^{-4}+\epsilon$ for all~$x\in f(D^i)$ and dropped the indicators of events~$A_{K,R}^i$ and sets~$\wt D_\delta^i$. But Lemma~\ref{lemma-6.1} implies that
\begin{equation}
\Phi^{D,\wt D}\circ f^{-1}\laweq \Phi^{D',\wt D'}
\end{equation}
and so the Gibbs-Markov property yields
\begin{equation}
E\bigl(\texte^{-M^D_{K,R,\delta}(u\circ f)}\bigr)
\ge E\bigl(\texte^{-\texte^{c K^{-1/2}}(1+\epsilon)\,M^{D'}((u+\epsilon)(|f'\circ f^{-1}|^{-4}+\epsilon))}\bigr).
\end{equation}
Taking $K\to\infty$, $\epsilon\downarrow0$, $R\to\infty$ and~$\delta\downarrow0$ (in this order) then gives the claim.
\end{proofsect}

\section{Connection to Liouville Quantum Gravity}
\label{sec8}\noindent
The goal of this section is to identify our $Z^D$-measures with the critical Liouville Quantum Gravity defined, somewhat implicitly, by the limit \eqref{E:2.25ua}.  We will do this by checking that the version of critical LQG defined by \twoeqref{E:2.22}{E:2.25ua} obeys the conditions of Theorem~\ref{thm:1.6}. The key tool, and the reason for working with Seneta-Heyde norming, is Kahane's inequality (cf Proposition~\ref{lemma-Kahane}). 

\subsection{Proof modulo a key proposition}
Let $M^D_\infty$ be the measure from \eqref{E:2.25ua}. 
It is straightforward to check that $M^D_\infty$ obeys conditions (0,1,2) of Theorem~\ref{thm:1.6}, as well as the independence part of the Gibbs-Markov property~\eqref{E:2.18ua}, so the key is to verify the part of the Gibbs-Markov property \eqref{E:2.19ua} dealing with restrictions to a subdomain and the Laplace-transform tail in \eqref{E:1.25}. The former is the content of:

\begin{lemma}[Gibbs-Markov for LQG measure]
\label{lemma-GM-M}
Suppose $\wt D,D\in\mathfrak D$ are such that~$\wt D\subseteq D$ and $\leb(D\smallsetminus\wt D)=0$. Then for each~$f\in C_\cc(\wt D)$,
\begin{equation}
\label{E:8.1}
E\bigl(\texte^{-\langle M^D_\infty,\,f\rangle}\bigr)
=E\bigl(\texte^{-\langle M^{\wt D}_\infty,\,f\texte^{\alpha\Phi^{D,\wt D}}\rangle}\bigr)
\end{equation}
where $M^{\wt D}_\infty$ and~$\Phi^{D,\wt D}$ on the right are regarded as independent.
\end{lemma}

\begin{proofsect}{Proof}
The Seneta-Heyde norming allows us to use the Kahane convexity inequality from Proposition~\ref{lemma-Kahane}. Fix $f\in C_\cc(\wt D)$ and denote
\begin{equation}
\label{E:8.2}
G^D_t(x,y):=\Cov\bigl(\varphi_t(x),\varphi_t(y)\bigr)=\int_{\texte^{-2t}}^\infty p^D_s(x,y)\textd s\,.
\end{equation}
A routine coupling argument for Brownian motion shows that $p_s^D(x,y)-p_s^{\wt D}(x,y)$ is integrable for small~$s$ uniformly in~$x,y\in\supp(f)$. This shows
\begin{equation}
\epsilon(t):=\sup_{x,y\in\supp(f)}\bigl|G^D_t(x,y)-G^{\wt D}_t(x,y)-C^{D,\wt D}(x,y)\bigr|\,\underset{t\to\infty}\longrightarrow\,0\,.
\end{equation}
Kahane's inequality along with  $\rad_D(x)^2=\rad_{\wt D}(x)^2\texte^{\frac12 \alpha^2 C^{D,\wt D}(x,x)}$  then implies, for~$\epsilon:=\epsilon(t)$,
\begin{equation}
E\bigl(\texte^{-\texte^{Y_\epsilon-\epsilon/2}\langle M^D_\infty,\,f\rangle}\bigr)
\ge E\bigl(\texte^{-\langle M^{\wt D}_\infty,\,f\texte^{\alpha\Phi^{D,\wt D}}\rangle}\bigr)\,,
\end{equation}
where~$Y_\epsilon\laweq\NN(0,\epsilon)$ is independent of~$M^D_\infty$ on the left-hand side. As~$\epsilon(t)\to0$ as~$t\to\infty$, we get ``$\ge$'' in \eqref{E:8.1}. The other inequality is completely analogous so we omit it.
\end{proofsect}

The main task is thus the proof of \eqref{E:1.25}. 
Since $u\mapsto u\texte^{-\lambda u}$ is bounded on~$(0,\infty)$ for any~$\lambda>0$, we can work with $M_t^D$ instead of~$M^D$ and, by \eqref{E:2.25ua}, even restrict~$t$ to integers (which we will denote as~$n$) provided we take these to infinity before taking~$\lambda$ to zero. The definition \eqref{E:2.22} along with Tonelli's Theorem yield,  for any measurable $A\subseteq D$,
\begin{equation}
\label{8E:8.5}
E\bigl[M_t^D(A)\texte^{-\lambda M_t^D(D)}\bigr]=\int_A\sqrt{t}\,E\Bigl[\texte^{\alpha\varphi_t(x)-\frac12\alpha^2\Var(\varphi_t(x))}\texte^{-\lambda M_t^D(D)}\Bigr]\,\rad_D(x)^2\,\textd x\,.
\end{equation}
This leads to the next important point, which also explains our reliance on the Seneta-Heyde normalization:

\begin{lemma}[Girsanov argument]
\label{lemma-Girsanov}
For each~$t>0$ and any~$x\in D$,
\begin{equation}
\label{8E:8.6}
E\Bigl[\texte^{\alpha\varphi_t(x)-\frac12\alpha^2\Var(\varphi_t(x))}\texte^{-\lambda M_t^D(D)}\Bigr]
=E\biggl[\exp\Bigl\{-\lambda\int_D\texte^{\alpha^2G^D_t(x,y)}M_t^D(\textd y)\Bigr\}\biggr].
\end{equation}
\end{lemma}

\begin{proofsect}{Proof}
The proof is based on a Girsanov-type argument. Fix~$x\in D$, recall \eqref{E:8.2} and set
\begin{equation}
\chi_t(y):=\varphi_t(y)-\varphi_t(x)\frac{G_t(x,y)}{G_t(x,x)}\,.
\end{equation}
Then~$\Cov(\chi_t(y),\varphi_t(x))=0$ for all~$y$ and so~$\chi_t$ is (being multivariate Gaussian) independent of~$\varphi_t(x)$. Writing $\wt M_t^D$ for the object in \eqref{E:2.22} with~$\varphi_t$ replaced by~$\chi_t$, we have
\begin{equation}
\label{8E:8.8}
M_t^D(\textd y)=\exp\biggl\{\alpha\varphi_t(x)\frac{G_t(x,y)}{G_t(x,x)}-\frac{\alpha^2}2\frac{G_t(x,y)^2}{G_t(x,x)}\biggr\}\wt M_t^D(\textd y).
\end{equation}
The left hand side of \eqref{8E:8.6} is the expectation of $\texte^{-\lambda M_t^D(D)}$ under the measure
\begin{equation}
P_{t,x}(A):=E\bigl(1_A\texte^{\alpha\varphi_t(x)-\frac12\alpha^2\Var(\varphi_t(x))}\bigr).
\end{equation}
As is now readily verified
\begin{equation}
(\varphi_t(x),\chi_t\bigr) \text{ under }P_{t,x}\,\,\laweq\,\,\bigl(\varphi_t(x)+\alpha G_t(x,x),\chi_t\bigr)\text{ under }P.
\end{equation}
Using this in \eqref{8E:8.8}, we get
\begin{equation}
M_t^D(\textd y)\text{ under }P_{t,x}\,\,\laweq\,\,\texte^{\alpha^2 G_t^D(x,y)}M_t^D(\textd y)\text{ under }P.
\end{equation}
This now readily implies the claim.
\end{proofsect}

The benefit of the rewrite \eqref{8E:8.6} is that the resulting expression is amenable to applications of Kahane's convexity inequality. This underlies the proof of:

\begin{proposition}
\label{prop-8.3}
For any~$D\in\mathfrak D$ and any open~$A$ with~$\overline A\subset D$,
\begin{equation}
\label{8E:1.7}
\lim_{\lambda\downarrow0}\,\limsup_{n\to\infty}\,
\sup_{x\in A}\,\Biggl|\,\frac{\sqrt n}{\log(1/\lambda)}\,E\biggl[\exp\Bigl\{-\lambda\int_D\texte^{\alpha^2G^D_n(x,y)}M_n^D(\textd y)\Bigr\}\biggr] -\frac1{\sqrt{2\pi}}\Biggr|\,=\,0.
\end{equation}
\end{proposition}

Deferring the proof temporarily, we use this to give:

\begin{proofsect}{Proof of Theorem~\ref{thm:LQG} from Proposition~\ref{prop-8.3}}
We have already verified properties (0-3) of Theorem~\ref{thm:1.6}. Property~(4),  with $c=\frac1{\sqrt{2\pi}}$,  then follows by combining \eqref{8E:8.5} with Lemma~\ref{lemma-Girsanov} and invoking Proposition~\ref{prop-8.3} along with the Bounded Convergence Theorem. 
\end{proofsect}

\subsection{Proof of Proposition~\ref{prop-8.3}, upper bound}
We are left to prove the asymptotic formula \eqref{8E:1.7} which we will do by proving separately upper and lower bounds on the expectation. Thanks to the scaling property in Theorem~\ref{thm:1.6}(3) and the assumption that~$A$ does not reach the boundary of~$D$ we may assume that~$D$ is so large that it compactly embeds the open unit (Euclidean) disc centered at~$x$. In the upper bound we simply restrict the integral to that disc. Noting that, on this disc,  $\rad_D$  is larger than some~$\delta>0$, by absorbing $\delta$ into~$\lambda$ it then suffices to prove the upper bound with $M_t^D$ replaced by the same measure without the  conformal radius  term; namely
\begin{equation}
\label{E:2.22-new}
\wh M_t^D(\textd x):= 1_D(x)\,\sqrt{t}\,\texte^{\alpha\varphi_t(x)-\frac12\alpha^2\Var(\varphi_t(x))}\,\textd x\,.
\end{equation}

Consider the collection of annuli 
\begin{equation}
A_k:=\{x\in\C\colon \texte^{-k}\le |x|<\texte^{-k+1}\},\quad k=1,\dots, n\,,
\end{equation}
where~$n$ will play the role of variable~$t$ above.
For~$\delta>0$ small we also set
\begin{equation}
A_k^\delta:=\bigl\{x\in A_k\colon\dist(x,A_k^\cc)>\delta\texte^{-k}\},
\quad k=1,\dots,n.
\end{equation}
In order to use Kahane's inequality, we observe:

\begin{lemma}
\label{lemma-8.4}
For each~$D$ as above and each~$\delta>0$ small there is~$\hat c=\hat c(\delta,D)\in(0,\infty)$ such that for all~$n\in\N$, all $k,\ell\in\{1,\dots,n\}$, all~$x\in A_k^\delta$ and all~$y\in A_\ell^\delta$, we have
\begin{equation}
\label{8E:1.10}
G^D_n(x,y)\le\sum_{m=1}^n G_n^{A_m}(x,y)+g(k\wedge\ell)+\hat c.
\end{equation}
\end{lemma}

\begin{proofsect}{Proof}
When $\ell=k$ we use $G^D_n(x,y)=-g\log(|x-y|\vee\texte^{-n})+O(1)$ and
\begin{equation}
G_n^{A_k}(x,y)=-g\log(|x-y|\vee\texte^{-n})- gk+O(1)\,,
\end{equation}
where the factor~$-gk$ marks the~$k$-dependence to the ``harmonic correction'' in the formula for the Green function. As $G_n^{A_m}(x,y)=0$ for~$m\ne k$, we get the desired inequality.

When $\ell<k$, we have $G_n^{A_m}(x,y)=0$ for all~$m=0,\dots,n$ while
\begin{equation}
G_n^D(x,y)=g\ell+O(1)
\end{equation}
because $|x-y|$ is of order~$\texte^{-\ell}$. The claim follows in this case as well.
\end{proofsect}

Consider a random walk $\{S_k\colon k\ge0\}$ with $S_k-S_{k-1}$ i.i.d.~$\NN(0,g)$ for~$k\ge1$ and $S_0=\NN(0,\hat c)$ for~$\hat c$ the constant in \eqref{8E:1.10} (and independent of the increments of~$S$). Note that then
\begin{equation}
\Cov(S_k,S_\ell)=g(k\wedge\ell)+\hat c\,.
\end{equation}
In light of \eqref{8E:1.10}, Kahane's inequality then bounds the expectation in \eqref{8E:1.7} by
\begin{equation}
\label{8E:1.14}
\begin{aligned}
E\biggl[\exp\Bigl\{-\lambda\int_D\texte^{\alpha^2G^D_n(0,x)}&\wh M_n^D(\textd x)\Bigr\}\biggr]
\le E\biggl[\exp\Bigl\{-\lambda\sum_{k=1}^{n/2}\int_{A_k^\delta}
\texte^{\alpha^2G^D_n(0,x)}\wh M_n^D(\textd x)\Bigr\}\biggr]
\\
&\le
E\biggl[\exp\Bigl\{-\lambda\sum_{k=1}^{n/2}\int_{A_k^\delta}
\texte^{\alpha^2G^D_n(0,x)}\texte^{\alpha S_k-\frac12\alpha^2 (g k+\hat c)}
\wh M_n^{A_k}(\textd x)\Bigr\}\biggr]
\\
&\le
E\biggl[\exp\Bigl\{-c\lambda\sum_{k=1}^{n/2}\texte^{\alpha S_k} \texte^{2k}
\wh M_n^{A_k}(A_k^\delta)\Bigr\}\biggr],
\end{aligned}
\end{equation}
where $c$ is a positive constant and the measures $\{\wh M_n^{A_k}\colon k=1,\dots,n\}$ are independent of each other and of the random walk~$S$. The constant~$c$ arises from the observation that, since $G^D_n(0,x)=g\log(|x|\vee\texte^{-n})+O(1)$ and $\alpha^2g=4$,
\begin{equation}
\texte^{\alpha^2G^D_n(0,x)-\frac12\alpha^2(gk+\hat c)}\ge c\texte^{2k},\quad x\in A_k^\delta,
\end{equation}
holds for all $n\in\N$ and all~$k=1,\dots,n$.

Abbreviating, for the objects as above, 
\begin{equation}
\label{8E:1.16}
X_{n,k}:=\texte^{2k}
\wh M_n^{A_k}(A_k^\delta),\quad 1\le k\le n,
\end{equation}
a key point to observe in the last line of \eqref{8E:1.14} is that
\begin{equation}
\label{8E:8.18}
X_{n,k}\,\,\overset{\text{law}}=\,\sqrt{\frac{n}{n-k+1}}\,\,X_{n-k+1,1}, \quad 1\le k< n.
\end{equation}
We will need the following bound from \cite[Corollary~6]{DRSV2} (see also Remark~29 there) that
\begin{equation}
\label{8E:8.19}
\sup_{n\ge1}\,\,E\bigl(1/X_{n,1}\bigr)<\infty\,.
\end{equation}
As most of~$X_{n,k}$ will thus be bounded away from zero, to prevent the sum $\sum_{k=1}^{n/2} \texte^{\alpha S_k} X_{n,k}$ in \eqref{8E:1.14} from exploding as~$n\to\infty$, the random walk~$S_k$ has to stay below $\alpha^{-1}\log(1/\lambda)$, and in fact diverge to~$-\infty$, as~$k$ increases. This is articulated precisely in: 

\begin{lemma}
\label{8lemma-1.4}
Recall the notation in \eqref{8E:1.16} and abbreviate $S^\star_n:=\max_{k\le n}S_k$.
For any~$\epsilon>0$,
\begin{equation}
\label{8E:8.21}
\limsup_{\lambda\downarrow0}\,
\limsup_{n\to\infty}\,\,\frac{\sqrt n}{\log(1/\lambda)}\,E\biggl[1_{\{S^\star_n>(1+\epsilon)\alpha^{-1}\log(1/\lambda)\}}\exp\Bigl\{-c\lambda\sum_{k=1}^{n}\texte^{\alpha S_k} X_{n,k}\Bigr\}\biggr]=0.
\end{equation}
\end{lemma}

\begin{proofsect}{Proof}
Abbreviate $q_\lambda:=(1+\epsilon)\alpha^{-1}\log(1/\lambda)$ and let~$K_n\in\{0,\dots,n\}$ be the smallest index where~$S$ equals~$S^\star_n$. Neglecting, in the sum in \eqref{8E:8.21}, all terms but that of~$k=K_n$, we get
\begin{equation}
E\biggl[1_{\{S^\star_n>q_\lambda\}}\exp\Bigl\{-c\lambda\sum_{k=1}^n\texte^{\alpha S_k} X_{n,k}\Bigr\}\biggr]
\le E\Bigl(\texte^{-c\lambda \texte^{\alpha S^\star_n}X_{n,K_n}}1_{\{S^\star_n>q_\lambda\}}\Bigr).
\end{equation}
Next we decompose the expectation according to which of the integer intervals~$S^\star_n/q_\lambda$ belongs to and then use that for $S^\star_n\ge\ell q_\lambda$ with~$\ell\ge1$ we have $\lambda\texte^{\alpha S^\star_n}\ge\lambda^{1-(1+\epsilon)\ell}\ge\lambda^{-\epsilon\ell}$ to get
\begin{equation}
\label{8E:8.27}
\begin{aligned}
E\Bigl(\texte^{-c\lambda \texte^{\alpha S^\star_n}X_{n,K_n}}1_{\{S^\star_n>q_\lambda\}}\Bigr)
&=\sum_{\ell\ge1}E\Bigl(\texte^{-c\lambda \texte^{\alpha S^\star_n}X_{n,K_n}}1_{\{S^\star_n/q_\lambda\in(\ell,\ell+1]\}}\Bigr)
\\
&\le\sum_{\ell\ge1}E\Bigl(\texte^{-c\lambda^{-\ell\epsilon}X_{n,K_n}}1_{\{S^\star_n\le(\ell+1)q_\lambda\}}\Bigr)\,.
\end{aligned}
\end{equation}
We split the last expectation according to whether~$X_{n,K_n}\le\lambda^{-\ell\epsilon}$ or not. Since the~$X_{n,k}$'s are independent of~$S$, we can do this under conditioning on~$S$; \twoeqref{8E:8.18}{8E:8.19} along with the Markov inequality then yield
\begin{equation}
\label{8E:8.28}
E\Bigl(\texte^{-c\lambda^{-\ell\epsilon}X_{n,K_n}}1_{\{S^\star_n\le(\ell+1)q_\lambda\}}\Bigr)
\le\Bigl(\texte^{-c\lambda^{-\ell\epsilon/2}}+C\lambda^{-\ell\epsilon/2}\Bigr)P\bigl(S^\star_n\le(\ell+1)q_\lambda\bigr),
\end{equation}
where~$C$ denotes the supremum in \eqref{8E:8.19}. Standard ballot-problem estimates show
\begin{equation}
\label{8E:8.32}
P\bigl(S^\star_n\le(\ell+1)q_\lambda\bigr)\le \tilde c(\ell+1)\frac{q_\lambda}{\sqrt n}.
\end{equation}
The prefactors on the right-hand side of \eqref{8E:8.28} provide sufficient decay to control the term~$(\ell+1)$ on the right; plugging the resulting bound in \eqref{8E:8.27} then readily yields the claim.
\end{proofsect}

\begin{proofsect}{Proof of upper bound in \eqref{8E:1.7}}
Re-using our notation $S^\star_n:=\max_{k\le n}S_k$, the proof is based on the sharp asymptotic version of \eqref{8E:8.32}:
\begin{equation}
P(S_n^\star\le t)\le\frac t{\sqrt n}\bigl(1+o(1)\bigr),
\end{equation}
where we used that the steps of~$S$ have variance~$g=2/\pi$ and where~$o(1)$ is a term that vanishes in the limits $n\to\infty$ followed by~$t\to\infty$. This follows by a comparison of our random walk (that has Gaussian, albeit not time-homogene\-ous, steps) with standard Brownian motion; see, e.g., \cite[Section~4.3]{BL-new}.

Denote $q_\lambda:=(1+\epsilon)\alpha^{-1}\log(1/\lambda)$. From \eqref{8E:1.14} and Lemma~\ref{8lemma-1.4} we then get
\begin{multline}
\quad
\limsup_{\lambda\downarrow0}\,
\limsup_{n\to\infty}
\frac{\sqrt n}{\log(1/\lambda)}\,E\biggl[\exp\Bigl\{-\lambda\int_D\texte^{\alpha^2G^D_n(0,x)}\wh M_n^D(\textd x)\Bigr\}\biggr]
\\
=\lim_{\epsilon\downarrow0}\,\limsup_{\lambda\downarrow0}\,
\limsup_{n\to\infty}
\frac{\sqrt n}{\log(1/\lambda)}\,E\biggl[1_{\{S^\star\le q_\lambda\}}\exp\Bigl\{-\lambda\int_D\texte^{\alpha^2G^D_n(0,x)}\wh M_n^D(\textd x)\Bigr\}\biggr]
\\
\le \lim_{\epsilon\downarrow0}\limsup_{\lambda\downarrow0}\,
\limsup_{n\to\infty}\frac{\sqrt n}{\log(1/\lambda)} P\bigl(S^\star_n\le q_\lambda\bigr)\le
\lim_{\epsilon\downarrow0}\,(1+\epsilon)\,\alpha^{-1}=\frac1{\sqrt{2\pi}},
\quad
\end{multline}
where we used that $\alpha=2/\sqrt g=\sqrt{2\pi}$ in the last step.
\end{proofsect}

\subsection{Proof of Proposition~\ref{prop-8.3}, lower bound}
For the lower bound we will actually proceed along a very much the same argument. Keeping our notation for the annuli~$A_k$, we first note:
 
\begin{lemma}
There is $k_0\ge1$ such that for all $n\ge k_0$, all $k,\ell\in\{0,\dots,n\}$, all $x\in A_{k}$ and all~$y\in A_{\ell}$, we have
\begin{equation}
\label{8E:8.34}
G^D_n(x,y)\ge\sum_{m=1}^n G_n^{A_m}(x,y)+g(k\wedge\ell - k_0)_+
\end{equation}
\end{lemma}

\begin{proofsect}{Proof}
Fix~$x,y\in D$ and let~$k,\ell\ge0$ be such that~$x\in A_k$ and~$y\in A_\ell$. If~$k=\ell > 0$ the inequality holds without taking the positive part on the right hand side, since $x$ and $y$ are deep inside $D$. The monotonicity estimate $G^D_n(x,y)\ge G_n^{A_k}(x,y)$ implied by the corresponding comparison of the (substochastic) transition kernels, then gives~\eqref{8E:8.34} and also handles the case $k=\ell=0$.

For~$k<\ell$ the point~$y$ is ``deep'' inside~$D$ and so, by the fact that~$G_t^{2D}(x,y)-G_t^{D}(x,y)$ is uniformly bounded from above in~$t$ and~$y$, we still have that $G^D_n(x,y)=-g\log(|x-y|\vee\texte^{-n})+O(1)$. Since~$|x-y|\le 2\texte^{-k+1}$, there is~$\hat c>0$ such that
\begin{equation}
G^D_n(x,y)\ge g(k\wedge\ell)-\hat c
\end{equation}
regardless of the choice of~$x$ and~$y$ (with~$k\ne \ell$). Since $G_n^{A_m}(x,y)=0$ for all~$m=1,\dots,n$, taking~$k_0$ be the minimal integer with~$gk_0\ge\hat c$, we get \eqref{8E:8.34} for all~$k,\ell\ge k_0$. The right-hand side of \eqref{8E:8.34} is zero in  the remaining cases (with~$k\ne\ell$) while~$G^D_n(x,y)\ge0$.
\end{proofsect}

We are now ready to give:

\begin{proofsect}{Proof of lower bound in \eqref{8E:1.7}}
As in the upper bound, since $\rad_D$ is bounded from above on~$D$, we can work with $\wh{M}_t^D$ in place of~$M_t^D$. Abusing our earlier notation, let~$S$ now denote the random walk with~$S_0=\dots=S_{k_0}=0$ and~$\{S_{k+1}-S_k\colon k\ge k_0\}$ i.i.d.~$\NN(0,g)$. From Kahane's convexity inequality we then get
\begin{equation}
\begin{aligned}
E\biggl[\exp\Bigl\{-\lambda\int_D\texte^{\alpha^2G^D_n(0,x)}&\wh M_n^D(\textd x)\Bigr\}\biggr]
= E\biggl[\exp\Bigl\{-\lambda\sum_{k=1}^n\int_{A_k}
\texte^{\alpha^2G^D_n(0,x)}\wh M_n^D(\textd x)\Bigr\}\biggr]
\\
&\ge
E\biggl[\exp\Bigl\{-\lambda\sum_{k=1}^n\int_{A_k}
\texte^{\alpha^2G^D_n(0,x)}\texte^{\alpha S_k-\frac12\alpha^2 g k}
\wh M_n^{A_k}(\textd x)\Bigr\}\biggr]
\\
&\ge
E\biggl[\exp\Bigl\{-c\lambda\sum_{k=1}^n\texte^{\alpha S_k} \texte^{2k}
\wh M_n^{A_k}(A_k)\Bigr\}\biggr].
\end{aligned}
\end{equation}
Abusing our notation further, we now define
\begin{equation}
X_{n,k}:=\texte^{2k}\wh M_n^{A_k}(A_k), \quad 1\le k\le n,
\end{equation}
and note that
\begin{equation}
\{X_{n,1},\dots,X_{n,n}\}\text{ are independent with } X_{n,k}\,\,\overset{\text{law}}=\,\,\sqrt{\frac n{n-k+1}}\,X_{n-k+1,1}.
\end{equation}
Although $X_{n,1},\dots,X_{n,n}$ are  not exactly equidistributed; what matters is that we have
\begin{equation}
\forall p<1\colon \quad\sup_{n\ge1}E\bigl(X_{n,1}^p\bigr)<\infty.
\end{equation}
In particular, given~$\epsilon\in(0,1)$ and denoting
\begin{equation}
E_{n,\lambda}:=\bigcap_{k=1}^n\biggl\{X_{n,k}\le \lambda^{-\epsilon/2}k^2\sqrt{\frac n{n-k+1}}\biggr\}
\end{equation}
a standard tail calculation shows 
\begin{equation}
\label{8E:8.41}
\lim_{\lambda\downarrow0}\inf_{n\ge1}\,P(E_{n,\lambda})=1.
\end{equation}
Let now~$q_\lambda:=(1-\epsilon)\alpha^{-1}\log(1/\lambda)$ and let
\begin{equation}
F_{n,\lambda}:=\bigcap_{k=1}^n\bigl\{S_k\le q_\lambda-(\log k)^2\bigr\}.
\end{equation}
We now estimate
\begin{multline}
\qquad
E\biggl[\exp\Bigl\{-c\lambda\sum_{k=1}^n\texte^{\alpha S_k} \texte^{2k}
\wh M_n^{A_k}(A_k)\Bigr\}\biggr]
\\
\ge
E\biggl[1_{E_{n,\lambda}\cap F_{n,\lambda}}\,\exp\Bigl\{-c\lambda\sum_{k=1}^n\texte^{\alpha S_k} \texte^{2k}
\wh M_n^{A_k}(A_k)\Bigr\}\biggr]
\\
\ge
\exp\biggl\{-c\lambda^{\epsilon/2}\sum_{k=1}^n \texte^{-\alpha(\log k)^2}k^2\sqrt{\frac{n}{n-k+1}}\,\biggr\}P\bigl(E_{n,\lambda}\cap F_{n,\lambda}\bigr).
\qquad
\end{multline}
The exponential term tends to unity in the limit~$n\to\infty$ followed by~$\lambda\downarrow0$. In light of the independence of~$E_{n,\lambda}$ and~$F_{n,\lambda}$ and \eqref{8E:8.41}, to get the claim it suffices to show that
\begin{equation}
\lim_{\epsilon\downarrow0}\,\liminf_{\lambda\downarrow0}\,\liminf_{n\to\infty}\,\frac{\sqrt n}{\log(1/\lambda)}
P\bigl(F_{n,\lambda})=\alpha^{-1}=\frac1{\sqrt{2\pi}}.
\end{equation}
This is proved by first extending event (as a lower bound) to full path of Brownian motion of time-length~$n$ and then applying \cite[Proposition~4.7]{BL-new}.
\end{proofsect}

\section*{Appendix: Convergence of harmonic measures}
\setcounter{theorem}{0}
\setcounter{equation}{0}
\renewcommand{\thesection}{A}
\renewcommand{\thetheorem}{A.\arabic{theorem}}
\renewcommand{\theequation}{A.\arabic{equation}}
\addcontentsline{toc}{section}
{{\tocsection {}{\thesection}{\!\!\!\!Appendix: Convergence of harmonic measures\dotfill}}{}}
\noindent
The purpose of this section is to prove that the harmonic measure converges as the underlying domain (discrete or continuous) approaches a limit set $D\in\mathfrak D$. We begin with the harder of the two statements; namely, the one dealing with convergence of the discrete harmonic measure $H^D_N$ to its continuous counterpart~$\Pi^D$:

\begin{lemma}
\label{lemma-HM}
For any~$D\in\mathfrak D$ and any bounded continuous~$f\colon\C\to\R$,
\begin{equation}
\label{E:3.21a}
\sum_{z\in\partial D_N}H^D_N\bigl(\lfloor xN\rfloor,z\bigr)\,f(z/N)\,\underset{N\to\infty}\longrightarrow\,\int_{\partial D}\Pi^D(x,\textd z)f(z)
\end{equation}
uniformly as~$x$ varies over closed subsets of~$D$. 
\end{lemma}


Before we give a proof, let us note that results of this kind (of course) exist in the literature.
For instance, Lawler and Limi\'c~\cite[Proposition~7.3.3]{Lawler-Limic} discuss this for a specific class of simply connected domains (in general dimension). Two dimensional cases have also been treated, e.g., in the recent work of Chelkak and Smirnov~\cite{Chelkak-Smirnov} but these apply only to monotone sequences of domains with the boundary of the limit domain composed of a finite number of Jordan arcs (Proposition~3.3 \cite{Chelkak-Smirnov}) or arbitrary sequences of simply connected domains converging in Carath\'eodory sense (Theorem~3.12 of \cite{Chelkak-Smirnov}). However, none of these settings include our class of domains and so we provide a probabilistic proof here. Our argument is akin to that in \cite{Lawler-Limic} albeit quite special to two dimensions. It is here where the assumption on the number and the diameter of the connected components of~$\partial D$ plays an important role.

\begin{proofsect}{Proof of Lemma~\ref{lemma-HM}}
It is obvious that one can take subsequential limits and that they would all be concentrated on~$\partial D$, so the key issue is to identify the limit with the harmonic measure in the continuum domain. We will do this by a coupling argument. 

Pick~$x\in D$ and let~$B=\{B_t\colon t\ge0\}$ be a standard Brownian motion started from~$x$. Consider the simple random walk $X=\{X_n\colon n\ge0\}$ on $\Z^2$ started from $\lfloor xN\rfloor$. Let $T_R:=\inf\{t\ge0\colon |B_t|\ge R\}$. By Donsker's Theorem, for each~$N\ge1$ there exists a coupling $P_{N,x}$ of the two processes on the same probability space such that, for each~$r>0$ and each $R>0$, uniformly in~$x\in D$,
\begin{equation}
\label{E:A.2}
\lim_{N\to\infty} P_{N,x}\bigl(\,\,\sup_{0\le t\le T_R}|B_t-N^{-1}X_{\lfloor tN^2\rfloor}|>r\bigr)=0.
\end{equation}
The key is to show that the exit point of~$B$ from~$D$ is with high probability close to the exit point (scaled by~$N^{-1}$) of~$X$ from~$D_N$. Specifically, setting
\begin{equation}
\tau:=\inf\{t\ge0\colon B_t\not\in D\}
\end{equation}
and
\begin{equation}
\hat\tau:=\inf\{n\ge1\colon X_n\not\in D_N\},
\end{equation}
it suffices to show that $B_\tau-N^{-1}X_{\hat\tau}$ is small with high probability.

We begin by some general considerations about behavior of Brownian paths. Pick~$r>0$ and consider the figure-eight domain
\begin{equation}
A_r(x):=\bigl\{z\in\C\colon 3r<|x-z-2r|<4r\bigr\}\cup\bigl\{z\in\C\colon 3r<|x-z+2r|<4r\bigr\}.
\end{equation}
Let $\gamma:=\{z\in\C\colon |x-z\pm 2r|=\frac72r\}$ be the center curve in~$A_r(x)$ and let $G_r(x)$ be the event that the Brownian motion hits~$\gamma$ in $A_r(x)\cap\{z\in\C\colon r<|z-x|<2r\}$ and, between this time and the first exit time from~$A_r(x)$ does the following: it runs a complete ``circle'' around one of the annuli until it hits its path again, and then enters the other of the two annuli and runs a complete ``circle'' around it until it hits itself again. 

It is a fact that the probability of~$G_r(x)$ is positive and, by recurrence and scale invariance of~$B$, that so uniformly in~$r>0$ and $x$ with $|x|<r$. It follows from the Strong Markov Property that, for each~$\epsilon>0$ there is~$K\ge1$ such that $\bigcup_{m=1}^K G_{4^m r}(x)$ occurs with probability at least~$1-\epsilon$ for all~$r>0$ and all~$x$ with $|x|<r$. A key observation is that, due to the construction, on~$G_r(x)$ intersected with the event in \eqref{E:A.2} with $R>4r$, both the Brownian motion $B$ and the scaled random walk~$\wt X_t:=N^{-1}X_{\lfloor tN^2\rfloor}$ complete a loop around the disc $\{z\in\C\colon |z-x|<r\}$. 

Let~$\delta>0$ be smaller than the diameters of all connected components of~$\partial D$ and, given~$\epsilon>0$ and~$K\ge1$ as above, pick~$r>0$ so that $4^Kr+r<\delta$. Let $D^r$ be as in \eqref{E:3.21ww} with~$\delta$ replaced by~$r$ and set
\begin{equation}
\tau_r:=\inf\{t\ge0\colon B_t\not\in D^r\}.
\end{equation}
By the Strong Markov property for the Brownian motion we know that
\begin{equation}
\label{E:A.5}
P_{N,x}\Bigl(\bigcup_{m=1}^K G_{4^m r}(B_{\tau_r})\text{ occurs for }B_{\tau_r+\cdot}\Bigr)\ge1-\epsilon
\end{equation}
uniformly in~$N\ge1$ and~$x$ compact subsets of~$D^r$. Since our constructions guarantee that some connected component of~$\partial D$ ``crosses'' $\{z\in\C\colon r<|z-B_{\tau_r}|<4^Kr\}$, on the event in \eqref{E:A.5} we have $|B_{\tau}-B_{\tau_r}|<\delta$. Assuming that the event in \eqref{E:A.2} occurs for some~$R>\delta$, in light of our assumptions \twoeqref{E:1.1}{E:1.1a} on~$D_N$ also $X$ must hit $\partial D_N$ within $N\delta$ of~$NB_\tau$. It follows that
\begin{equation}
\liminf_{N\to\infty}\,\,P_{N,x}\bigl(\,|B_{\tau}-N^{-1}X_{\hat\tau_0}|<2\delta\bigr)>1-\epsilon,
\end{equation}
uniformly in~$x$ over compact subsets of~$D^r$. As~$\epsilon$ and~$\delta$ were arbitrary (just small enough), this yields the claim.
\end{proofsect}

The second approximation statement concerns continuum domains. Since we are dealing with monotone sequences, we could perhaps appeal to the Perron construction of solutions to the Dirichlet problem. (We thank John Garnett for making this point to us.) However, a direct probabilistic argument works just as well: 

\begin{lemma}
\label{lemma-cont-h.m.}
Given $D\in\mathfrak D$, let $D^\delta$ be as in \eqref{E:3.21ww}. Then for any bounded continuous~$f\colon\C\to\R$,
\begin{equation}
\int_{\partial D^\delta}\Pi^{D^\delta}(x,\textd z)f(z)
\,\underset{\delta\downarrow0}\longrightarrow\,
\int_{\partial D}\Pi^D(x,\textd z)f(z)
\end{equation}
uniformly as~$x$ varies over closed subsets of~$D$.
\end{lemma}

\begin{proofsect}{Proof (sketch)}
The proof uses the same argument as that of Lemma~\ref{lemma-HM} except that the coupling is now trivial (both processes are standard Brownian motions). The construction in the above proof guarantees that, after exiting~$D^r$, with probability at least~$1-\epsilon$, the Brownian motion will hit a component of~$\partial D$ within distance~$4^Kr+r$. As $K$ is fixed and~$r\downarrow0$, the claim follows.
\end{proofsect}

\section*{Acknowledgments}
\noindent
This research has been partially supported by European Union's Seventh Framework Programme (FP7/2007-2013]) under the grant agreement number 276923-MC--MOTIPROX, and also ISF grant 1382/17, NSF grant DMS-1407558,  the United States-Israel Binational Science Foundation grant~2018330, the German-Israeli Foundation for Scientific Research and Development grant I-2494-304.6/2017 and GA\v CR project P201/11/1558. Part of this work was carried out when the authors visited the Research Institute for Mathematical Sciences in Kyoto University.  We wish to thank Mo Wong for pointing out that the Laplace transform tail only implies \eqref{E:2.27ua} and not \eqref{E:2.28ua} and to an anonymous referee for many constructive suggestions on the presentation.


\begin{thebibliography}{Ai}

\bibitem{Adler}
R.J.~Adler (1990). \textit{An introduction to continuity, extrema, and related topics for general Gaussian processes}. Institute of Mathematical Statistics Lecture Notes--Monograph Series, vol.~12. Institute of Mathematical Statistics, Hayward, CA, x+160 pp.


\bibitem{BKNSW}
J.~Barral, A.~Kupiainen, M.~Nikula, E.~Saksman, C.~Webb (2015).
Basic properties of critical lognormal multiplicative chaos.
\textit{Ann. Probab.} \textbf{43}, no.~5, 2205--2249.


\bibitem{BPR}
N.~Berestycki, E.~Powell, G.~Ray (2019). A characterisation of the Gaussian free field. arXiv:1802.01195

\bibitem{BGT}
N. H. Bingham, C. M. Goldie and J. L. Teugels (1987). \textit{Regular Variation}. Encyclopedia of Mathematics and Its Applications, vol. 27, Cambridge University Press, Cambridge, xix + 491 pp.


\bibitem{Biskup-PIMS}
M. Biskup (2020). 
\textit{Extrema of the two-dimensional Discrete Gaussian Free Field}.
In: M. Barlow and G. Slade (eds.): Random Graphs, Phase Transitions, and the Gaussian Free Field. SSPROB 2017. Springer Proceedings in Mathematics \&\ Statistics, vol.~304, pp 163--407. Springer, Cham.

\bibitem{BGL}
M.~Biskup, S.~Gufler and O.~Louidor (2019). Near-maxima of two-dimensional Discrete Gaussian Free Field and support sets of critical Liouville Quantum Gravity. In preparation.


\bibitem{Biskup-Louidor}
M. Biskup and O. Louidor (2016).
Extreme local extrema of two-dimensional discrete Gaussian free field.
\textit{Commun. Math. Phys.} \textbf{345}, no. 1, 271--304 

\bibitem{BL-new}
M. Biskup and O. Louidor (2016).
Full extremal process, universality and freezing for the two-dimensional discrete Gaussian Free Field.
\textit{Adv. Math.} \textbf{330} (2018) 589--687

\bibitem{BL-intermediate}
M. Biskup and O. Louidor (2019). 
On intermediate level sets of two-dimensional discrete Gaussian Free Field.
\textit{Ann. Inst. Henri Poincar\'e} \textbf{55}, no. 4, 1948--1987.


\bibitem{BDingZ}
M. Bramson, J. Ding and O. Zeitouni (2016). 
Convergence in law of the maximum of the two-dimensional discrete Gaussian free field. 
\textit{Commun. Pure Appl. Math} \textbf{69}, no.~1, 62--123

\bibitem{BZ}
M. Bramson and O. Zeitouni (2011). 
Tightness of the recentered maximum of the two-dimensional discrete Gaussian free field. 
\textit{Comm. Pure Appl. Math.} \textbf{65}, 1--20.

\bibitem{Chelkak-Smirnov}
D.~Chelkak and S.~Smirnov (2011).
Discrete complex analysis on isoradial graphs.
\textit{Adv. Math.} \textbf{228} 1590--1630.

\bibitem{DPRZ}
A. Dembo, Y. Peres, J. Rosen, and O. Zeitouni (2004). 
Cover times for Brownian motion and random walks in two dimensions. 
\textit{Ann. of Math. (2)} \textbf{160}, 433--464.

\bibitem{Ding}
J. Ding (2013). 
Exponential and double exponential tails for maximum of two-dimensional discrete Gaussian free field. 
\textit{Probab. Theory Rel. Fields} \textbf{157}, no.~1--2, 285--299.

\bibitem{DZ}
J. Ding and O. Zeitouni (2014). 
Extreme values for two-dimensional discrete Gaussian free field. 
\textit{Ann. Probab.}~\textbf{42}, no.~4, 1480--1515

\bibitem{DRSV1}
B. Duplantier, R.~Rhodes, S.~Sheffield and V.~Vargas (2014).
Critical Gaussian multiplicative chaos: Convergence of the derivative martingale. 
\textit{Ann. Probab.}~\textbf{42}, no.~5, 1769--1808

\bibitem{DRSV2}
B. Duplantier, R.~Rhodes, S.~Sheffield and V.~Vargas (2014).
Renormalization of critical Gaussian multiplicative chaos and KPZ formula.  
\textit{Commun. Math. Phys.} \textbf{330}, no.~1,  283--330.


\bibitem{Durrett-Liggett}
R.~Durrett and T.M.~Liggett (1983). Fixed points of the smoothing transformation. 
\textit{Probab. Theory Rel. Fields} \textbf{64}, no.~3, 275--301.

\bibitem{HRV}
Y. Huang, R. Rhodes, and V. Vargas (2018). Liouville Quantum Gravity on the unit disk. 
\textit{Ann. Inst. Henri Poincar\'e} \textbf{54}, no.~3, 1694--1730

\bibitem{Junilla-Saksman}
J. Junnila and E. Saksman (2017). 
Uniqueness of critical Gaussian chaos. \textit{Elect. J.~Probab} \textbf{22}, 1--31.

\bibitem{Kahane}
J.-P. Kahane  (1985). 
Sur le chaos multiplicatif. 
\textit{Ann. Sci. Math. Qu\'ebec} \textbf{9}, no.2, 105--150.

\bibitem{KozmaSchreiber04}
G.~Kozma and E.~Schreiber  (2004). 
An asymptotic expansion for the discrete harmonic potential.
\textit{Electron. J. Probab.} \textbf{9}, Paper no. 1, pages 10--17.


\bibitem{Lawler}
G.F.~Lawler (1991). \textit{Intersections of Random Walks}. 
Birkh\"auser Boston, Inc., Boston, MA.

\bibitem{Lawler-Limic}
G.F.~Lawler and V.~Limic (2010).
\textit{Random walk: a modern introduction}. 
Cambridge Studies in Advanced Mathematics, vol.~123. Cambridge University Press, Cambridge, xii+364.

\bibitem{Madaule}
T. Madaule (2015). 
Maximum of a log-correlated Gaussian field.
\textit{Annales Inst. Henri Poincar\'e. Prob. Stat.} \textbf{51}, no. 4 , 1369--1431.

\bibitem{Powell}
E.~Powell (2018).
Critical Gaussian chaos: convergence and uniqueness in the derivative normalisation.
\textit{Electron. J. Probab.} \textbf{23}, paper no. 31, 26 pp.

\bibitem{Rhodes-Vargas}
R.~Rhodes and V.~Vargas (2014).
Gaussian multiplicative chaos and applications: a review.
\textit{Probab. Surveys} \textbf{11}, 315--392.

\bibitem{Robert-Vargas}
R.~Robert and V.~Vargas (2010). 
Gaussian multiplicative chaos revisited. 
\textit{Ann. Probab.} \textbf{38}, no.~2, 605--631.

\bibitem{Shamov}
A.~Shamov (2016). On Gaussian multiplicative chaos. 
\textit{J. Funct. Anal.} \textbf{270}, no.~9, 3224--3261


\bibitem{Sheffield-review}
S.~Sheffield (2007). 
Gaussian free fields for mathematicians.
\textit{Probab. Theory Rel. Fields} \textbf{139}, no.~3--4, 521--541.

\bibitem{Wong}
M.D.~Wong (2019). Universal tail profile of Gaussian multiplicative chaos. arXiv:1902.04054

\bibitem{Wong2}
M.D.~Wong (2019). Tail universality of critical Gaussian multiplicative chaos. arXiv:1912.02755

\end{thebibliography}
\end{document}